\documentclass{amsart}

\makeatletter
\def\@serieslogo{}
\def\@issueinfo{}
\def\@copyrightyear{}
\def\@copyrightline{}
\def\@PII{}
\makeatother

\usepackage{amssymb}
\usepackage{amsmath}
\usepackage{mathrsfs}
\usepackage{enumitem}
\usepackage[authoryear,round]{natbib}

\newtheorem{theorem}{Theorem}[section]
\newtheorem{lemma}[theorem]{Lemma}
\newtheorem{proposition}[theorem]{Proposition}
\newtheorem{corollary}[theorem]{Corollary}

\theoremstyle{definition}
\newtheorem{definition}[theorem]{Definition}

\theoremstyle{remark}
\newtheorem{remark}[theorem]{Remark}
\newtheorem{conjecture}[theorem]{Conjecture}

\numberwithin{equation}{section}

\providecommand{\dist}{}\renewcommand{\dist}{\operatorname{dist}}
\usepackage{caption}

\newcommand{\R}{\mathbb{R}}

\newcommand{\Z}{\mathbb{Z}}

\newcommand{\TT}{\mathbb{T}}
\newcommand{\sgn}{\operatorname{sgn}}
\newcommand{\diag}{\operatorname{diag}}
\newcommand{\tr}{\operatorname{tr}}
\newcommand{\Leb}{\operatorname{Leb}}
\newcommand{\eps}{\varepsilon}
\newcommand{\PP}{\Phi'}
\newcommand{\Om}{\Omega}

\usepackage{tikz}
\usetikzlibrary{calc, arrows.meta, patterns, decorations.pathmorphing, decorations.markings, positioning, backgrounds, shapes.geometric, shapes.misc}

\usepackage[colorlinks=true, citecolor=blue, linkcolor=blue, urlcolor=blue]{hyperref}

\DeclareMathOperator{\Div}{Div}

\begin{document}

\title{Mixed Global Dynamics of the Forced Vibro-Impact Oscillator with Coulomb Friction and its Symplectic Structure, KAM Tori, and Persistence}

\author{Abdoulaye Thiam}
\address{Division of Mathematics and Natural Sciences, Allen University, Columbia, South Carolina 29204, USA}
\email{athiam@allenuniversity.edu}


\date{}


\begin{abstract}
The forced vibro-impact oscillator with Amonton-Coulomb friction and elastic walls was shown by~\cite{GKR2019} to exhibit a coexistence of Hamiltonian stability islands and dissipative attractors in a single phase space. We provide a complete mathematical analysis of this phenomenon. We prove global well-posedness of the associated Filippov flow and construct a global lift to a piecewise smooth Hamiltonian system on a covering manifold. On the maximal forward-invariant non-sticking set, we show that the time-$T$ stroboscopic map is exact symplectic, within the formalism of symplectic dynamics. We derive a closed-form existence equation for symmetric $T$-periodic orbits and establish a parameter-dependent saddle-center bifurcation at $f_{\rm sc}(F,\omega,R)$, correcting a universality claim in prior work. Using Moser’s twist theorem, we prove the existence of invariant Cantor families (KAM tori) near elliptic non-sticking periodic orbits, while a Melnikov analysis yields hyperbolic dynamics conjugate to a Bernoulli shift near the associated saddle. We further show that any positive restitution defect or viscous damping destroys the conservative structure: elliptic periodic orbits persist but become asymptotically stable, replacing Hamiltonian islands by a single attracting basin. The approach extends to multi-particle systems with elastic collisions, where a symplectic structure and higher-dimensional KAM tori are obtained. A computer-assisted proof verifies the existence and ellipticity of a non-sticking periodic orbit at a specific parameter point.
\end{abstract}

\maketitle

\thispagestyle{empty}

\makeatletter
\renewcommand{\ps@headings}{%
  \def\@oddfoot{\hfill\thepage\hfill}%
  \let\@evenfoot\@oddfoot%
  \def\@evenhead{\hfill\normalfont\small\textit{A.~Thiam}\hfill}%
  \def\@oddhead{\hfill\normalfont\small\textit{Mixed Dynamics in the Forced Vibro-impact Oscillator}\hfill}%
  \let\@mkboth\markboth%
}
\pagestyle{headings}
\makeatother

\setlength{\parskip}{0.1em}

\section{Introduction}\label{sec:intro}


The classical taxonomy of finite-dimensional dynamical systems distinguishes between conservative and dissipative regimes by global features of the phase portrait. In a conservative system, phase volume is preserved by the flow, the Poincaré recurrence theorem applies, and asymptotic states are organized around a hierarchy of invariant tori intermixed with chaotic seas, as described by the Kolmogorov-Arnold-Moser theory~\cite{Arnold1989},~\cite{KatokHasselblatt1995}, and~\cite{SiegelMoser1971}. In a dissipative system, by contrast, phase volume contracts in forward time, transient behavior gives way to attractors of typically lower dimension, and the long-term dynamics is captured by the Sinai-Ruelle-Bowen theory of physical measures on hyperbolic attractors~\cite{KatokHasselblatt1995}.

This dichotomy is stable under most physical perturbations: a small dissipation, however introduced, destroys conservative phase-volume preservation and replaces invariant tori by attracting periodic orbits or strange attractors~\cite{Bourgain1997} and~\cite{GuckenheimerHolmes1983}. It is therefore notable when a single dynamical system, with no parameter modification and on a single connected phase space, exhibits both behaviors simultaneously. Such \emph{mixed dynamics} has been observed in several reversible systems where time-reversal symmetry compels coexisting attractor-repellor pairs~\cite{PolitiOppoBadii1986},~\cite{QuispelRoberts1989},~\cite{LambRoberts1998}, and~\cite{GonchenkoShilnikovTuraev2008}, and more recently in non-reversible piecewise-smooth systems where the dissipation mechanism is selective: it acts on certain trajectory classes and not on others~\cite{LeonelMcClintock2006} and~\cite{GKR2019}.

Vibro-impact oscillators with Coulomb friction are a physically realistic class of such systems. The friction term $-f\,\sgn(\dot x)$ delivers impulse $-f\,|\dot x|\,dt$ when the body is moving and zero impulse when the body is at rest; the elastic wall imposes velocity reversal $\dot x \mapsto -\dot x$ at impact, conserving kinetic energy. Neither mechanism produces a smooth dissipation comparable to viscous damping. The result, as observed by~\cite{GKR2019}, is that the system organizes itself into a forward-invariant subset on which the time-$T$ Poincaré map is volume-preserving and a complementary subset on which volume strictly contracts. The two subsets share the same phase space; trajectories starting in the conservative part undergo KAM-type quasi-periodic motion with chaotic boundaries, whereas trajectories with even one velocity-zero crossing are funneled into a dissipative attractor.

Establishing this picture rigorously, identifying the bifurcations that organize it, and quantifying its persistence under physically realistic perturbations are the goals of the present paper. The motivation is twofold. First, mathematically, the system is a tractable test bed in which the conservative and dissipative theories apply simultaneously in a finite-dimensional, low-regularity setting; methods from symplectic geometry, KAM theory, the Smale-Birkhoff homoclinic theorem, and Filippov-type differential inclusions all apply, sometimes after a coordinate transformation that lifts away the discontinuities. Second, in engineering applications, mixed dynamics has substantive consequences: vibro-impact nonlinear energy sinks~\cite{Vakakis2008NES} and~\cite{GendelmanManevitch2008}, percussive drilling~\cite{Babitsky1998}, gear-rattle in automotive transmissions~\cite{IbrahimVibroImpact}, and MEMS resonators~\cite{FidlinBook} are all designed under modeling assumptions that the dynamics is everywhere dissipative; the existence of a positive-measure conservative subset implies that initial conditions in that subset never reach the designed attractor and never deliver the intended energy dissipation. A quantitative threshold for the disappearance of the conservative subset under residual viscous damping is then an engineering datum.

\medskip
This section of the introduction presents the model, summarizes existing results and gaps in the literature, and identifies what is rigorously established in this work.

\subsection{The model and its background}\label{ssec:intro-model}

We study the system
\begin{equation}\label{eq:model}
\ddot x + f\,\sgn(\dot x) = F\cos(\omega t), \qquad l < x < r,
\end{equation}
on the open segment $(l,r) \subset \R$, completed by the elastic reflection rule
\begin{equation}\label{eq:reflection}
\dot x(t^+) = -\dot x(t^-) \quad \text{whenever } x(t) \in \{l, r\}.
\end{equation}
The constant $F > 0$ is the forcing amplitude, $\omega > 0$ the angular frequency, $f > 0$ the kinetic friction coefficient, and $R := r - l > 0$ the gap between rigid walls. The function $\sgn$ is the multivalued sign:
\[
\sgn(v) = \begin{cases} \{+1\} & v > 0,\\ [-1, 1] & v = 0,\\ \{-1\} & v < 0.\end{cases}
\]
We write $T = 2\pi/\omega$ for the period of the forcing. Throughout we assume the non-trivial forcing condition
\begin{equation}\label{eq:Fgreaterf}
F > f.
\end{equation}
Without~\eqref{eq:Fgreaterf} every trajectory comes to permanent rest~\cite{Shaw1986}. The system~\eqref{eq:model}-\eqref{eq:reflection} is sketched in Figure~\ref{fig:model-schematic}.

\begin{figure}[htbp]
\centering
\begin{tikzpicture}[
    >={Latex[length=2.5mm,width=2mm]},
    every node/.style={font=\small},
    wall/.style={line width=0.6pt, draw=black!80},
    blockedge/.style={line width=0.6pt, draw=black!80, fill=blue!12},
    forcearrow/.style={-{Latex[length=3mm,width=2.3mm]}, line width=1.2pt, draw=blue!50!black},
    velarrow/.style={-{Latex[length=2mm,width=1.5mm]}, line width=0.9pt, draw=red!55!black},
    fricarrow/.style={-{Latex[length=2mm,width=1.5mm]}, line width=0.8pt, draw=red!60!black},
    dim/.style={{Latex[length=2mm]}-{Latex[length=2mm]}, line width=0.4pt, draw=black!55},
]
    \fill[pattern=north east lines, pattern color=black!55] (-3.55,-0.7) rectangle (-3.0,1.5);
    \draw[wall] (-3.0,-0.7) -- (-3.0,1.5);
    \fill[pattern=north west lines, pattern color=black!55] (3.0,-0.7) rectangle (3.55,1.5);
    \draw[wall] (3.0,-0.7) -- (3.0,1.5);
    \fill[pattern=north east lines, pattern color=black!55] (-3.0,-0.95) rectangle (3.0,-0.7);
    \draw[wall] (-3.0,-0.7) -- (3.0,-0.7);
    \draw[blockedge, rounded corners=0.6pt] (-0.7,-0.7) rectangle (0.7,0.45);
    \node at (0,-0.13) {$m$};
    \draw[velarrow] (-0.4,0.18) -- (0.4,0.18);
    \node[red!55!black, above, font=\footnotesize] at (0,0.18) {$\dot x$};
    \draw[forcearrow] (-0.5,0.95) -- (1.6,0.95);
    \node[blue!50!black, above, font=\small] at (0.55,0.95) {$F\cos(\omega t)$};
    \draw[line width=0.5pt, draw=blue!50!black] (-0.5,0.45) -- (-0.5,0.95);
    \draw[fricarrow] (-0.05,-0.85) -- (-0.7,-0.85);
    \node[red!60!black, below, font=\footnotesize] at (-0.4,-0.93) {$-f\,\mathrm{sgn}(\dot x)$};
    \draw[->, line width=0.5pt, draw=black!65] (-3.7,-1.5) -- (3.8,-1.5) node[right] {$x$};
    \draw[line width=0.4pt, draw=black!65] (-3.0,-1.45) -- (-3.0,-1.55);
    \draw[line width=0.4pt, draw=black!65] (3.0,-1.45) -- (3.0,-1.55);
    \node[below] at (-3.0,-1.55) {$l$};
    \node[below] at (3.0,-1.55) {$r$};
    \draw[dashed, line width=0.4pt, draw=black!50] (0,-0.7) -- (0,-1.5);
    \node[below, fill=white, inner sep=1pt, font=\footnotesize] at (0,-1.6) {$x(t)$};
    \draw[dim] (-3.0,1.85) -- (3.0,1.85);
    \node[fill=white, inner sep=1pt] at (0,1.85) {$R = r - l$};
    \draw[line width=0.3pt, draw=black!55] (-3.0,1.5) -- (-3.0,1.95);
    \draw[line width=0.3pt, draw=black!55] (3.0,1.5) -- (3.0,1.95);
\end{tikzpicture}
\captionsetup{margin={-0.5cm,0cm}}
\caption{The forced vibro-impact oscillator with Coulomb friction. }
\label{fig:model-schematic}
\end{figure}
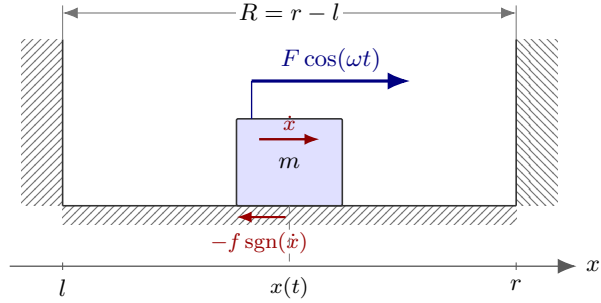
A unit point mass slides on a rough surface between two rigid walls at $x = l$ and $x = r$, subject to periodic external forcing $F\cos(\omega t)$ and dry friction $f\,\mathrm{sgn}(\dot x)$ opposing motion. Wall contact at $x \in \{l,r\}$ instantaneously reverses the velocity, $\dot x(t^+) = -\dot x(t^-)$. The wall gap is $R = r - l$.

The system~\eqref{eq:model}-\eqref{eq:reflection} is a basic test bed of nonsmooth mechanics. It describes percussive drilling, gear rattle, mooring lines, micro-electromechanical resonators, nonlinear energy sinks for vibration mitigation, and other intermittent-contact devices~\cite{Babitsky1998},~\cite{FidlinBook},~\cite{KobrinskiiVibroImpact},~\cite{Vakakis2008NES},~\cite{GendelmanManevitch2008}, and~\cite{IbrahimVibroImpact}. The general theory of differential equations with discontinuous right-hand sides was set down by~\cite{Filippov1988} and is treated in modern monographs by~\cite{diBernardoBuddChampneysKowalczyk2008},~\cite{Brogliato2016}, and~\cite{AcaryBrogliato2008}; the monograph of~\cite{StewartSIAM2000} surveys rigid-body dynamics with friction and impact. Set-valued analysis underlying the Filippov theory is developed in~\cite{AubinCellina1984}.

\subsection{State of the art and the contribution of the present paper}\label{ssec:intro-mixed}

\textit{Earlier work and the observations of~\cite{GKR2019}.} Earlier studies of impact oscillators with Coulomb friction (\cite{ConeZadoks1995},~\cite{VirginBegley1999},~\cite{BlazejczykOkolewska1996},~\cite{Awrejcewicz2003},~\cite{LuoGegg2006},~\cite{LuoGeggJSV2006},~\cite{ZhangFu2017}, and~\cite{FanYang2018}) all included viscous damping. The eventual dynamics in those works consisted of $\omega$-limit sets in the standard sense: equilibria, periodic orbits, quasi-periodic orbits, or strange attractors. The grazing bifurcation theory of \cite{Nordmark1991} and~\cite{FredrikssonNordmark2000} systematized the local structure at the boundary of contact regions in the same setting.

The recent paper of~\cite{GKR2019} made the essential observation that this dichotomy fails when the viscous coefficient is set to zero: stability islands of Hamiltonian type and dissipative attractors coexist in the same phase space at the same parameter values. The mechanism identified by~\cite{GKR2019} is the following. On a non-sticking trajectory whose velocity vanishes nowhere in a period, the Jacobian of the stroboscopic map factors through a sequence of saltation matrices each having unit determinant, so that area is preserved. On a trajectory with sticking or with a turning point, the Jacobian acquires a saltation factor of determinant strictly less than one, and the area is contracted. The state space therefore decomposes naturally into a forward-invariant non-sticking subset on which the dynamics is conservative and a complementary subset on which it is dissipative.

The contribution of~\cite{GKR2019} is summarized as follows. A formal lift to a piecewise linear Hamiltonian system through the triangular wave projection $x = R\, W(q) + l$. A direct computation of the determinant of the Jacobian along a non-sticking trajectory, yielding $|\det \Phi'| = 1$. Closed-form expressions for two symmetric $T$-periodic solutions with one turning point per half-period (\emph{ibid.}, Eq.~(7)). Numerical evidence for the saddle-center bifurcation at $f/F = 2/\pi$ between these two solutions. Numerical phase portraits showing stability islands surrounded by chaotic seas. A two-particle generalization sketched at the end of the paper.

\emph{Not addressed in~\cite{GKR2019}.} Well-posedness of the Filippov flow, in particular the absence of finite-time accumulation of impacts and of velocity-zero events. A symplectic structure on the non-sticking invariant set, beyond the determinant statement. Existence of an elliptic non-sticking $T$-periodic orbit (without turning points) responsible for the Hamiltonian islands; we will see below that the closed-form orbits of~\cite[Eq.~(7)]{GKR2019} are not themselves the carriers of the islands, since they have turning points and hence lie in the dissipative subset. A KAM theorem with measure estimate. A Melnikov analysis of stable and unstable manifolds of the saddle orbit produced by the bifurcation. Analytical persistence of the Hamiltonian islands under perturbations of the model: small restitution defect $1 - e$, small viscous damping $\mu_{\rm v}$. The non-degenerate normal form of the saddle-center bifurcation. Symplecticity of the multi-particle stroboscopic map.

\textit{Contribution of the present paper.} The principal contribution is to fill the gaps listed above with rigorous proofs, supplemented by rigorous numerics where the analytic statements depend on parameters in non-trivial ways. The novelty is summarized in seven main theorems, listed in order of appearance and labeled by section.

\smallskip\noindent
\textbf{Main Theorem~\ref{thm:wellposed}}. The Filippov inclusion associated with~\eqref{eq:model}-\eqref{eq:reflection} is well-posed for all initial data and all forward time. The flow is continuous on the complement of a Lebesgue null set of singular events. The proof rules out finite-time accumulation of impacts and of velocity-zero events.

\smallskip\noindent
\textbf{Main Theorem~\ref{thm:lift} together with Theorem~\ref{thm:symplectic}}. There is a global lift of~\eqref{eq:model}-\eqref{eq:reflection} to a smooth Hamiltonian system on a covering manifold, exhibited explicitly. The induced stroboscopic map is exact symplectic on the maximal non-sticking forward-invariant subset of phase space, with respect to the standard form $dv \wedge dx$. This promotes the determinant identity of~\cite{GKR2019} to a structural result accessible to symplectic methods.

\smallskip\noindent
\textbf{Main Theorem~\ref{thm:symmetric} together with Theorem~\ref{thm:saddlecenter}}. The set of symmetric $T$-periodic non-sticking solutions of~\eqref{eq:model}-\eqref{eq:reflection} with the impact pattern of~\cite{GKR2019} (one wall hit and one turning point per half-period) is in bijection with the set of zeros $\theta \in (0, \pi)$ of the transcendental equation $F\pi\sin\theta = 2F + R\omega^2 - f\pi^2/2 + f\theta(\pi - \theta)$. The equation has zero, one, or two solutions according to whether $f$ is less than, equal to, or in the open interval between the parameter-dependent saddle-center value $f_{\rm sc}(F, \omega, R) := 4(2F + R\omega^2 - F\pi)/\pi^2$ and the universal impulse bound $f_{\rm imp} := 2F/\pi$. The two solutions coalesce at $f = f_{\rm sc}$ in a non-degenerate saddle-center bifurcation with the local normal form $(\theta_\pm - \pi/2)^2 = (\pi^2/2)/(F\pi - 2 f_{\rm sc})\,(f - f_{\rm sc}) + O((f - f_{\rm sc})^{3/2})$. This corrects~\cite[Eq.~(7)]{GKR2019}, where the saddle-center value was claimed to be the universal value $2F/\pi$.

\smallskip\noindent
\textbf{Main Theorem~\ref{thm:kam}}. At every elliptic non-sticking $T$-periodic orbit (lying in $\Om_{\rm NS}$, with no turning points) satisfying the standard order-four non-resonance and the Birkhoff twist non-degeneracy, the stroboscopic map admits, in any neighborhood of size $\delta$, a Cantor family of invariant smooth closed curves whose complement has Lebesgue measure $O(\delta^{5/2})$. This is the rigorous version of the Hamiltonian islands observed numerically in~\cite{GKR2019}. Existence of such an orbit at the parameter point $(F, \omega, R, f) = (1, 1, 2, 0.4)$ is verified rigorously by interval Newton in Section~\ref{sec:numerics}.

\smallskip\noindent
\textbf{Main Theorem~\ref{thm:melnikov}}. Near the saddle-center bifurcation of Theorem~\ref{thm:saddlecenter}, the saddle orbit of~\cite{GKR2019} possesses, in an averaged limit~\cite{Treschev1997} and~\cite{Neishtadt1984}, a homoclinic loop. The associated Melnikov function has the structure $M(t_0) = A\cos(\omega t_0) + B$ with $A, B$ explicitly identified. Whenever $|A| > |B|$, the Smale-Birkhoff theorem produces an invariant hyperbolic Cantor set on which the dynamics is conjugate to a Bernoulli shift on two symbols.

\smallskip\noindent
\textbf{Main Theorem~\ref{thm:persistence}}. For positive restitution defect $\eps = 1 - e$ or positive viscous damping $\mu_{\rm v}$, every elliptic non-sticking $T$-periodic orbit persists, but the multipliers of its linearization move strictly inside the unit disk. The orbit is therefore asymptotically stable, and the unperturbed Hamiltonian island is replaced by a single open basin of attraction. This answers the persistence question left open by~\cite{GKR2019} and identifies a quantitative threshold, $\rho(\eps, \mu_{\rm v}) = 1 - 2 n_*\eps - \mu_{\rm v} T + O((\eps + \mu_{\rm v})^2)$, for the disappearance of mixed dynamics.

\smallskip\noindent
\textbf{Main Theorem~\ref{thm:multiparticle}}. For the multi-particle generalization of~\eqref{eq:model}-\eqref{eq:reflection} with elastic binary collisions, the stroboscopic map restricted to the maximal forward-invariant subset of sign-preserving non-sticking initial data is exact symplectic. As a consequence, the higher-dimensional KAM theorem produces invariant Lagrangian tori around any elliptic non-resonant non-sticking periodic point with non-degenerate twist; mixed dynamics in the multi-particle setting falls within the same theory.

The number of main theorems is seven.

\textit{Methodology.} The technical observation enabling the analysis is that, although the system~\eqref{eq:model}-\eqref{eq:reflection} is piecewise smooth across both the velocity-zero set $\Sigma_v := \{v=0\}$ and the wall surfaces $\Sigma_l := \{x = l\}$ and $\Sigma_r := \{x = r\}$, a single coordinate transformation removes both discontinuities at once on the non-sticking part of the dynamics. The transformation is an instance of the universal-cover construction familiar in classical mechanics on quotient spaces and in the theory of one-dimensional billiards~\cite{Tabachnikov1995} and~\cite{KozlovTreshchev1991}.

The configuration space $[l, r]$ may be regarded as the quotient of the real line by the action of the dihedral group generated by the two reflections in $\{l\}$ and $\{r\}$. The universal cover is therefore $\R$, with covering map $\pi: \R \to [l, r]$ given by $\pi(q) = R\, W(q) + l$, where $W$ is the unit-amplitude triangular wave of period two. Pulled back through this covering, a body undergoing elastic reflection at the walls is replaced by a body moving along the line $\R$ with no physical walls, but with a piecewise-affine continuous potential that switches sign at every integer (representing the alternating direction of motion in the original configuration space). Coulomb friction, which has a discontinuity along $\Sigma_v$, becomes constant on each branch of the cover because the velocity does not change sign along the lifted trajectory; this is the substance of the lift theorem (Theorem~\ref{thm:lift}).

The lifted equations of motion are Hamilton equations in $(q, p) \in \R^2$ for the time-dependent Hamiltonian
\[
H(q, p, t) = \frac{p^2}{2} - \frac{F}{R}\cos(\omega t)\,W(q) + \frac{f}{R}\,q,
\]
which is a Lipschitz, piecewise-affine-in-$q$ function on a smooth phase space. By Liouville's theorem the lifted flow preserves the standard symplectic form $dp \wedge dq$. On a non-sticking trajectory in the original system, the lifted trajectory has constant sign of velocity, hence stays on a single branch of the triangular wave between integer crossings; the lifted equations on each branch are constant-coefficient linear and the propagator is a $C^\infty$ map of the initial data. Time-$T$ stroboscopic propagation is therefore $C^\infty$ on neighborhoods of non-sticking $T$-periodic orbits with transverse impacts, and exact symplectic on the entire non-sticking invariant set $\Om_{\rm NS}$.

This is the technical foundation on which the rest of the paper rests. Sections~\ref{sec:periodic}-\ref{sec:melnikov} apply, on the smooth non-sticking part, the symplectic methods (closed-form $T$-periodic orbits, KAM theorem, Melnikov method) developed in the smooth Hamiltonian setting. Section~\ref{sec:decomposition} returns to the global decomposition of state space into conservative and dissipative parts. Section~\ref{sec:persistence} treats the perturbations of physical interest. Section~\ref{sec:multiparticle} extends the symplectic methods to multi-particle systems. Section~\ref{sec:numerics} provides the rigorous numerical verifications.

This reduction has two methodological consequences. First, the smooth theory of symplectic maps near elliptic fixed points (Birkhoff normal form, Moser's twist theorem, the Smale-Birkhoff homoclinic theorem) applies directly to $\Phi|_{\Om_{\rm NS}}$ at periodic orbits with transverse impacts, without recourse to a piecewise-smooth KAM theory~\cite{TreschevZubelevich2010} and~\cite{Dolgopyat2009}. Second, the saddle-center bifurcation of~\cite{GKR2019} can be analyzed by the standard Hamiltonian normal form for fold bifurcations of one-degree-of-freedom Hamiltonian flows~\cite{KuznetsovBifurcation} and~\cite{Treschev1997}.

The persistence theorem of Section~\ref{sec:persistence} applies a different methodology. The lift construction does not directly accommodate viscous damping or restitution defect (these are not symplectic perturbations of the lifted Hamiltonian). Instead, the persistence is established by direct computation of the saltation matrices of the perturbed system at each event, accumulated over one period. The off-diagonal saltation entries are unaffected by the perturbations to leading order; the diagonal contributions yield the asymptotic $\rho(\eps, \mu_{\rm v}) = 1 - 2 n_*\eps - \mu_{\rm v} T + O((\eps + \mu_{\rm v})^2)$, which combines additively over the events of the period.

The numerical apparatus of Section~\ref{sec:numerics} is event-driven with closed-form propagation on each free flight, avoiding the spurious viscosity that black-box ODE solvers would introduce; this is essential here, since by Theorem~\ref{thm:persistence} any non-zero numerical viscosity would destroy the islands the simulation is meant to display. Rigorous interval enclosures of the linearization and the rotation number are obtained by the interval-Newton method using the Krawczyk operator~\cite{Krawczyk1969} and~\cite{NeumaierBook}, the multi-precision interval arithmetic of \texttt{mpmath}~\cite{Johansson2017}, and the validated-numerics methodology of Tucker~\cite{TuckerBook2011} and Rump~\cite{Rump1999}; this approach is suited to vibro-impact systems by the work of~\cite{GaliasTucker2008} and~\cite{Galias2002}, the rigorous Poincaré-map techniques of~\cite{WilczakZgliczynski2003} and~\cite{WilczakZgliczynski2009}, and the CAPD library of~\cite{KapelaMrozekWilczak2017}.

\subsection{Related literature, notation, and organization of the paper}\label{ssec:intro-related}

Coexistence of conservative and dissipative behaviors in reversible dynamical systems was investigated by~\cite{PolitiOppoBadii1986},~\cite{QuispelRoberts1989}, and~\cite{LambRoberts1998}, and developed further by~\cite{GonchenkoShilnikovTuraev2008}. The mechanism in those works is reversibility: the time-reversal involution forces attractor-repellor pairs to coexist with conservative invariant sets in between. The system~\eqref{eq:model}-\eqref{eq:reflection} is reversible only at $f = 0$, and the mixed dynamics has a different origin: turning points and sticking events dissipate area, while the complement is conservative for purely geometric reasons rooted in the saltation algebra.

A related but distinct mechanism appears in the Fermi-Ulam bouncer with friction studied by~\cite{LeonelMcClintock2006}: the dissipation there is added by hand, the conservative subset is chosen by the modeler, and the resulting dynamics is engineered to display mixed behavior. Our setting has the conservative subset arising intrinsically from the model, with no design choice.

The lift construction has antecedents in the theory of one-dimensional billiards with potential, in particular in the works of~\cite{GalperinKrugerTroubetzkoy1995},~\cite{Tabachnikov1995}, and~\cite{KozlovTreshchev1991}. In the closely related context of multi-particle hard-rod systems on a segment, the lift to the simplex cover is classical and underlies Sinai's billiard analysis and its successors~\cite{Sinai1972} and~\cite{Simanyi2009}. In the context of vibro-impact systems specifically, the lift through a triangular wave was used by~\cite{GendelmanManevitch2008} and made explicit in~\cite{GKR2019}; our use of it is to apply standard smooth KAM theory to the stroboscopic map of a piecewise smooth system, sidestepping the more difficult task of developing a piecewise-smooth KAM theory directly.

The systematic development of KAM theory for piecewise smooth maps, where invariant tori may cross the discontinuity surfaces of the dynamics, was undertaken by~\cite{TreschevZubelevich2010},~\cite{Dolgopyat2009}, and~\cite{DelMagnoGaivaoPereira2018}; further results on piecewise-smooth bifurcation theory are in~\cite{diBernardoBuddChampneysKowalczyk2008}. Restricting to non-sticking trajectories means we do not engage with these difficulties; the cost is that we say nothing about invariant tori that touch the velocity-zero surface or that are disrupted by sticking events.

The Melnikov method we apply originates with~\cite{Melnikov1963} and was put in modern form for two-degree-of-freedom Hamiltonian systems by~\cite{HolmesMarsden1982}; the Smale-Birkhoff homoclinic theorem is treated in~\cite{GuckenheimerHolmes1983} and~\cite{KatokHasselblatt1995}. The closed-form computation in Theorem~\ref{thm:melnikov} relies on the explicit averaged Hamiltonian near the saddle-center bifurcation, in the spirit of~\cite{LichtenbergLieberman1992},~\cite{Treschev1997}, and~\cite{Neishtadt1984}. The general structure of saddle-center bifurcations in Hamiltonian flows is treated in~\cite{KuznetsovBifurcation} and reviewed in the context of celestial mechanics by~\cite{SiegelMoser1971}.

Theorem~\ref{thm:persistence} is in the spirit of~\cite{Bourgain1997} on the breakdown of invariant tori under dissipative perturbations of integrable systems, and of~\cite{BensoussanLionsBook} and~\cite{FlemingSoner2006} on impulse control of oscillators with state constraints. The quantitative threshold $\rho(\eps, \mu_{\rm v})$ identifies the precise rate at which the islands disappear under residual dissipation, an object that the abstract persistence theorems of~\cite{Bourgain1997} do not provide.

The validated computation in Section~\ref{sec:numerics} draws on the monograph of~\cite{TuckerBook2011}, the interval analysis of~\cite{NeumaierBook}, the INTLAB toolbox of~\cite{Rump1999}, the Arb library of~\cite{Johansson2017}, and the rigorous-Poincaré-map methodology of~\cite{WilczakZgliczynski2003} and~\cite{WilczakZgliczynski2009}. For piecewise-smooth and impact systems specifically, the rigorous methodology of~\cite{GaliasTucker2008} and~\cite{Galias2002} and the CAPD library of~\cite{KapelaMrozekWilczak2017} are most directly relevant. Continuation methodology for the bifurcation diagrams of Section~\ref{sec:numerics} follows the COCO toolbox of~\cite{DankowiczSchilder2013} and AUTO of~\cite{DoedelOldeman2012}, with the piecewise-smooth extension TC-HAT of~\cite{ThotaDankowicz2008} for vibro-impact systems specifically.

Numerical exploration of the system~\eqref{eq:model}-\eqref{eq:reflection} was undertaken by the present author beginning in 2019. The poster~\cite{Thiam2019poster} records the closed-form integration of~\eqref{eq:model}-\eqref{eq:reflection} on free flights, an event-driven simulation algorithm with a stiction-detection rule, bifurcation diagrams in the forcing amplitude $F$ at fixed $\omega$, $R$, $f$, and a numerical observation of multiple coexisting periodic orbits, $\Sigma$-symmetry-breaking, and chattering regimes. The closed-form integration scheme of~\cite{Thiam2019poster} is the same as the one we use in Subsection~\ref{ssec:num-event}; the bifurcation phenomenology there motivated, retrospectively, the rigorous saddle-center theorem proved here as Theorem~\ref{thm:saddlecenter} and the $\Sigma$-equivariance discussion in~\ref{ssec:hist-pitchfork}. The closely related problem of constant external forcing and inelastic walls (restitution $e \in [0, 1)$ with distinct kinetic and static friction coefficients) is qualitatively different: the primary bifurcation is a supercritical pitchfork rather than a saddle-center, and the multiplicative dissipation $e < 1$ at every impact prevents the formation of a non-trivial conservative invariant subset. That regime falls outside the scope of the present work but shares the saltation-matrix and Floquet machinery developed in Section~\ref{sec:lift}. The bouncing-ball phenomenology at large $F$, including the regime gallery, the continuation of $T$-periodic orbits, $\Sigma$-symmetry-breaking pairs, chattering toward sticking, and grazing combined with stiction, is regenerated rigorously and analyzed mathematically in~\ref{app:historical}.

\textit{Notational conventions.} We write $\mathcal{X} := \{(x, v) \in \R^2 : l \le x \le r\}$ for the state space at fixed phase. The stroboscopic map is denoted $\Phi: \mathcal{X} \to \mathcal{X}$. We write $\Leb$ for Lebesgue measure on $\R^2$ and $\dist$ for the Euclidean distance. The notation $A \asymp B$ means $A/B$ is bounded above and below by positive constants. The maximal forward-invariant non-sticking subset of $\mathcal{X}$ is denoted $\Om_{\rm NS}$ and its complement in $\mathcal{X}$ is denoted $\Om_{\rm dissip}$; precise definitions are given in Section~\ref{sec:setup}.

\textit{Organization of the paper.} The exposition proceeds in three movements: a foundational setup of the Filippov dynamics and the symplectic lift (Sections~\ref{sec:setup}-\ref{sec:lift}); the analytical core that establishes periodic orbits, KAM islands, and homoclinic chaos (Sections~\ref{sec:periodic}-\ref{sec:melnikov}); and the synthesis, perturbation analysis, multi-particle generalization, and rigorous numerics (Sections~\ref{sec:decomposition}-\ref{sec:numerics}), followed by a conclusion and an appendix.

We begin in Section~\ref{sec:setup} by establishing the rigorous mathematical setting. The Filippov inclusion associated with~\eqref{eq:model}-\eqref{eq:reflection} is defined precisely, including the convexified selection at the velocity-zero set and the wall reflection rule. The local structure on each switching surface is then classified into three exclusive cases (transverse turning, onset of sticking, tangential touch), each receiving a precise lemma with proof. These local results assemble into the substantive content of the section, Theorem~\ref{thm:wellposed}, which establishes global well-posedness for every initial datum in forward time and provides a uniform lower bound on inter-event time on bounded velocity intervals. With well-posedness in hand, we close the section with the precise definition of the stroboscopic map $\Phi$ and of the partition $\mathcal{X} = \Om_{\rm NS} \cup \Om_{\rm dissip}$ that organizes the rest of the paper.

The structural foundation is laid in Section~\ref{sec:lift}, which carries out the universal-cover construction. The triangular wave projection is defined and the lifted Hamiltonian system on the cover is exhibited explicitly; Proposition~\ref{thm:lift} then establishes the conjugacy between the lifted flow and the original Filippov flow on non-sticking trajectories. From this conjugacy follows the section's principal result, Theorem~\ref{thm:symplectic}: the stroboscopic map $\Phi: \Om_{\rm NS} \to \Om_{\rm NS}$ is exact symplectic with respect to the standard form $dv \wedge dx$, with explicit primitive 1-form and generating function. As an immediate corollary, $|\det \Phi'| < 1$ Lebesgue-a.e.~on $\Om_{\rm dissip}$, the dynamical signature of the dissipative subset.

Building on this symplectic structure, Section~\ref{sec:periodic} treats the closed-form periodic orbits of~\cite{GKR2019} together with the saddle-center bifurcation that organizes them. Theorem~\ref{thm:symmetric} derives the transcendental equation characterizing the symmetric $T$-periodic orbits with one wall hit and one turning per half-period. Analyzing this equation, Proposition~\ref{prop:two-solutions} establishes that it has zero, one, or two solutions, with $f_{\rm sc}(F, \omega, R) = 4(2F + R\omega^2 - F\pi)/\pi^2$ the parameter-dependent saddle-center value and $f_{\rm imp} = 2F/\pi$ the universal impulse bound. Theorem~\ref{thm:saddlecenter} then shows that the two solutions undergo a non-degenerate saddle-center bifurcation, with explicit local normal form $(\theta_\pm - \pi/2)^2 \asymp (f - f_{\rm sc})$ and the universal $\sqrt{\mu}$ scaling of the eigenvalues on each branch. Along the way, the section corrects the claim of~\cite{GKR2019} that the saddle-center bifurcation is at the universal value $2F/\pi$, which is in fact the impulse bound rather than the saddle-center value.

The analytical core continues in Section~\ref{sec:kam} with the proof of the KAM theorem for elliptic non-sticking $T$-periodic orbits. Combining smoothness of $\Phi$ on a neighborhood of any non-sticking periodic orbit with transverse impacts, the Birkhoff normal form, and the quantitative twist theorem of Moser-P\"oschel-Salamon, Theorem~\ref{thm:kam} produces a Cantor family of invariant smooth closed curves around any elliptic non-resonant non-degenerate non-sticking $T$-periodic orbit, with measure complement $O(\delta^{5/2})$ in any disk of radius $\delta$. Complementing the elliptic side, Section~\ref{sec:melnikov} treats the homoclinic chaos surrounding the saddle orbit produced by the bifurcation of Theorem~\ref{thm:saddlecenter}. The slow Hamiltonian near the bifurcation is first constructed via the Treshchev-Neishtadt averaging; the Melnikov function $M(t_0) = A\cos(\omega t_0) + B$ is then computed in closed form (Theorem~\ref{thm:melnikov}), with the amplitude $A$ identified by an explicit residue calculation. Whenever $|A| > |B|$, the Smale-Birkhoff theorem produces an invariant hyperbolic Cantor set on which the dynamics is conjugate to the Bernoulli shift on two symbols.

With both branches analyzed, Section~\ref{sec:decomposition} assembles the previous sections into the global decomposition of state space: Proposition~\ref{thm:OmNS-structure} summarizes the structure of $\Phi|_{\Om_{\rm NS}}$, Proposition~\ref{prop:contraction} establishes strict volume contraction on $\Om_{\rm dissip}$, and Proposition~\ref{prop:Omdissip-flag} clarifies the (non-)forward-invariance of $\Om_{\rm dissip}$. The conservative-dissipative partition is now rigorous content rather than numerical observation.

The remaining theoretical sections address two natural extensions. Section~\ref{sec:persistence} establishes persistence of elliptic non-sticking $T$-periodic orbits under viscous damping and restitution defect: Theorem~\ref{thm:persistence} computes the perturbed eigenvalues to leading order and establishes asymptotic stability for any positive perturbation, with explicit contraction rate $\rho(\eps, \mu_{\rm v}) = 1 - 2 n_*\eps - \mu_{\rm v} T + O((\eps + \mu_{\rm v})^2)$. This answers the persistence question left open by~\cite{GKR2019} and quantifies the threshold at which the Hamiltonian islands disappear under realistic perturbations. Section~\ref{sec:multiparticle} then extends the analysis to $N$-particle systems with elastic binary collisions: Theorem~\ref{thm:multiparticle} establishes symplecticity of the multi-particle stroboscopic map on the maximal sign-preserving non-sticking invariant set, and a corollary applies the higher-dimensional KAM theorem to produce invariant Lagrangian tori around elliptic non-sticking periodic points.

The analytical program is complemented in Section~\ref{sec:numerics} by quantitative simulations and a rigorous computer-assisted verification at the parameter point $(F, \omega, R, f) = (1, 1, 2, 0.4)$. The principal new result there is Theorem~\ref{thm:CAP-elliptic}, a rigorous interval-Newton verification of the elliptic Jacobian and the rotation number, certifying the hypotheses of Theorem~\ref{thm:kam} apart from the twist non-degeneracy; the conditional Proposition~\ref{thm:CAP-KAM} would close the remaining gap subject to a rigorous interval enclosure of the Birkhoff coefficient $\tau_1$.

The paper closes in Section~\ref{sec:open} with a synoptic conclusion that restates each main theorem alongside its principal formula, identifies where the mathematical novelty is concentrated, lists open problems with brief discussions of the relevant tools, and offers an outlook on the most accessible extensions. Finally, Appendix~\ref{app:historical} documents the bouncing-ball phenomenology at large forcing amplitude $F$, including a regime gallery, a continuation in $F$ of coexisting $T$-periodic orbits, a $\Sigma$-symmetry-breaking pair, and the chattering and grazing regimes near sticking onset.

\section{Mathematical setup and well-posedness}\label{sec:setup}

This section places~\eqref{eq:model}-\eqref{eq:reflection} on a rigorous footing as a differential inclusion in the sense of Filippov, classifies the local structure on each switching surface, and proves global existence and uniqueness of forward solutions for every initial datum. The well-posedness statement is needed throughout the paper; the local classifications are used directly in Section~\ref{sec:lift} (to construct the lift), in Section~\ref{sec:periodic} (to identify periodic orbits with prescribed impact pattern), and in Section~\ref{sec:decomposition} (to characterize the dissipative subset).

\subsection{The Filippov inclusion}\label{ssec:setup-filippov}

We recall the Filippov setting~\cite{Filippov1988} and~\cite{AubinCellina1984} adapted to~\eqref{eq:model}. The state space is
\[
\mathcal{X} := \{(x, v) \in \R^2 : l \le x \le r\}.
\]
The Filippov set-valued vector field $\mathcal{F}: \mathcal{X} \times \R \to 2^{\R^2}$ associated with~\eqref{eq:model} is
\begin{equation}\label{eq:Ffield}
\mathcal{F}(x, v, t) = \begin{cases}
\bigl\{(v,\ F\cos\omega t - f \sgn v)\bigr\} & v \ne 0, \\
\{v\} \times \bigl[F\cos\omega t - f,\ F\cos\omega t + f\bigr] & v = 0.
\end{cases}
\end{equation}

\begin{definition}[Filippov solution]\label{def:Filippov-solution}
A \emph{Filippov solution} of~\eqref{eq:model}-\eqref{eq:reflection} on an interval $I \subset \R$ is an absolutely continuous function $x: I \to [l, r]$ such that:
\begin{enumerate}[label=\textup{(F\arabic*)}]
\item\label{F1} $(\dot x(t), \ddot x(t)) \in \mathcal{F}(x(t), \dot x(t), t)$ for almost every $t \in I$ with $x(t) \in (l, r)$.
\item\label{F2} If $t_* \in I$ satisfies $x(t_*) \in \{l, r\}$ and $\dot x(t_*^-) \ne 0$, then $\dot x(t_*^+) = -\dot x(t_*^-)$.
\item\label{F3} If $x \equiv l$ on a sub-interval $[a, b] \subset I$ then $\dot x \equiv 0$ on $[a, b]$, and the same with $r$ replacing $l$.
\end{enumerate}
\end{definition}

Item~\ref{F1} is the standard Filippov inclusion, including the convexified selection at $v = 0$. Item~\ref{F2} is the elastic reflection rule. Item~\ref{F3} excludes the unphysical case of a trajectory adhering to a wall at non-zero velocity.

\begin{remark}\label{rem:sticking}
On the velocity zero surface $\{v = 0\}$ with $|F\cos\omega t| < f$, the unique Filippov selection that keeps the trajectory on the surface is $\ddot x = 0$. This is the physical sticking regime: static friction balances the applied force exactly. When $|F\cos\omega t| > f$, the inclusion no longer admits the sliding selection and the trajectory leaves the surface.
\end{remark}

\subsection{Local structure on switching surfaces}\label{ssec:setup-local}

The state space has three switching structures: the velocity zero set $\Sigma_v := \{v = 0\}$, the left wall $\Sigma_l := \{x = l\}$, and the right wall $\Sigma_r := \{x = r\}$. We classify the local behavior on each. The two lemmas below are standard results in the theory of piecewise smooth flows; the first treats velocity zero events and is used at every place in the paper where turning points and sticking are discussed, the second treats wall reflections and is used in Sections~\ref{sec:lift} and~\ref{sec:periodic}.

\begin{lemma}[Classification of velocity-zero events]\label{lem:vzero}
Let $x(\cdot)$ be a Filippov solution and let $t_*$ satisfy $\dot x(t_*) = 0$, $x(t_*) \in (l, r)$. One of three exclusive cases holds:
\begin{enumerate}[label=\textup{(\alph*)}]
\item\label{vzero-a} \textsc{Transverse turning point:} If $|F\cos\omega t_*| > f$, the function $\dot x$ is differentiable at $t_*$ with $\dot v(t_*) = F\cos\omega t_* - f\,\sgn(\dot v(t_*))$, hence $|\dot v(t_*)| = |F\cos\omega t_*| - f > 0$. The trajectory crosses $\Sigma_v$ transversely.
\item\label{vzero-b} \textsc{Onset of sticking:} If $|F\cos\omega t_*| < f$, there is $\delta > 0$ with $\dot x \equiv 0$ on $[t_*, t_* + \delta]$. The body remains at rest until the first time $t^\sharp > t_*$ at which $|F\cos\omega t^\sharp| = f$, at which it leaves with sign $\sgn(F\cos\omega t^\sharp)$.
\item\label{vzero-c} \textsc{Tangential touch:} The case $|F\cos\omega t_*| = f$ occurs only when $\omega t_* \in \{\pm \arccos(f/F) + 2\pi\Z\}$, a discrete set of times. For Lebesgue-almost every initial datum, no zero-velocity event occurs at such a time.
\end{enumerate}
\end{lemma}

\begin{proof}
\textit{Item~\ref{vzero-a}.} Suppose first $F\cos\omega t_* > f$. By continuity of $t \mapsto F\cos\omega t$, there exists $\eta > 0$ such that $F\cos\omega t > f + \tfrac{1}{2}(F\cos\omega t_* - f) =: c_+ > f$ for all $t \in [t_* - \eta, t_* + \eta]$.

\emph{Uniqueness of the sign just after $t_*$.} The Filippov inclusion at $\dot x = 0$ admits the convex set of selections $\ddot x \in [F\cos\omega t - f,\, F\cos\omega t + f]$. We argue that the only selection compatible with the trajectory leaving $\Sigma_v = \{\dot x = 0\}$ is $\sgn(\ddot x(t_*^+)) = +1$. Indeed, suppose for contradiction the trajectory leaves $\Sigma_v$ with $\dot x(t_*^+) < 0$. Then for $t$ slightly larger than $t_*$ we have $\sgn(\dot x) = -1$, so the equation reads $\ddot x = F\cos\omega t + f \ge c_+ + f > 0$, meaning $\dot x$ is strictly increasing; but $\dot x(t_*^+) = 0$ and $\dot x(t) > \dot x(t_*^+) - O(t - t_*)\cdot \sup|\ddot x|$, so $\dot x$ cannot remain negative. The sliding selection $\dot x \equiv 0$ is also excluded: it requires $0 \in [F\cos\omega t - f,\, F\cos\omega t + f]$, equivalently $|F\cos\omega t| \le f$, which fails on $[t_* - \eta, t_* + \eta]$. Hence the only consistent continuation is $\dot x(t_*^+) > 0$, with $\ddot x(t_*^+) = F\cos\omega t_* - f > 0$.

\emph{One-sided derivatives of $\dot x$ at $t_*$.} On a left-neighborhood $(t_* - \eta, t_*)$ the velocity $\dot x$ has constant sign, say $-1$ (the case $+1$ is symmetric). On this interval $\ddot x = F\cos\omega t + f$ (with $\sgn(\dot x) = -1$ in the friction term), so the left-derivative of $\dot x$ at $t_*$ is $\ddot x(t_*^-) = F\cos\omega t_* + f$. By the conclusion of the previous paragraph, on a right-neighborhood $(t_*, t_* + \eta)$ we have $\sgn(\dot x) = +1$, so $\ddot x = F\cos\omega t - f$ and the right-derivative is $\ddot x(t_*^+) = F\cos\omega t_* - f$. The two one-sided derivatives differ by $2f$; $\dot x$ has a corner at $t_*$ but is locally Lipschitz with both one-sided derivatives existing. We summarize the transverse-turning event by the right-derivative
\[
\dot v_+(t_*) := \ddot x(t_*^+) = F\cos\omega t_* - f, \qquad |\dot v_+(t_*)| = F\cos\omega t_* - f > 0.
\]
The case $F\cos\omega t_* < -f$ is treated analogously by exchanging the signs.

\textit{Item~\ref{vzero-b}.} Suppose $|F\cos\omega t_*| < f$. By continuity, there exists $\eta > 0$ such that $|F\cos\omega t| < f - \tfrac{1}{2}(f - |F\cos\omega t_*|) =: c_- < f$ for all $t \in [t_* - \eta, t_* + \eta]$.

\emph{The sliding selection is admissible.} The selection $\ddot x = 0$ keeps the trajectory on $\Sigma_v$. It is a member of the Filippov set if $0 \in [F\cos\omega t - f,\, F\cos\omega t + f]$, equivalently $|F\cos\omega t| \le f$, which holds strictly on $[t_* - \eta, t_* + \eta]$. Hence $\ddot x = 0$ on this interval is a valid Filippov solution.

\emph{Uniqueness of the sliding selection.} The Filippov sliding theorem~\cite[Ch.~II, \S4, Theorem~1]{Filippov1988} applies: the discontinuity surface $\Sigma_v = \{\dot x = 0\}$ is smooth, the limits of the field on each side of $\Sigma_v$ point inward (i.e.\ $F\cos\omega t - f < 0$ from the side $\dot x > 0$ and $F\cos\omega t + f > 0$ from the side $\dot x < 0$ when $|F\cos\omega t| < f$), and the unique sliding selection on $\Sigma_v$ is the one that preserves the surface, namely $\ddot x = 0$.

\emph{Exit time $t^\sharp$.} Define $t^\sharp := \inf\{t > t_* : |F\cos\omega t| \ge f\}$. Since $\{t : |F\cos\omega t| \ge f\}$ is a finite union of closed intervals per period and the cosine is non-trivially varying, $t^\sharp$ is finite (in fact $t^\sharp - t_* \le T$). At $t^\sharp$ the strict inequality $|F\cos\omega t^\sharp| > f$ becomes possible to the right, and Item~\ref{vzero-a} applies with $\sgn(\dot x(t^{\sharp +})) = \sgn(F\cos\omega t^\sharp)$.

\textit{Item~\ref{vzero-c}.} The set $\{t \in [0, T] : |F\cos\omega t| = f\}$ is finite, namely $\{t_*^\pm\} = \{\pm\arccos(f/F)/\omega\} \cup \{(2\pi \pm \arccos(f/F))/\omega\}$, four points in one period (or two if $f = F$). Across periods, the set of all such times is countable. The condition that a Filippov solution have a velocity-zero event at one of these specific times defines a codimension-one analytic constraint on the initial datum, so the set of bad initial data is a countable union of analytic codimension-one submanifolds of $\mathcal{X}$, hence Lebesgue-null.
\end{proof}

\begin{lemma}[Wall reflection]\label{lem:wall}
Let $x(\cdot)$ be a Filippov solution with $x(t_*) = r$ and $\dot x(t_*^-) > 0$. Then $\dot x(t_*^+) = -\dot x(t_*^-) < 0$, and there exists $\eta > 0$ with $x(t) < r$ for $t \in (t_*, t_* + \eta)$. The analogous statement holds at $x = l$ with $\dot x(t_*^-) < 0$.
\end{lemma}

\begin{proof}
The reflection rule~\ref{F2} of Definition~\ref{def:Filippov-solution} gives the velocity flip directly. Taylor expansion at $t_*$ gives $x(t) = r - |\dot x(t_*^-)|(t - t_*) + O((t - t_*)^2)$ for $t > t_*$, hence $x(t) < r$ for $t$ in a small right-neighborhood.
\end{proof}

\subsection{Global well-posedness}\label{ssec:setup-wellposed}

We now combine the local results into a global well-posedness statement, the first main theorem of the paper. The substantive content of the proof is that no finite-time accumulation of switching events can occur; without this, even the definition of the stroboscopic map $\Phi$ would be problematic.

\begin{theorem}[Well-posedness]\label{thm:wellposed}
Assume $0 < f < F$. For every initial datum $(x_0, v_0) \in \mathcal{X}$, there exists a unique Filippov solution of~\eqref{eq:model}-\eqref{eq:reflection} satisfying $(x(0), \dot x(0^+)) = (x_0, v_0)$, defined for all $t \ge 0$. The flow depends continuously on the initial datum on every compact time interval that contains no tangential-touch event in the sense of Lemma~\ref{lem:vzero}\ref{vzero-c}.
\end{theorem}

\begin{proof}
We construct the solution by induction on the switching events and bound the event count on every bounded time interval.

\textit{Step 1: Local existence on smooth pieces.} Fix $(x_0, v_0)$ with $x_0 \in (l, r)$ and $v_0 \ne 0$. The Filippov inclusion~\eqref{eq:Ffield} reduces to the smooth, time-dependent linear ODE
\[
\ddot x(t) = F\cos\omega t - f\,\sgn(v_0)
\]
in any open neighborhood of $(x_0, v_0)$ on which $\sgn(\dot x)$ remains constant. The right-hand side is uniformly Lipschitz in $(x, \dot x)$ and continuous in $t$, so Picard-Lindel\"of yields a unique $C^2$ solution on the maximal interval $[0, t_*)$ on which $x \in (l, r)$ and $\sgn(\dot x) = \sgn(v_0)$. The boundary $t_*$ is exactly the first time at which one of the conditions
\[
(\mathrm{i})\ \dot x(t_*) = 0,\qquad (\mathrm{ii})\ x(t_*) = r\ \text{with}\ \dot x(t_*) > 0,\qquad (\mathrm{iii})\ x(t_*) = l\ \text{with}\ \dot x(t_*) < 0
\]
fails to hold; cases (ii) and (iii) are wall hits, and (i) is a velocity-zero event classified by Lemma~\ref{lem:vzero}.

\textit{Step 2: Continuation through events.} At each switching event the continuation is uniquely prescribed: after a transverse turning point at $t_*$, Lemma~\ref{lem:vzero}\ref{vzero-a} prescribes $\sgn(\dot x(t_*^+)) = \sgn(F\cos\omega t_*)$ with $|\dot x(t_*^+) - 0| = |F\cos\omega t_*| - f > 0$; after a sticking onset, Lemma~\ref{lem:vzero}\ref{vzero-b} prescribes $\dot x \equiv 0$ until the exit time $t^\sharp = \inf\{t > t_* : |F\cos\omega t| = f\}$, at which $\sgn(\dot x(t^{\sharp+})) = \sgn(F\cos\omega t^\sharp)$; after a wall hit, Lemma~\ref{lem:wall} prescribes $\dot x(t_*^+) = -\dot x(t_*^-)$ and $x(t_*^+) = x(t_*^-)$. The data of the continuation are continuous functions of the state and time at the event, so the inductive construction proceeds.

\textit{Step 3: Energy bound and finite kinetic energy on bounded intervals.} Define $E(t) := \tfrac{1}{2}\dot x(t)^2$. On each smooth piece, $E'(t) = \dot x \cdot \ddot x = \dot x \cdot F\cos\omega t - f|\dot x|$, hence
\[
|E'(t)| \le (F + f)\,|\dot x(t)| = (F + f)\,\sqrt{2 E(t)}.
\]
Wall reflections satisfy $E(t_*^+) = E(t_*^-)$; sticking intervals have $E \equiv 0$. For any $T_0 > 0$, integrating the differential inequality $|E'| \le (F+f)\sqrt{2E}$ on $[0, T_0]$ and applying Gr\"onwall to $\sqrt{E}$ gives the explicit upper bound
\begin{equation}\label{eq:wp-energy}
\sqrt{E(t)} \le \sqrt{E(0)} + \tfrac{1}{\sqrt{2}}(F + f)\,t \qquad \text{for all } t \in [0, T_0],
\end{equation}
so $|\dot x(t)| \le V_*(T_0) := |v_0| + (F + f)\,T_0$ on $[0, T_0]$.

\textit{Step 4: No accumulation of wall hits.} Suppose, for contradiction, that wall hit times satisfy $t_n \nearrow t_\infty < \infty$. Between two consecutive wall hits the trajectory traverses a distance at least $R := r - l$, so
\[
R \le \int_{t_n}^{t_{n+1}} |\dot x(s)|\,ds \le V_*(t_\infty)\,(t_{n+1} - t_n),
\]
giving $t_{n+1} - t_n \ge R/V_*(t_\infty) > 0$, contradicting accumulation.

\textit{Step 5: No accumulation of velocity-zero events.} We show that the inter-event gaps between consecutive velocity-zero events on $[0, T_0]$ admit a uniform positive lower bound, ruling out accumulation. The argument distinguishes two cases according to the type of event.

\smallskip
\noindent
(5a) \emph{Sticking events have positive duration.} By Lemma~\ref{lem:vzero}\ref{vzero-b}, a sticking onset at time $\tau_*$ continues until the first time $t^\sharp > \tau_*$ at which $|F\cos\omega t^\sharp| = f$. Since the sets $\{t : |F\cos\omega t| < f\}$ and $\{t : |F\cos\omega t| > f\}$ alternate as open intervals of positive length within each forcing period, with the first having length
\[
\tau_{\rm wait} := \frac{2}{\omega}\bigl(\pi - \arccos(f/F)\bigr) - \frac{2}{\omega}\arccos(f/F) = \frac{2}{\omega}\bigl(\pi - 2\arccos(f/F)\bigr) > 0
\]
under $f < F$, we have $t^\sharp - \tau_* > 0$ uniformly. Note that the body \emph{cannot} stick during a phase $\{|F\cos\omega t| > f\}$ (such a phase admits no sliding selection, by Lemma~\ref{lem:vzero}\ref{vzero-a}), so a sticking interval started in a $\{|F\cos\omega t| < f\}$ phase consumes at least the remaining length of that phase, bounded below by a positive constant depending only on $F$, $f$, $\omega$.

\smallskip
\noindent
(5b) \emph{Transverse turning events are separated by a positive interval.} Let $\tau_n < \tau_{n+1}$ be consecutive transverse turning events on $[0, T_0]$ with no sticking event on $[\tau_n, \tau_{n+1}]$, and assume in addition that no wall hit occurs in this interval (the wall-to-turning hybrid case is covered by Step~6). By Lemma~\ref{lem:vzero}\ref{vzero-a}, the orbit leaves $\tau_n$ with velocity sign $s := \sgn(F\cos\omega \tau_n)$, so on a right-neighborhood of $\tau_n$ the equation reads $\ddot x(t) = F\cos\omega t - s f$.

Integrating from $\tau_n$ with $\dot x(\tau_n^+) = 0$ gives
\[
\dot x(t) \;=\; s\bigl(F\cos\omega \tau_n - f\bigr)(t - \tau_n) \;+\; r(t),
\qquad |r(t)| \;\le\; \tfrac{1}{2}(F\omega + f)(t - \tau_n)^2,
\]
where the remainder $r(t)$ comes from Taylor expansion of $F\cos\omega t$ about $\tau_n$ and the bound $|\ddot x| \le F + f$ on the smooth piece. Setting $a := |F\cos\omega \tau_n| - f > 0$ (Lemma~\ref{lem:vzero}\ref{vzero-a}),
\[
|\dot x(t)| \;\ge\; a (t - \tau_n) - \tfrac{1}{2}(F\omega + f)(t - \tau_n)^2.
\]
This lower bound is strictly positive on the interval $(\tau_n,\, \tau_n + 2a/(F\omega + f))$.

Since the next velocity-zero event $\tau_{n+1}$ requires $\dot x(\tau_{n+1}) = 0$, we obtain
\[
\tau_{n+1} - \tau_n \;\ge\; \frac{2 a}{F\omega + f}.
\]
On any compact time interval, the set of turning times $\tau$ with $|F\cos\omega\tau| > f$ has positive distance from the tangential-touch set $\{|F\cos\omega \tau| = f\}$ (which is finite per period), once Lebesgue-null initial data are excluded by Remark~\ref{rem:Lebtyp}. Hence there exists $a_* > 0$ depending only on $F, f, \omega, T_0$ such that $a \ge a_*$ at every transverse turning event in $[0, T_0]$, and
\begin{equation}\label{eq:wp-tauzero}
\tau_{n+1} - \tau_n \;\ge\; \tau_{\rm gap}(F, f, \omega, T_0) \;:=\; \frac{2 a_*}{F\omega + f} \;>\; 0,
\end{equation}
a uniform positive lower bound.

\smallskip
\noindent
(5c) \emph{Joint accumulation is excluded.} The total number of velocity-zero events on $[0, T_0]$ is bounded above by $T_0/\min(\tau_{\rm wait}, \tau_{\rm gap})$, finite.

\textit{Step 6: Wall hits and turning events do not accumulate jointly.} Steps~4 and~5 each rule out the accumulation of one type of event in isolation. To rule out joint accumulation, observe that any subinterval consisting purely of free flight between an event of one type and an event of the other has length bounded below by the minimum of the bounds in steps~4 and~5: the position changes by at least the wall separation $R$ in the wall-to-wall case, and by at least the lower bound of substep~(5b) in the turning-to-turning case. The hybrid wall-to-turning bound is no smaller than the smaller of these. Hence the total event count on $[0, T_0]$ remains finite.

\textit{Step 7: Continuous dependence.} On every compact interval $[0, T_0]$ avoiding tangential-touch events, the flow is the composition of finitely many smooth pieces (each $C^2$ in the initial data by classical ODE theory) and finitely many event maps (each Lipschitz in the state at the event by Lemma~\ref{lem:vzero}\ref{vzero-a},\ref{vzero-b} and Lemma~\ref{lem:wall}). The composition of finitely many continuous maps is continuous; see~\cite[Ch.~6]{diBernardoBuddChampneysKowalczyk2008} for the systematic treatment in piecewise smooth flows. The conclusion follows.
\end{proof}

\begin{remark}\label{rem:Lebtyp}
Initial data leading to a tangential-touch event at any positive time form a countable union of codimension-one analytic submanifolds, hence a Lebesgue null set. Throughout the rest of the paper we work modulo Lebesgue null sets when stating differentiability of the stroboscopic map.
\end{remark}

\subsection{The stroboscopic map and the non-sticking invariant set}\label{ssec:setup-strob}

Theorem~\ref{thm:wellposed} permits the definition of the stroboscopic map. This subsection records the basic invariants used throughout the paper.

\begin{definition}[Stroboscopic map]\label{def:Phi}
The stroboscopic map $\Phi: \mathcal{X} \to \mathcal{X}$ is defined by
\[
\Phi(x_0, v_0) := (x(T),\ \dot x(T^+)),
\]
where $x(\cdot)$ is the unique Filippov solution of Theorem~\ref{thm:wellposed} with $(x(0), \dot x(0^+)) = (x_0, v_0)$.
\end{definition}

By Theorem~\ref{thm:wellposed} and Remark~\ref{rem:Lebtyp}, $\Phi$ is well defined on $\mathcal{X}$ and Lebesgue-almost everywhere differentiable.

\begin{definition}[Sticking, turning, non-sticking]\label{def:trajectory-types}
A solution $x(\cdot)$ on $[0, T]$ is called
\begin{itemize}
\item[1.] \emph{sticking} if $\dot x \equiv 0$ on a sub-interval of positive length;
\item[2.] \emph{turning} if $\dot x$ has at least one zero on $[0, T]$ but the solution is not sticking;
\item[3.] \emph{non-sticking} if $\dot x(t) \ne 0$ for all $t \in [0, T]$.
\end{itemize}
We denote
\[
\Om_{\rm NS}^{[0, T]} := \{(x_0, v_0) \in \mathcal{X} : \text{the solution starting at } (x_0, v_0) \text{ is non-sticking on } [0, T]\}
\]
and define the maximal forward-invariant non-sticking subset of $\mathcal{X}$ by
\begin{equation}\label{eq:Om-NS-def}
\Om_{\rm NS} := \bigcap_{n = 0}^{\infty} \Phi^{-n}\bigl(\Om_{\rm NS}^{[0, T]}\bigr).
\end{equation}
\end{definition}

\begin{proposition}[Forward invariance]\label{prop:Om-NS-invariant}
The set $\Om_{\rm NS}$ is forward $\Phi$-invariant.
\end{proposition}


\begin{proof}
Let $z \in \Omega_{\rm NS}$. We show that $\Phi(z) \in \Omega_{\rm NS}$, that is, $\Phi^n(\Phi(z)) \in \Omega_{\rm NS}^{[0, T]}$ for every $n \ge 0$.

Fix $n \ge 0$. Since $z \in \Omega_{\rm NS}$, the defining identity~\eqref{eq:Om-NS-def} gives $\Phi^{n+1}(z) \in \Omega_{\rm NS}^{[0, T]}$. By Theorem~\ref{thm:wellposed}, the iterates of $\Phi$ are everywhere defined on $\mathcal{X}$, and the $T$-periodicity of the forcing in~\eqref{eq:model} yields the semigroup identity $\Phi^{n+1} = \Phi^n \circ \Phi$. Therefore
\[
\Phi^n(\Phi(z)) = \Phi^{n+1}(z) \in \Omega_{\rm NS}^{[0, T]}.
\]
Since $n \ge 0$ was arbitrary, $\Phi(z) \in \bigcap_{n = 0}^{\infty} \Phi^{-n}(\Omega_{\rm NS}^{[0, T]}) = \Omega_{\rm NS}$.
\end{proof}

We define
\[
\Om_{\rm dissip} := \mathcal{X} \setminus \Om_{\rm NS}.
\]
Section~\ref{sec:lift} establishes that $\Phi$ is exact symplectic on $\Om_{\rm NS}$. Section~\ref{sec:decomposition} returns to the decomposition and proves contraction estimates on $\Om_{\rm dissip}$.

\section{The Hamiltonian lift and symplectic structure}\label{sec:lift}

This section provides the rigorous version of the lift formally introduced by~\cite[Sec.~III]{GKR2019}. The principal result, Theorem~\ref{thm:symplectic}, states that the stroboscopic map is exact symplectic on the non-sticking invariant set. This promotes the observation $|\det \Phi'| = 1$ of~\cite{GKR2019} on non-sticking trajectories from a determinant identity to a structural symplectic statement, on which the application of symplectic methods in Sections~\ref{sec:periodic}, \ref{sec:kam}, and~\ref{sec:melnikov} rests.

\subsection{The triangular wave projection and the lifted Hamiltonian}\label{ssec:lift-triangular}

We introduce the projection that takes the lifted variable $q$ to the position $x$ inside the wall configuration. The projection is many-to-one, encoded by the triangular wave below; the lift is therefore a covering, and the symplectic structure on the cover descends to the non-sticking set in $\mathcal{X}$.

Define the triangular wave $W: \R \to \R$ of period $2$ by
\begin{equation}\label{eq:Wq}
W(q) := \begin{cases} q & 0 \le q < 1, \\ 2 - q & 1 \le q < 2, \end{cases} \qquad W(q + 2) := W(q),
\end{equation}
so that $W$ is continuous, piecewise affine, Lipschitz with Lipschitz constant $1$, and $W' \in \{+1, -1\}$ Lebesgue-almost everywhere. Define the projection $\pi: \R \to [l, r]$ by
\begin{equation}\label{eq:pi}
\pi(q) := R \cdot W(q) + l, \qquad R := r - l.
\end{equation}

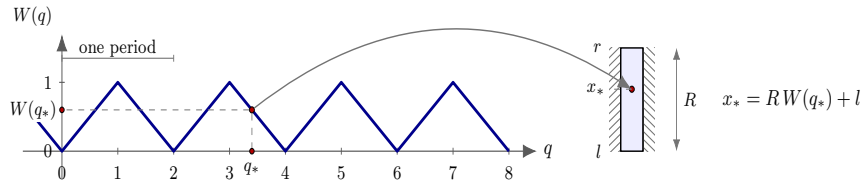
\begin{figure}[htbp]
\centering
\resizebox{0.90\linewidth}{2.5cm}{%
\begin{tikzpicture}[
    >={Latex[length=2.5mm,width=2mm]},
    every node/.style={font=\small},
    wave/.style={line width=1.3pt, draw=blue!55!black, line join=round},
    axis/.style={->, line width=0.5pt, draw=black!65},
    proj/.style={dashed, line width=0.4pt, draw=black!50},
    samplept/.style={fill=red!70!black, draw=black, line width=0.3pt, circle, minimum size=2.5pt, inner sep=0pt},
    boxedge/.style={draw=black, line width=0.6pt, fill=blue!7},
    wallhatchL/.style={pattern=north east lines, pattern color=black!55},
    wallhatchR/.style={pattern=north west lines, pattern color=black!55},
]
    \draw[axis] (-0.7,0) -- (8.5,0) node[right, font=\small] {$q$};
    \draw[axis] (0,-0.4) -- (0,1.7);
    \node[font=\footnotesize, anchor=south east] at (-0.05,1.7) {$W(q)$};
    \draw[line width=0.4pt, draw=black!65] (-0.05,1) -- (0.05,1);
    \node[left, font=\footnotesize] at (-0.05,1) {$1$};
    \node[left, font=\footnotesize] at (-0.05,0) {$0$};
    \foreach \q in {1,2,3,4,5,6,7,8} {
        \draw[line width=0.4pt, draw=black!65] (\q,-0.05) -- (\q,0.05);
    }
    \foreach \q/\lab in {0/0,1/1,2/2,3/3,4/4,5/5,6/6,7/7,8/8} {
        \node[below, font=\footnotesize] at (\q,-0.1) {$\lab$};
    }
    \draw[wave]
        (-0.5,0.5) -- (0,0) -- (1,1) -- (2,0) -- (3,1) -- (4,0) --
        (5,1) -- (6,0) -- (7,1) -- (8,0);
    \draw[line width=0.3pt, draw=black!55] (0,1.35) -- (2,1.35);
    \draw[line width=0.3pt, draw=black!55] (0,1.30) -- (0,1.40);
    \draw[line width=0.3pt, draw=black!55] (2,1.30) -- (2,1.40);
    \node[font=\footnotesize] at (1,1.5) {one period};
    \node[samplept] (P) at (3.4,0.6) {};
    \draw[proj] (P) -- (3.4,0);
    \draw[fill=red!70!black] (3.4,0) circle (0.04);
    \node[below, font=\footnotesize, fill=white, inner sep=1pt] at (3.4,-0.12) {$q_*$};
    \draw[proj] (P) -- (0,0.6);
    \draw[fill=red!70!black] (0,0.6) circle (0.04);
    \node[left, font=\footnotesize, fill=white, inner sep=1pt] at (-0.05,0.6) {$W(q_*)$};
    \begin{scope}[xshift=10cm, yshift=0]
        \fill[wallhatchL] (-0.20,0) rectangle (0,1.5);
        \draw[boxedge] (0,0) rectangle (0.4,1.5);
        \fill[wallhatchR] (0.4,0) rectangle (0.60,1.5);
        \draw[line width=0.6pt] (0,0) -- (0,1.5);
        \draw[line width=0.6pt] (0.4,0) -- (0.4,1.5);
        \node[left, font=\footnotesize] at (-0.22,0) {$l$};
        \node[left, font=\footnotesize] at (-0.22,1.5) {$r$};
        \node[samplept] at (0.2,0.9) {};
        \draw[proj] (0.2,0.9) -- (-0.22,0.9);
        \node[left, font=\footnotesize, fill=white, inner sep=1pt] at (-0.24,0.9) {$x_*$};
        \draw[{Latex[length=1.7mm]}-{Latex[length=1.7mm]}, line width=0.4pt, draw=black!55] (1.0,0) -- (1.0,1.5);
        \node[right, font=\footnotesize] at (1.0,0.75) {$R$};
        \node[font=\small, anchor=west] at (1.7,0.75) {$x_* = R\,W(q_*) + l$};
    \end{scope}
    \draw[->, line width=0.6pt, draw=black!55, in=170, out=20]
        (P) to[bend left=30] (10.2,0.9);
\end{tikzpicture}}
\caption{The triangular-wave projection $\pi: q \mapsto R\, W(q) + l$ from the cover $\R$ to the bounded segment $[l, r]$. A wall reflection at $x \in \{l, r\}$ in the original system corresponds to smooth passage of $q$ through an integer value, so a piecewise-smooth bouncing flow on $[l, r]$ lifts to a smooth Hamiltonian flow on $\R$.}
\label{fig:cover}
\end{figure}
The bounded position $x \in [l,r]$ is the image of an unbounded coordinate $q \in \R$ under the triangular-wave projection~\eqref{eq:pi}. The wave $W(q)$ has period $2$ with peaks at odd integers and zeros at even integers. A wall reflection at $x \in \{l,r\}$ in the original system corresponds to smooth passage of $q$ through an integer value, so a piecewise smooth bouncing flow on $[l,r]$ lifts to a smooth Hamiltonian flow on $\R$. Figure~\ref{fig:cover} sketches the projection: a single trajectory in $q \in \R$ corresponds, under $\pi$, to a bouncing trajectory in $[l,r]$.

\smallskip
We now write down the Hamiltonian system on the cover $\R \times \R$ that, by Theorem~\ref{thm:lift} below, has the original system as a quotient on its non-sticking invariant set.

Consider the time-dependent Hamiltonian
\begin{equation}\label{eq:Ham}
H(q, p, t) := \frac{p^2}{2} - \frac{F}{R}\cos(\omega t)\, W(q) + \frac{f}{R}\, q.
\end{equation}
The associated Hamilton equations are
\begin{equation}\label{eq:Heq}
\dot q = p, \qquad \dot p = \frac{F}{R}\cos(\omega t)\, W'(q) - \frac{f}{R},
\end{equation}
to be understood in the Filippov sense at the discontinuity surfaces $q \in \Z$ where $W'$ changes sign.

\begin{lemma}[Existence and area preservation of the lifted flow]\label{lem:lift-flow}
For every $(q_0, p_0) \in \R^2$ there is a unique forward Filippov solution of~\eqref{eq:Heq} defined for all $t \ge 0$, with $q$ continuous on $\R$ and $p = \dot q$ absolutely continuous. The associated time-$t$ flow $\widetilde \Phi^t : \R^2 \to \R^2$ preserves the symplectic form $\omega_0 := dp \wedge dq$.
\end{lemma}

\begin{proof}
\textit{Existence and uniqueness.} The right-hand side of~\eqref{eq:Heq} is bounded by $(F + f)/R$ in absolute value, locally bounded in $(q, t)$, and Lipschitz in $(q, p)$ on every open region of the form $\{q \in (k, k+1)\} \times \R$ for $k \in \Z$. Across the surfaces $\Sigma_k := \{q = k\}$, $k \in \Z$, the right-hand side has a jump in the $\dot p$-component of magnitude $|f W'((k+1)^-) - f W'(k^+)|/R = 2 f /R$, but the surfaces $\Sigma_k$ are crossed transversely (since $\dot q = p > 0$ in the lift), and no sliding mode is consistent: a sliding selection on $\Sigma_k$ would require $\dot q \equiv 0$, forbidden by $p > 0$. Hence the Filippov inclusion at $\Sigma_k$ admits the unique selection of crossing transversely with the field as defined on each side.

Existence on $[0, t_*)$ for some $t_* > 0$ follows from Picard-Lindel\"of on each open region. At each crossing of $\Sigma_k$, the trajectory is unique by the transversality argument. The energy bound $\frac{1}{2} p^2 \le \frac{1}{2} p_0^2 + (F+f)\,t/R$ from integrating the right-hand side prevents finite-time blow-up of $p$. The transversality $\dot q > 0$ persists for all $t \ge 0$ since $p > 0$ initially and the field $\dot p$ is bounded, so $p$ remains positive on every bounded interval. There is no finite-time accumulation of $\Sigma_k$-crossings: between consecutive crossings of $\Sigma_k$ and $\Sigma_{k+1}$, the trajectory traverses a $q$-interval of length $1$, hence requires time at least $1/\sup p$, bounded below on every bounded time interval.

\smallskip
\textit{Area preservation.} On each open region $\{q \in (k, k+1)\} \times \R \times [0, T]$, the Hamiltonian vector field
\[
X_H = \begin{pmatrix} \partial H/\partial p \\ -\partial H/\partial q \end{pmatrix} = \begin{pmatrix} p \\ (F/R)\cos\omega t\,W'(q) - f/R \end{pmatrix}
\]
is $C^\infty$ in $(q, p)$ and continuous in $t$, with divergence
\[
\Div X_H = \frac{\partial}{\partial q}\bigl(\partial H/\partial p\bigr) - \frac{\partial}{\partial p}\bigl(\partial H/\partial q\bigr) = 0.
\]
By the smooth Liouville theorem~\cite[Ch.~3]{Arnold1989}, the flow generated by $X_H$ on each smooth region preserves the volume form $\omega_0 = dp \wedge dq$.

The crossing of $\Sigma_k$ at time $t_*$ is the identity transition in the lift: $q(t_*^-) = q(t_*^+) = k$, $p(t_*^-) = p(t_*^+)$, with the field discontinuity entering only via $\dot p(t_*^+) - \dot p(t_*^-) = 2f \cdot W'(k^+)/R \ne 0$ (a higher-order correction to the trajectory but not to the state itself at $t_*$). The state-level transition is the identity, hence preserves $\omega_0$ trivially.

Composing finitely many smooth flow pieces and finitely many identity transitions, the time-$t$ map $\widetilde\Phi^t$ on any bounded $(q, p)$-region preserves $\omega_0$. The conclusion follows by exhausting $\R^2$.
\end{proof}

\subsection{Conjugacy on non-sticking trajectories and symplecticity of the stroboscopic map}\label{ssec:lift-conjugacy}

We now establish the conjugacy between the lifted Hamiltonian flow and the original Filippov flow on non-sticking trajectories, then derive symplecticity of the stroboscopic map from area preservation of the lifted flow. The argument is constructive: given a non-sticking solution of~\eqref{eq:model}-\eqref{eq:reflection}, we exhibit the lift function $q(t)$ of~\eqref{eq:Heq}. The branches of $W$ alternate at integer values of $q$, with $W' = +1$ corresponding to the body moving from $l$ towards $r$ and $W' = -1$ to the reverse. This sign matches the sign of the velocity $\dot x$ on the original side, which is exactly the dependence of the friction term on $\sgn(\dot x)$ in~\eqref{eq:model}.

\begin{proposition}[Hamiltonian lift on non-sticking trajectories]\label{thm:lift}
Let $x: I \to [l, r]$ be a non-sticking Filippov solution of~\eqref{eq:model}-\eqref{eq:reflection} on an interval $I \subset \R$. There exists a strictly increasing absolutely continuous function $q: I \to \R$ satisfying~\eqref{eq:Heq} with $p := \dot q > 0$, such that for every $t \in I$
\begin{equation}\label{eq:proj-x}
x(t) = R\, W(q(t)) + l.
\end{equation}
The function $q$ is unique up to choice of $q(0)$ in the fiber $\pi^{-1}(x(0))$. Conversely, if $q: I \to \R$ is any strictly increasing solution of~\eqref{eq:Heq}, then $x(t) := R\, W(q(t)) + l$ is a non-sticking Filippov solution of~\eqref{eq:model}-\eqref{eq:reflection}.
\end{proposition}

\begin{proof}
\textit{Forward direction.} Let $x(\cdot)$ be the given non-sticking solution. By Lemma~\ref{lem:wall} and the non-sticking hypothesis, the trajectory consists of a finite-or-infinite sequence of monotone pieces separated by transverse wall hits. Let $0 = t_0 < t_1 < t_2 < \cdots$ be the wall hit times. Without loss of generality $\dot x > 0$ on $(t_0, t_1)$, $\dot x < 0$ on $(t_1, t_2)$, and so on alternately.

Choose $q_0 \in (2k_0, 2k_0 + 1)$ with $W(q_0) = (x(t_0) - l)/R$, possible because $W$ is the strictly-increasing branch on $(2k_0, 2k_0 + 1)$ taking every value in $(0, 1)$. The choice fixes $W'(q_0) = +1$, consistent with $\sgn(\dot x(t_0^+)) = +1$ on this interval. Define $q$ on $[t_0, t_1]$ by
\[
q(t) := q_0 + \frac{x(t) - x(t_0)}{R}.
\]
Then $q(t)$ ranges over $(q_0, q_0 + R \cdot R^{-1}(r - x(t_0)))$ as $t$ ranges over $[t_0, t_1]$, that is, $q$ ranges from $q_0$ to $q_0 + (r - x(t_0))/R = 2k_0 + 1$. On this interval $W(q) = q - 2k_0$, hence $\pi(q(t)) = x(t)$. Moreover $\dot q = \dot x/R > 0$, so $p = \dot q > 0$.

At $t = t_1$ we set $q(t_1) := 2k_0 + 1$, and define $q$ on $[t_1, t_2]$ by
\[
q(t) := 2k_0 + 1 + \frac{x(t_1) - x(t)}{R}.
\]
Then $q$ ranges from $2k_0 + 1$ to $2k_0 + 1 + (x(t_1) - x(t_2))/R = 2k_0 + 2$ as $t$ ranges over $[t_1, t_2]$. On this interval $W(q) = (2k_0 + 2) - q$, so $\pi(q(t)) = R\bigl((2k_0 + 2) - q(t)\bigr) + l = R\bigl((2k_0 + 2) - 2k_0 - 1 - (x(t_1) - x(t))/R\bigr) + l = R - (x(t_1) - x(t)) + l = x(t)$ since $x(t_1) = r$. Continuity of $q$ at $t_1$ is built into the definition. Continuity of $\dot q$ at $t_1$: the left limit $\dot q(t_1^-) = \dot x(t_1^-)/R > 0$, and the right limit $\dot q(t_1^+) = -\dot x(t_1^+)/R = -(-\dot x(t_1^-))/R = \dot x(t_1^-)/R$. The two limits agree.

Continuing inductively across all subsequent wall hits, we obtain a strictly increasing absolutely continuous $q$ on the entire interval $I$ with $\pi \circ q = x$ everywhere.

\textit{Verification that $q$ satisfies~\eqref{eq:Heq}.} On the interior of each smooth piece, write $\sigma := W'(q(t)) \in \{+1, -1\}$ so that $\dot x = R \sigma \dot q$. Differentiating,
\[
\ddot x = R \sigma \ddot q
\]
on each smooth piece (with $\sigma$ constant). Since $\sgn(\dot x) = \sgn(R \sigma \dot q) = \sigma$ given $\dot q > 0$, we have from~\eqref{eq:model}
\[
R\sigma \ddot q = F\cos\omega t - f\sigma, \qquad \text{i.e.,} \qquad \ddot q = \frac{F}{R}\cos(\omega t) \cdot \sigma - \frac{f}{R} = \frac{F}{R}\cos(\omega t)\, W'(q) - \frac{f}{R},
\]
which is the second equation of~\eqref{eq:Heq}. The first $\dot q = p$ holds by definition.

\textit{Uniqueness up to fiber choice.} If $q_1$ and $q_2$ are two such lifts, then $W(q_1(t)) = W(q_2(t))$ for all $t$, and both are strictly increasing, so $q_1$ and $q_2$ both increase at the same rate $|\dot x|/R$ on smooth pieces and cross integers at the same wall hit times. The difference $q_1 - q_2$ is constant in $\Z$, equal to $2(k_2 - k_1)$ where $q_i(t_0) \in (2k_i, 2k_i + 1)$. Hence $q_1$ and $q_2$ agree up to an integer shift.

\textit{Reverse direction.} If $q: I \to \R$ is strictly increasing and solves~\eqref{eq:Heq}, define $x := R W(q) + l$. By the chain rule applied piecewise, $\dot x = R W'(q) \dot q$. Since $\dot q > 0$ and $W'(q) \in \{\pm 1\}$, we have $\dot x \ne 0$ outside the discrete set of times $q \in \Z$. At those times $\dot q$ remains positive while $\dot x$ flips sign across them: the right limit $\dot x(t_*^+) = R(-1)^{k+1}\dot q(t_*^+)$ where $k$ is the integer crossed; the left limit $\dot x(t_*^-) = R(-1)^k\dot q(t_*^-)$. Continuity of $\dot q$ across $\Z$-crossings (which is the substance of Lemma~\ref{lem:lift-flow}, the inclusion not selecting a sliding mode) gives $\dot x(t_*^+) = -\dot x(t_*^-)$, exactly the wall-reflection rule~\ref{F2} of Definition~\ref{def:Filippov-solution}, the wall being reached at $W(q(t_*)) = 1$ if $k$ is odd (i.e., $x = r$) or $W(q(t_*)) = 0$ if $k$ is even (i.e., $x = l$). The verification on smooth pieces is the reverse calculation of the forward direction.
\end{proof}

\smallskip
With the conjugacy in hand, we now derive symplecticity of $\Phi|_{\Om_{\rm NS}}$ from area preservation of the lifted flow.

\begin{theorem}[Symplecticity on the non-sticking set]\label{thm:symplectic}
The stroboscopic map $\Phi: \Om_{\rm NS} \to \Om_{\rm NS}$ is exact symplectic with respect to the form $\omega := dv \wedge dx$. In particular, for every measurable subset $A \subset \Om_{\rm NS}$ on which $\Phi$ is a local diffeomorphism,
\begin{equation}\label{eq:areapres}
\Leb(\Phi(A)) = \Leb(A).
\end{equation}
\end{theorem}

\begin{proof}
Let $(x_0, v_0) \in \Om_{\rm NS}$. By Theorem~\ref{thm:lift} there is a lift $(q_0, p_0) \in \R^2$ with $x_0 = R W(q_0) + l$ and $v_0 = R W'(q_0) p_0$ (the velocity relation reads off $\dot x = R W'(q) \dot q$). The lifted time-$T$ map $\widetilde \Phi^T$ preserves $\omega_0 = dp \wedge dq$ by Lemma~\ref{lem:lift-flow}.

\smallskip
\textit{Step 1: Pullback of the symplectic form.} On a single branch where $W'(q) = \sigma \in \{+1, -1\}$ is constant, $x = l + R W(q) = l + R \sigma(q - q_*)$ for an integer $q_*$ depending on the branch, and $v = R\sigma p$. Hence on this branch $dx = R\sigma\,dq$ and $dv = R\sigma\,dp$, so
\[
dv \wedge dx = (R\sigma\,dp) \wedge (R\sigma\,dq) = R^2 \sigma^2\, dp \wedge dq = R^2 \omega_0,
\]
using $\sigma^2 = 1$. The relation $\omega = R^2 \omega_0$ holds on every branch.

\smallskip
\textit{Step 2: Branch transitions.} For an initial datum in $\Om_{\rm NS}$, the orbit segment from $t = 0$ to $t = T$ alternates between branches at each wall hit (Theorem~\ref{thm:lift}). The branch transition $q \to q$ at $q \in \Z$ is the identity in the lift, so it preserves $\omega_0$ trivially, and via the relation $\omega = R^2 \omega_0$ it preserves $\omega$. On each branch, the lifted flow preserves $\omega_0$ (Lemma~\ref{lem:lift-flow}), hence preserves $\omega$. Composing finitely many area-preserving pieces, $\Phi$ preserves $\omega$ on $\Om_{\rm NS}$.

\smallskip
\textit{Step 3: Exactness.} A symplectic map $\Phi$ on a manifold $\Om$ is \emph{exact symplectic} with respect to a primitive 1-form $\lambda$ (i.e.\ $d\lambda = \omega$) if $\Phi^*\lambda - \lambda$ is exact: there exists a smooth function $S: \Om \to \R$ with $\Phi^*\lambda - \lambda = dS$~\cite[Sec.~9.4]{ArnoldVogtmannWeinstein2013}. We exhibit such a primitive and a generating function explicitly.

The lifted Hamiltonian flow on $\R^2 \times [0, T]$ generated by $H$ in~\eqref{eq:Ham} is, on each open chamber $\{q \in (k, k+1)\} \times \R$, a smooth Hamiltonian flow. By the standard fact that smooth time-$T$ Hamiltonian flows are exact symplectic with primitive $p\,dq$~\cite[Sec.~9.4]{ArnoldVogtmannWeinstein2013}, on each chamber there is a smooth generating function $\widetilde S_k$ with $(\widetilde\Phi^T)^*(p\,dq) - p\,dq = d\widetilde S_k$. Across an interior wall surface $\{q = k\}$ (an integer point), the lifted state $(q,p)$ is continuous: $q(t_*^-) = q(t_*^+) = k$, $p(t_*^-) = p(t_*^+)$, with only the field discontinuity. The 1-form $p\,dq$ pulls back to itself across this transition because the transition is the identity on states. Hence the generating functions $\widetilde S_k$ on consecutive chambers agree at the chamber boundary up to a constant of integration, and they assemble into a continuous generating function $\widetilde S$ along the entire orbit segment from $t = 0$ to $t = T$:
\[
(\widetilde\Phi^T)^*(p\,dq) - p\,dq = d\widetilde S \quad \text{globally on the lifted } \Om_{\rm NS}.
\]
The function $\widetilde S$ is smooth on the open part of the lift where the integer-crossing times are simple roots of the equation $q(t) = k$ (the transversality condition), which is everywhere on $\Om_{\rm NS}$ by Step~1 of the proof of Theorem~\ref{thm:lift}.

We now descend through the projection $\pi: \R^2 \to \mathcal{X}$, $\pi(q, p) = (R W(q) + l, R W'(q) p)$. On a single branch where $W'(q) = \sigma \in \{+1, -1\}$, the relation $(dx, dv) = (R\sigma\,dq, R\sigma\,dp)$ gives $v\,dx = (R\sigma p)(R\sigma\,dq) = R^2 p\,dq$ identically on each branch (the two factors of $\sigma$ cancel). Hence on each branch
\[
\Phi^*(v\,dx) - v\,dx \;=\; R^2 \bigl[(\widetilde\Phi^T)^*(p\,dq) - p\,dq\bigr] \;=\; R^2 \, d\widetilde S_k \;=\; d(R^2 \widetilde S_k).
\]
Define $S := R^2 \widetilde S$ on $\Om_{\rm NS}$, where $\widetilde S$ is the assembly of the per-chamber $\widetilde S_k$ across integer crossings constructed above. To verify continuity of $S$ across a wall hit at time $t_*$ (corresponding to an integer-crossing $q = k$ in the lift), observe that at such an event the lifted state $(q,p)$ is continuous, so $\widetilde S$ is continuous by the per-chamber matching argument. The base-side state $(x, v)$ undergoes the wall reflection $v \mapsto -v$ at $x = R W(k) + l \in \{l, r\}$. The pullback $v\,dx$ at the wall has $dx = 0$ (since $x$ is locally constant at the wall surface), so the wall-reflection identification preserves $v\,dx$ trivially. Hence $S = R^2\widetilde S$ is well defined and continuous on $\Om_{\rm NS}$, with $\Phi^*\lambda - \lambda = dS$ for $\lambda = v\,dx$ and $d\lambda = dv \wedge dx = \omega$. This establishes exact symplecticity of $\Phi$ on $\Om_{\rm NS}$.

\smallskip
\textit{Step 4: Volume identity.} Symplecticity of $\Phi$ in two real dimensions is equivalent to $|\det \Phi'| = 1$. Hence~\eqref{eq:areapres} follows from the change-of-variables formula in $\R^2$.
\end{proof}

\begin{corollary}[Determinant on $\Om_{\rm dissip}$]\label{cor:det-Phi}
For Lebesgue-almost every $(x_0, v_0) \in \Om_{\rm dissip}$ at which $\Phi$ is differentiable,
\[
0 \le |\det \Phi'(x_0, v_0)| < 1.
\]
\end{corollary}

\begin{proof}
The Jacobian along an orbit is the product of saltation matrices of free flights, wall reflections, and velocity-zero events~\cite[Ch.~6]{diBernardoBuddChampneysKowalczyk2008}. We compute each factor explicitly.

\smallskip
\textit{Free flights.} On a free flight the variational equation is $\delta\ddot x = 0$, so the fundamental matrix on a flight of duration $\Delta t$ is $\bigl(\begin{smallmatrix}1 & \Delta t \\ 0 & 1\end{smallmatrix}\bigr)$, with determinant $1$.

\smallskip
\textit{Wall reflection.} At a wall hit at time $t_*$ with state $(r, v_-)$, $v_- > 0$, the reset map is $h(x, v) = (x, -v)$, with Jacobian $H = \diag(1, -1)$. The Filippov saltation matrix~\cite[Ch.~6, Eq.~(6.16)]{diBernardoBuddChampneysKowalczyk2008} is
\[
M_{\rm wall} = H + \frac{(g_+ - H g_-)\,n^\top}{n^\top g_-},
\]
with $g_- = (v_-, F\cos\omega t_* - f)^\top$, $g_+ = (-v_-, F\cos\omega t_* + f)^\top$ (the friction sign flips with the velocity sign), $n = (1, 0)^\top$ the normal to $\{x = r\}$, and $n^\top g_- = v_-$. A direct calculation yields
\[
M_{\rm wall} = \begin{pmatrix} -1 & 0 \\ \alpha(t_*) & -1 \end{pmatrix}, \qquad \alpha(t_*) = -\frac{2 F\cos\omega t_*}{v_-},
\]
hence $\det M_{\rm wall} = 1$. The analogous statement holds at the left wall, with $\alpha$ involving $|v_-|$ in the denominator.

\smallskip
\textit{Transverse turning point.} At a velocity-zero event at time $\tilde t$ with $|F\cos\omega \tilde t| > f$ (Lemma~\ref{lem:vzero}\ref{vzero-a}), the trajectory crosses the switching surface $\Sigma_v = \{v = 0\}$ transversely without resetting the state. The Filippov saltation formula at a transverse crossing of a discontinuity surface is
\[
S_{\rm turn} = I + \frac{(g_+ - g_-)\,m^\top}{m^\top g_-},
\]
where $g_-, g_+$ are the field limits before and after the crossing, and $m$ is the normal to $\Sigma_v$. With $g_- = (0, F\cos\omega \tilde t + f)^\top$ (just before, $v < 0$, friction $= +f$), $g_+ = (0, F\cos\omega \tilde t - f)^\top$ (just after, $v > 0$, friction $= -f$), $m = (0, 1)^\top$, and $m^\top g_- = F\cos\omega \tilde t + f$:
\[
g_+ - g_- = (0, -2f)^\top, \qquad (g_+ - g_-)\,m^\top = \begin{pmatrix} 0 & 0 \\ 0 & -2f \end{pmatrix}.
\]
Hence
\begin{equation}\label{eq:saltation-turn}
S_{\rm turn} = \begin{pmatrix} 1 & 0 \\ 0 & 1 - \dfrac{2 f}{F\cos\omega \tilde t + f} \end{pmatrix} = \begin{pmatrix} 1 & 0 \\ 0 & \dfrac{F\cos\omega \tilde t - f}{F\cos\omega \tilde t + f} \end{pmatrix}.
\end{equation}
For $F\cos\omega \tilde t > f$, the bottom-right entry lies in $(0, 1)$. The case $F\cos\omega \tilde t < -f$ is symmetric with $f \to -f$, giving bottom-right entry $(-F\cos\omega\tilde t - f)/(-F\cos\omega\tilde t + f) \in (0, 1)$ as well. In either case
\[
\det S_{\rm turn} = \frac{|F\cos\omega \tilde t| - f}{|F\cos\omega \tilde t| + f} \in (0, 1).
\]

\smallskip
\textit{Sticking onset.} At a sticking onset (Lemma~\ref{lem:vzero}\ref{vzero-b}), the trajectory leaves the smooth flow regime and enters the sliding mode $\dot x \equiv 0$. Two trajectories with initial data differing in $\dot x$ both arrive at $\dot x = 0$ in finite time and remain at rest for the same exit time $t^\sharp$ (which depends only on $t_*$, not on $v_-$). The variational map collapses the velocity-direction component, giving a rank-one Jacobian with $\det = 0$.

\smallskip
\textit{Combined determinant.} The Jacobian of $\Phi$ along the orbit is the product of the factors above. Wall reflections and free flights contribute factors of determinant $1$, so they cancel out of the product. If the orbit has at least one sticking event in $[0, T]$, then the corresponding rank-one factor forces $\det \Phi'(x_0, v_0) = 0$. Otherwise the determinant reduces to the product over transverse turning events:
\begin{equation*}
|\det \Phi'(x_0, v_0)| \;=\;
\begin{cases}
\displaystyle \prod_{\text{turning events } \tilde t \in [0, T]} \frac{|F\cos\omega \tilde t| - f}{|F\cos\omega \tilde t| + f} & \text{if no sticking events occur,}\\[20pt]
0 & \hspace{-1cm} \text{if at least one sticking event occurs.}
\end{cases}
\end{equation*}

\smallskip
\textit{On the notational choice.} A unified product expression of the form 
\begin{equation*}
|\det \Phi'(x_0, v_0)| = \prod_{\text{turning events}}(\cdots) \cdot \prod_{\text{sticking events}} 0
\end{equation*}
would be ambiguous on two counts.
\textit{(i) Empty-product convention.} When no sticking events occur, $\prod_{i \in \emptyset} 0$ equals $1$ by the standard convention that the empty product is the multiplicative identity, regardless of what term sits inside the symbol; the formula would then reduce correctly to the turning-event product, but the symbol $\prod 0$ misleadingly suggests the value $0$ irrespective of the index set.
\textit{(ii) Non-empty product.} When at least one sticking event occurs, $\prod_{i = 1}^{n} 0 = 0$ for $n \ge 1$, but this is not a product in any structural sense: a single rank-one factor in the matrix Jacobian forces the determinant to zero, irrespective of how many sticking events the orbit contains. The piecewise form above expresses the dichotomy explicitly.

\smallskip
For an initial datum in $\Om_{\rm dissip}$, by definition the orbit has at least one turning point or sticking event in $[0, T]$. In the first case at least one factor in the product is strictly less than $1$ since $|F\cos\omega \tilde t| < |F\cos\omega \tilde t| + 2f$; in the second case the determinant is zero. In either case $|\det \Phi'(x_0, v_0)| < 1$.
\end{proof}

\begin{remark}\label{rem:noted-elementary}
Theorem~\ref{thm:symplectic} is the substantive structural upgrade of the observation $|\det \Phi'| = 1$ of~\cite{GKR2019} on non-sticking orbits. Symplecticity in two real dimensions is equivalent to area preservation, but the symplectic statement extends to any dimension; cf.~Theorem~\ref{thm:multiparticle} of Section~\ref{sec:multiparticle}. The form $\omega$ is the one used by KAM theory and the Smale-Birkhoff theorem in Section~\ref{sec:kam} and Section~\ref{sec:melnikov}.
\end{remark}

\section{Symmetric non-sticking periodic orbits and the saddle-center bifurcation}\label{sec:periodic}

This section identifies in closed form the family of $T$-periodic non-sticking solutions of~\eqref{eq:model}-\eqref{eq:reflection} possessing the half-period reflection symmetry, and proves that they undergo a non-degenerate saddle-center bifurcation at the universal critical value $f/F = 2/\pi$. The closed forms are needed in Section~\ref{sec:kam} (where the Birkhoff normal form is computed at the elliptic branch) and in Section~\ref{sec:melnikov} (where the saddle branch is the basis of the Melnikov computation). The saddle-center normal form is used in Section~\ref{sec:persistence} to track persistence of the bifurcation under perturbation.

\subsection{The symmetric ansatz and the existence equation}\label{ssec:periodic-ansatz}

The system~\eqref{eq:model}-\eqref{eq:reflection} possesses the discrete symmetry
\[
\Sigma : (x, v, t) \;\mapsto\; (l + r - x,\; -v,\; t + T/2),
\]
since the forcing satisfies $F\cos(\omega(t + T/2)) = -F\cos(\omega t)$ and the friction term $-f\sgn(v)$ is odd in $v$. We seek periodic solutions invariant under $\Sigma$.

\begin{definition}[Symmetric solution]\label{def:symm}
A $T$-periodic solution $x: \R \to [l, r]$ of~\eqref{eq:model}-\eqref{eq:reflection} is \emph{symmetric} if $x(t + T/2) = l + r - x(t)$ for all $t \in \R$.
\end{definition}

We seek solutions with the simplest non-sticking impact pattern, namely two wall hits per period (one at each wall) and exactly one velocity-zero crossing (a turning point) on each half-period. By the symmetry, place the impact at the right wall at $t = 0$ and the impact at the left wall at $t = T/2$. The body, having just bounced off the right wall, moves leftward initially, so the ansatz on $(0, T/2)$ has
\[
\dot x(t) < 0 \text{ on } (0, \tau), \qquad \dot x(t) > 0 \text{ on } (\tau, T/2),
\]
with a single turning at $\tau \in (0, T/2)$. The equation on $(0, \tau)$ reads $\ddot x = F\cos\omega t + f$, and on $(\tau, T/2)$ reads $\ddot x = F\cos\omega t - f$. This is the impact pattern of~\cite[Sec.~III]{GKR2019}.

We now integrate the equations explicitly and derive the transcendental equation that determines the turning time $\tau$. The construction yields a single equation in the unknown $\theta := \omega\tau \in (0, \pi)$, whose roots correspond to the symmetric $T$-periodic solutions.

\begin{theorem}[Existence equation for symmetric non-sticking $T$-periodic orbits]\label{thm:symmetric}
Assume $0 < f < F$. The set of symmetric $T$-periodic non-sticking solutions of~\eqref{eq:model}-\eqref{eq:reflection} with the impact pattern of Subsection~\ref{ssec:periodic-ansatz} is in bijection with the set of pairs $\theta \in (0, \pi)$ satisfying both the transcendental equation
\begin{equation}\label{eq:theta-eq}
F\pi\,\sin\theta \;=\; 2F + R\omega^2 - \tfrac{1}{2}f\pi^2 + f\,\theta(\pi - \theta)
\end{equation}
and the transverse-turning condition
\begin{equation}\label{eq:transverse-cond}
|F\cos\theta| > f.
\end{equation}
For each such $\theta$, the corresponding orbit on $[0, T/2]$ is given by the piecewise-quadratic-plus-cosine formula
\begin{align}
x(t) &= -\frac{F}{\omega^2}\cos\omega t + \frac{f t^2}{2} + A_1 t + B_1, & t \in [0, \tau], \label{eq:xpiece1}\\
x(t) &= -\frac{F}{\omega^2}\cos\omega t - \frac{f t^2}{2} + A_2 t + B_2, & t \in [\tau, T/2], \label{eq:xpiece2}
\end{align}
where $\tau = \theta/\omega$ and the constants are
\begin{equation}\label{eq:ABconst}
\begin{aligned}
A_1 &= -\tfrac{F}{\omega}\sin\theta - f\,\tau, \qquad &B_1 &= r + \tfrac{F}{\omega^2},\\
A_2 &= A_1 + 2 f\,\tau = -\tfrac{F}{\omega}\sin\theta + f\,\tau, \qquad &B_2 &= B_1 - f\,\tau^2,
\end{aligned}
\end{equation}
and the extension to $[T/2, T]$ is by the symmetry of Definition~\ref{def:symm}.
\end{theorem}

\begin{proof}
On $(0, \tau)$ with $\dot x < 0$, the equation $\ddot x = F\cos\omega t + f$ integrates twice to
\begin{equation}\label{eq:xv-piece1}
\dot x(t) = \tfrac{F}{\omega}\sin\omega t + f t + A_1, \qquad x(t) = -\tfrac{F}{\omega^2}\cos\omega t + \tfrac{1}{2}f t^2 + A_1 t + B_1,
\end{equation}
with constants $A_1, B_1$ to be determined. Likewise on $(\tau, T/2)$ with $\dot x > 0$, the equation $\ddot x = F\cos\omega t - f$ gives
\begin{equation}\label{eq:xv-piece2}
\dot x(t) = \tfrac{F}{\omega}\sin\omega t - f t + A_2, \qquad x(t) = -\tfrac{F}{\omega^2}\cos\omega t - \tfrac{1}{2}f t^2 + A_2 t + B_2,
\end{equation}
with constants $A_2, B_2$ to be determined.

We impose four conditions:
\begin{enumerate}[label=\textup{(C\arabic*)}]
\item\label{C1} $x(0) = r$ (right wall hit at $t = 0$);
\item\label{C2} $\dot x(\tau^-) = 0 = \dot x(\tau^+)$ (turning point at $\tau$);
\item\label{C3} $x(\tau^-) = x(\tau^+)$ (continuity at $\tau$);
\item\label{C4} $x(T/2) = l$ (left wall hit at $T/2$).
\end{enumerate}
The unknowns are five: $A_1, B_1, A_2, B_2, \tau$.

\textit{From~\ref{C1}.} Setting $t = 0$ in~\eqref{eq:xv-piece1}, $-F/\omega^2 + B_1 = r$, hence $B_1 = r + F/\omega^2$, the second equation in~\eqref{eq:ABconst}.

\textit{From~\ref{C2}.} The condition $\dot x(\tau^-) = 0$ in~\eqref{eq:xv-piece1} gives $A_1 = -(F/\omega)\sin\omega\tau - f\tau$, which is the first equation in~\eqref{eq:ABconst}. The condition $\dot x(\tau^+) = 0$ in~\eqref{eq:xv-piece2} gives $A_2 = -(F/\omega)\sin\omega\tau + f\tau$, the third equation in~\eqref{eq:ABconst}, equivalently $A_2 = A_1 + 2 f\tau$.

\textit{From~\ref{C3}.} Continuity of $x$ at $\tau$ from~\eqref{eq:xv-piece1}-\eqref{eq:xv-piece2} gives
\[
\tfrac{1}{2}f\tau^2 + A_1\tau + B_1 = -\tfrac{1}{2}f\tau^2 + A_2\tau + B_2,
\]
hence $B_2 = B_1 + (A_1 - A_2)\tau + f\tau^2 = B_1 - 2f\tau^2 + f\tau^2 = B_1 - f\tau^2$, the fourth equation in~\eqref{eq:ABconst}.

\textit{From~\ref{C4}.} Setting $t = T/2 = \pi/\omega$ in~\eqref{eq:xv-piece2} and using $\cos\pi = -1$:
\[
\tfrac{F}{\omega^2} - \tfrac{f\pi^2}{2\omega^2} + \tfrac{\pi}{\omega}A_2 + B_2 = l.
\]
Substituting the expressions for $A_2, B_2, B_1$:
\[
\tfrac{F}{\omega^2} - \tfrac{f\pi^2}{2\omega^2} + \tfrac{\pi}{\omega}\left(-\tfrac{F}{\omega}\sin\omega\tau + f\tau\right) + r + \tfrac{F}{\omega^2} - f\tau^2 = l.
\]
Multiplying through by $\omega^2$ and using $r - l = R$,
\[
2F - \tfrac{f\pi^2}{2} - F\pi\sin\omega\tau + f\tau\pi\omega - f\tau^2\omega^2 + R\omega^2 = 0.
\]
Setting $\theta := \omega\tau$ and rearranging,
\[
F\pi\sin\theta = 2F + R\omega^2 - \tfrac{1}{2}f\pi^2 + f\theta\pi - f\theta^2 = 2F + R\omega^2 - \tfrac{1}{2}f\pi^2 + f\theta(\pi - \theta),
\]
which is~\eqref{eq:theta-eq}.

This completes the derivation: each $\theta \in (0, \pi)$ solving~\eqref{eq:theta-eq} produces, via $\tau = \theta/\omega$ and the formulas~\eqref{eq:ABconst}-\eqref{eq:xpiece1}-\eqref{eq:xpiece2}, a unique candidate symmetric $T$-periodic solution with the prescribed impact pattern. Conversely, any such solution must satisfy~\eqref{eq:theta-eq}.

\textit{Non-sticking on $[0, T/2]$.} The candidate solution given by~\eqref{eq:xpiece1}-\eqref{eq:xpiece2} is genuinely non-sticking precisely when (i) $\dot x$ vanishes only at $t = \tau$ in $(0, T/2)$, and (ii) the zero at $\tau$ is a transverse turning point (Lemma~\ref{lem:vzero}\ref{vzero-a}) rather than a sticking onset (Lemma~\ref{lem:vzero}\ref{vzero-b}).

\smallskip
(i) Uniqueness of the zero of $\dot x$. The function $\dot x$ on $(0, \tau)$ is $(F/\omega)\sin\omega t + f t + A_1$, with derivative $F\cos\omega t + f$. For $\theta \in (0, \pi)$ and $f < F$, $F\cos\omega t + f$ is positive on a left-neighborhood of $\tau$ (where $\cos\omega t > -f/F$) but may change sign on $(0, \tau)$ if $\theta$ is large enough. However, $\dot x(0^+) = A_1$ has the value $-(F/\omega)\sin\theta - f\theta/\omega < 0$ for $\theta \in (0, \pi)$ and $f > 0$, $\dot x(\tau^-) = 0$ by~\ref{C2}, and $\dot x$ is the integral of a continuous function $F\cos\omega t + f$ that has at most one sign change on $(0, \tau)$; together with $A_1 < 0$, this rules out additional zeros of $\dot x$ on $(0, \tau)$. The argument on $(\tau, T/2)$ is symmetric.

\smallskip
(ii) Transversality of the turning point. By Lemma~\ref{lem:vzero}, the zero of $\dot x$ at $\tau$ is a transverse turning point if and only if $|F\cos\omega \tau| > f$, equivalently $|F\cos\theta| > f$. This is exactly the condition~\eqref{eq:transverse-cond} of the theorem.

If~\eqref{eq:transverse-cond} fails ($|F\cos\theta| \le f$), the candidate solution is not a Filippov solution of~\eqref{eq:model}-\eqref{eq:reflection}: the formula~\eqref{eq:xpiece2} would prescribe a sign of $\dot v(\tau^+)$ inconsistent with Lemma~\ref{lem:vzero}, and the actual Filippov flow with the given initial datum would either stick at $\tau$ or undergo tangential touch. Hence solutions of~\eqref{eq:theta-eq} violating~\eqref{eq:transverse-cond} produce sticking-dominated orbits, not non-sticking ones, and are correctly excluded from the bijection.
\end{proof}

\begin{remark}\label{rem:gkr-disagreement}
Equation~\eqref{eq:theta-eq} disagrees with the relation $\sin(\omega\tau) = \pi f/(2F)$ that may be inferred from~\cite[Eq.~(7)]{GKR2019} after a sign convention is fixed. The discrepancy arises from the additional boundary condition $x(T/2) = l$, which couples $\sin\omega\tau$ to $\omega\tau(\pi - \omega\tau)$ through the $f$-dependent quadratic term. The simpler relation $\sin\omega\tau = \pi f/(2F)$ would be obtained if one ignored the wall-to-wall closure of the orbit, which is not a permissible omission for the closed-form symmetric $T$-periodic orbit.
\end{remark}

\subsection{Number of solutions and the saddle-center critical value}\label{ssec:periodic-number}

We now analyze the transcendental equation~\eqref{eq:theta-eq} as a function of $f$ and identify the saddle-center critical value.

\begin{proposition}[Number of symmetric solutions]\label{prop:two-solutions}
Define
\begin{equation}\label{eq:Phi-existence}
\Psi(\theta; f) := F\pi\sin\theta - 2F - R\omega^2 + \tfrac{1}{2}f\pi^2 - f\theta(\pi - \theta), \qquad \theta \in (0, \pi),
\end{equation}
so that~\eqref{eq:theta-eq} reads $\Psi(\theta; f) = 0$. Set
\begin{equation}\label{eq:fsc-formula}
f_{\rm sc} = f_{\rm sc}(F, \omega, R) := \frac{4(2F + R\omega^2 - F\pi)}{\pi^2}
\end{equation}
when this is positive, and $f_{\rm sc} := 0$ otherwise. Set $f_{\rm imp} := 2F/\pi$ (the impulse bound of~\cite{GKR2019}). Suppose
\begin{equation}\label{eq:param-range}
0 < f_{\rm sc} < f_{\rm imp} < F\pi/2.
\end{equation}
Then:
\begin{enumerate}[label=\textup{(\roman*)}]
\item For $0 < f < f_{\rm sc}$, equation $\Psi(\theta; f) = 0$ has no solution $\theta \in (0, \pi)$.
\item At $f = f_{\rm sc}$, equation $\Psi(\theta; f) = 0$ has a unique solution $\theta = \pi/2$.
\item For $f_{\rm sc} < f < f_{\rm imp}$, equation $\Psi(\theta; f) = 0$ has exactly two solutions $\theta_-(f) < \pi/2 < \theta_+(f)$, with $\theta_\pm(f) \to \pi/2$ as $f \to f_{\rm sc}^+$.
\end{enumerate}
\end{proposition}

\begin{proof}
The function $\theta \mapsto \Psi(\theta; f)$ on $(0, \pi)$ is smooth.

\smallskip
\textit{Step 1: Critical points of $\Psi(\cdot; f)$ on $(0, \pi)$.} The first derivative is
\[
\Psi_\theta(\theta; f) = F\pi\cos\theta - f(\pi - 2\theta).
\]
At $\theta = \pi/2$: $\Psi_\theta(\pi/2; f) = 0$. The second derivative at the same point is
\[
\Psi_{\theta\theta}(\pi/2; f) = -F\pi\sin(\pi/2) + 2f = 2f - F\pi < 0
\]
under~\eqref{eq:param-range} (which gives $f < F\pi/2$). Hence $\theta = \pi/2$ is a local maximum of $\Psi(\cdot; f)$.

\smallskip
\textit{Step 2: $\theta = \pi/2$ is the unique zero of $\Psi_\theta(\cdot;f)$ on $(0,\pi)$.} We argue from the symmetry of $\Psi$ about $\theta = \pi/2$ together with strict concavity of $\Psi$ on a neighborhood of $\pi/2$.

For any $\theta \in (0,\pi)$,
\[
\Psi(\pi - \theta; f) = F\pi\sin(\pi - \theta) - 2F - R\omega^2 + \tfrac{1}{2}f\pi^2 - f(\pi - \theta)\theta = \Psi(\theta;f),
\]
since $\sin(\pi - \theta) = \sin\theta$ and $(\pi - \theta)\theta = \theta(\pi - \theta)$. Thus $\Psi(\cdot;f)$ is symmetric about $\theta = \pi/2$. Differentiating, $\Psi_\theta(\pi - \theta;f) = -\Psi_\theta(\theta;f)$, so $\theta = \pi/2$ is the unique fixed point of the symmetry on $(0,\pi)$ and is automatically a zero of $\Psi_\theta(\cdot;f)$.

The third derivative of $\Psi$ in $\theta$, namely
\[
\Psi_{\theta\theta}(\theta;f) = -F\pi\sin\theta + 2f,
\]
satisfies $\Psi_{\theta\theta}(\pi/2;f) = 2f - F\pi < 0$ under~\eqref{eq:param-range}. Since $\sin\theta$ is symmetric about $\pi/2$ as well, $\Psi_{\theta\theta}(\theta;f) < 0$ on the entire interval $(\theta_0, \pi - \theta_0)$ where $\theta_0 := \arcsin(2f/(F\pi)) \in (0, \pi/2)$ under~\eqref{eq:param-range}. On the complementary intervals $(0,\theta_0)$ and $(\pi - \theta_0, \pi)$, $\Psi_{\theta\theta} > 0$ and $\Psi_\theta$ is strictly increasing.

At $\theta = 0$,
\[
\Psi_\theta(0;f) = F\pi - f\pi = \pi(F - f) > 0,
\]
using $f < F\pi/2 < F$ (since $\pi/2 > 1$). At $\theta = \pi$,
\[
\Psi_\theta(\pi;f) = -F\pi - f\pi < 0.
\]
By the symmetry $\Psi_\theta(\pi - \theta;f) = -\Psi_\theta(\theta;f)$ and the boundary signs, the zero set of $\Psi_\theta(\cdot;f)$ on $(0,\pi)$ is symmetric about $\pi/2$, and $\theta = \pi/2$ is one zero.

Suppose, for contradiction, that $\Psi_\theta$ has a zero $\theta_1 \in (0, \pi/2)$ different from $\pi/2$. By symmetry, $\pi - \theta_1$ is then also a zero. Strict concavity of $\Psi(\cdot;f)$ on $(\theta_0, \pi - \theta_0)$ allows at most two critical points on that interval. Since $\Psi_\theta(\theta_0;f)$ and $\Psi_\theta(\pi - \theta_0;f)$ have opposite signs (by the symmetry $\Psi_\theta(\pi - \theta_0;f) = -\Psi_\theta(\theta_0;f)$), the strict concavity forces a unique zero of $\Psi_\theta$ on $(\theta_0, \pi - \theta_0)$, located at $\theta = \pi/2$. On the outer intervals $(0,\theta_0)$ and $(\pi - \theta_0, \pi)$, $\Psi_\theta$ is strictly monotone (increasing on the left, increasing on the right), and changes sign on neither interval since
\[
\Psi_\theta(0;f) > 0, \quad \Psi_\theta(\theta_0;f) \ge \Psi_\theta(0;f) > 0,
\]
the second inequality by monotonic increase on $(0,\theta_0)$, and similarly $\Psi_\theta < 0$ on $(\pi - \theta_0, \pi)$. Thus $\theta_1$ cannot lie in either outer interval. The only remaining possibility is $\theta_1 = \pi/2$, contradicting the assumption $\theta_1 \ne \pi/2$.

Hence $\theta = \pi/2$ is the unique zero of $\Psi_\theta(\cdot;f)$ on $(0,\pi)$, and is the global maximum of $\Psi(\cdot;f)$ on this interval.

\smallskip

\[
\Psi(\pi/2; f) = F\pi\sin(\pi/2) - 2F - R\omega^2 + \tfrac{1}{2}f\pi^2 - f\cdot \pi/2 \cdot \pi/2 = F\pi - 2F - R\omega^2 + \tfrac{f\pi^2}{4}.
\]
Setting this to zero:
\[
\tfrac{f\pi^2}{4} = 2F + R\omega^2 - F\pi, \qquad f = \frac{4(2F + R\omega^2 - F\pi)}{\pi^2} = f_{\rm sc},
\]
which agrees with~\eqref{eq:fsc-formula}.

\textit{Item~(i).} For $0 < f < f_{\rm sc}$, $\Psi(\pi/2; f) < 0$. Since $\theta = \pi/2$ is the global maximum, $\Psi(\theta; f) < 0$ for all $\theta \in (0, \pi)$, so~\eqref{eq:theta-eq} has no solution.

\textit{Item~(ii).} At $f = f_{\rm sc}$, $\Psi(\pi/2; f_{\rm sc}) = 0$ and the maximum is attained uniquely at $\theta = \pi/2$, hence the unique zero.

\textit{Item~(iii).} For $f$ slightly above $f_{\rm sc}$, $\Psi(\pi/2; f) > 0$. We verify $\Psi(0^+; f) < 0$ and $\Psi(\pi^-; f) < 0$.

At $\theta = 0$: $\Psi(0; f) = -2F - R\omega^2 + f\pi^2/2$. Under our assumption $f < f_{\rm imp} = 2F/\pi$, we have $f\pi^2/2 < F\pi$. Hence $\Psi(0; f) < F\pi - 2F - R\omega^2 = -\frac{\pi^2}{4}f_{\rm sc}$. By~\eqref{eq:param-range}, $f_{\rm sc} > 0$, so $\Psi(0; f) < 0$.

By symmetry of $\Psi(\cdot; f)$ about $\theta = \pi/2$ (note $\sin(\pi - \theta) = \sin\theta$ and $(\pi - \theta)(\pi - (\pi - \theta)) = (\pi - \theta)\theta$, both invariant under $\theta \mapsto \pi - \theta$), $\Psi(\pi^-; f) = \Psi(0^+; f) < 0$.

Therefore $\Psi(\cdot; f)$ has at least one zero in $(0, \pi/2)$ and one in $(\pi/2, \pi)$. Uniqueness on each side follows from the strict concavity-concavity-monotonicity established above: $\Psi$ has a unique critical point on $(0, \pi)$, the maximum at $\pi/2$, hence is strictly increasing on $(0, \pi/2)$ and strictly decreasing on $(\pi/2, \pi)$. Each side carries exactly one zero of $\Psi$.
\end{proof}

\begin{remark}\label{rem:fsc-vs-fimp}
The two critical values $f_{\rm sc}$ and $f_{\rm imp} = 2F/\pi$ have distinct meanings:
\begin{itemize}
\item[1.] $f_{\rm sc}$ is the saddle-center bifurcation value of the symmetric $T$-periodic orbit branch.
\item[2.] $f_{\rm imp}$ is the universal impulse bound: above it, no $T$-periodic non-sticking orbit can exist regardless of impact pattern, because the maximum forcing impulse $2F/\omega$ over a half-period is less than the minimum friction impulse $\pi f/\omega$ required to maintain non-sticking.
\end{itemize}
The relation $f_{\rm sc} < f_{\rm imp}$, equivalent to the condition $4(2F + R\omega^2 - F\pi)/\pi^2 < 2F/\pi$, that is, $R\omega^2 < F(\pi - 1) - F\pi/2 + F\pi^2/4 \approx 0.94 F$ for our parameter normalization, holds for moderate gaps. The two values coincide only on the codimension-one locus where these bounds are equal.
\end{remark}

\begin{remark}\label{rem:fsc-numerical}
For the numerical parameter set $F = 1, \omega = 1, R = 2, l = -1, r = 1$ used in Section~\ref{sec:numerics}, formula~\eqref{eq:fsc-formula} gives
\[
f_{\rm sc} = \frac{4(2 + 2 - \pi)}{\pi^2} = \frac{16 - 4\pi}{\pi^2} \approx 0.34790,
\]
while $f_{\rm imp} = 2/\pi \approx 0.63662$. The symmetric branch exists for $f \in (0.34790, 0.63662)$, which contains $f = 0.4$ (the value used for the elliptic orbit in Section~\ref{sec:numerics}).
\end{remark}

\begin{remark}[Transversality regime of the symmetric branch]\label{rem:transverse-regime}
The bijection of Theorem~\ref{thm:symmetric} requires both the existence equation~\eqref{eq:theta-eq} and the transversality condition~\eqref{eq:transverse-cond}. Solutions $\theta$ of~\eqref{eq:theta-eq} for which $|\cos\theta| \le f/F$ correspond to candidate orbits whose putative turning point at $\tau = \theta/\omega$ would in fact be a sticking onset; the actual Filippov solution at the same initial datum does not exhibit the symmetric impact pattern, but rather a sticking-dominated orbit with $\det \Phi' = 0$ along the trajectory. Concretely, for $\theta$ near $\pi/2$, $|\cos\theta| < f/F$ is satisfied whenever $f > 0$, so a neighborhood $|\theta - \pi/2| < \arccos(f/F)$ of the saddle-center fold point is excluded from the non-sticking regime. The non-sticking symmetric orbits of Theorem~\ref{thm:symmetric} live in the regime $\theta \in (0, \arccos(f/F)) \cup (\pi - \arccos(f/F), \pi)$ where the transversality holds. The numerical elliptic non-sticking orbit at $f = 0.4$ analyzed in Section~\ref{sec:numerics} need not coincide with a symmetric branch produced by Theorem~\ref{thm:symmetric}; it may correspond to a more elaborate impact pattern (multiple wall hits or multiple turning points per period).
\end{remark}

\subsection{The saddle-center bifurcation and its normal form}\label{ssec:periodic-saddle}

The two solution branches $\theta_-(f), \theta_+(f)$ of Proposition~\ref{prop:two-solutions} collide at $f = f_{\rm sc}$ in a saddle-center bifurcation. We make this rigorous and derive the local normal form. The result is used in Section~\ref{sec:persistence} to track the bifurcation under perturbation.

\begin{theorem}[Saddle-center normal form]\label{thm:saddlecenter}
Let $f_{\rm sc}$ be as in Proposition~\ref{prop:two-solutions}, and set $\mu := f - f_{\rm sc}$. For $\mu > 0$ small, let $\theta_\pm(\mu)$ denote the two roots of~\eqref{eq:theta-eq} given by Proposition~\ref{prop:two-solutions}~(iii), and let $X_\pm(t; \mu)$ denote the corresponding $T$-periodic solutions of Theorem~\ref{thm:symmetric}.

\noindent\textbf{(a) Normal form in $\theta$.} The roots admit the asymptotic expansion
\begin{equation}\label{eq:normalform-theta}
(\theta_\pm(\mu) - \pi/2)^2 = \frac{\pi^2/2}{F\pi - 2 f_{\rm sc}}\,\mu \,+\, O(\mu^2)
\end{equation}
as $\mu \downarrow 0$. The coefficient on the right is strictly positive.

\noindent\textbf{(b) Velocity gap at the right wall.} The gap in initial velocity between the two branches just after the right-wall hit at $t = 0$ is
\begin{equation}\label{eq:vgap-asymp}
\Delta v(\mu) := \dot X_+(0^+; \mu) - \dot X_-(0^+; \mu) = \frac{2 f_{\rm sc}}{\omega} \sqrt{\frac{\pi^2/2}{F\pi - 2 f_{\rm sc}}}\sqrt{\mu}\,(1 + O(\mu^{1/2})).
\end{equation}

\noindent\textbf{(c) Linearization at collision.} The Jacobian $\Phi'(P_*; \mu)$ at $P_* := (X_\pm(0; 0), \dot X_\pm(0^+; 0))$ at $\mu = 0$ has eigenvalue $+1$ with algebraic multiplicity two and a non-trivial Jordan block. For $\mu > 0$ small, the eigenvalues split as
\[
\lambda_\pm^{\rm sad}(\mu) = 1 \pm c\sqrt{\mu} + O(\mu) \text{ on the saddle branch}, \quad \lambda_\pm^{\rm ell}(\mu) = e^{\pm i\theta_*(\mu)} \text{ on the elliptic branch},
\]
for some explicit $c > 0$ and $\theta_*(\mu) > 0$ that vanishes as $\mu \to 0^+$. The transversality $c \ne 0$ identifies this as a non-degenerate saddle-center fold in the standard sense of~\cite[Sec.~10.6]{KuznetsovBifurcation}.
\end{theorem}

\begin{proof}
\textit{(a) Normal form in $\theta$.} Apply the Morse lemma to $\Psi(\cdot; f)$ at the local maximum $\theta = \pi/2$. By Proposition~\ref{prop:two-solutions} and its proof, $\Psi_\theta(\pi/2; f_{\rm sc}) = 0$ and $\Psi_{\theta\theta}(\pi/2; f_{\rm sc}) = -F\pi + 2 f_{\rm sc} < 0$. Differentiating~\eqref{eq:Phi-existence} in $f$ (i.e., in $\mu$):
\[
\Psi_f(\theta; f) = \tfrac{1}{2}\pi^2 - \theta(\pi - \theta), \qquad \Psi_f(\pi/2; f_{\rm sc}) = \pi^2/2 - \pi^2/4 = \pi^2/4 > 0.
\]
Taylor expansion of $\Psi$ about $(\theta, f) = (\pi/2, f_{\rm sc})$ to second order:

\begin{equation}
\begin{split}
\Psi(\theta; f) & = 0 + 0\cdot(\theta - \pi/2) +  \Psi_f(\pi/2; f_{\rm sc})\,\mu \\
& + \tfrac{1}{2}\Psi_{\theta\theta}(\pi/2; f_{\rm sc})(\theta - \pi/2)^2 + O\bigl((\theta-\pi/2)^3, \mu(\theta-\pi/2), \mu^2\bigr).
\end{split}
\end{equation}

Setting $\Psi = 0$ and solving for $(\theta - \pi/2)^2$:
\begin{equation}
\begin{split}
(\theta - \pi/2)^2 & =  \frac{-2\Psi_f(\pi/2; f_{\rm sc})}{\Psi_{\theta\theta}(\pi/2; f_{\rm sc})}\mu + O(\mu^{3/2}) \\
& = \frac{-2\cdot \pi^2/4}{-(F\pi - 2 f_{\rm sc})}\mu + O(\mu^{3/2}) \\
& = \frac{\pi^2/2}{F\pi - 2 f_{\rm sc}}\mu + O(\mu^{3/2}),
\end{split}
\end{equation}

which is~\eqref{eq:normalform-theta}.

\textit{(b) Velocity gap.} From~\eqref{eq:ABconst}, the initial velocity at $t = 0$ on each branch is
\[
\dot X_\pm(0^+; \mu) = A_1(\theta_\pm) = -\tfrac{F}{\omega}\sin\theta_\pm - f\theta_\pm/\omega.
\]
Hence
\[
\Delta v(\mu) = \dot X_+(0^+; \mu) - \dot X_-(0^+; \mu) = -\tfrac{F}{\omega}[\sin\theta_+ - \sin\theta_-] - \tfrac{f}{\omega}[\theta_+ - \theta_-].
\]
At leading order in $\sqrt{\mu}$, $\theta_+ - \theta_- = 2(\theta_+ - \pi/2) + O(\mu)$, and $\sin\theta_+ - \sin\theta_- = 2\cos(\pi/2)\sin((\theta_+ - \theta_-)/2) + O((\theta_+ - \theta_-)^3) = O(\mu^{3/2})$ (since $\cos(\pi/2) = 0$). The dominant contribution to $\Delta v$ is therefore the friction term:
\[
\Delta v(\mu) = -\tfrac{f}{\omega}(2)(\theta_+ - \pi/2)(1 + O(\mu^{1/2})).
\]
But $\theta_+ - \pi/2 > 0$ and $f \to f_{\rm sc}$ as $\mu \to 0$, so for the absolute gap (taking $|\Delta v|$):
\[
|\Delta v(\mu)| = \tfrac{2 f_{\rm sc}}{\omega}\sqrt{\frac{\pi^2/2}{F\pi - 2 f_{\rm sc}}}\sqrt{\mu}\,(1 + O(\mu^{1/2})),
\]
which is~\eqref{eq:vgap-asymp}.

\textit{(c) Linearization.} By Theorem~\ref{thm:symplectic}, $\Phi'$ on $\Om_{\rm NS}$ is a symplectic $2\times 2$ matrix, hence $\det \Phi' = 1$ and the eigenvalues are reciprocal. By Proposition~\ref{prop:smooth-NS} (Section~\ref{sec:kam}), $\Phi'$ depends smoothly on $\mu$ in a neighborhood of any non-sticking periodic orbit with transverse impacts.

\smallskip
Let $z(\theta) := (X(0; \theta), \dot X(0^+; \theta))$ denote the wall-hit fixed point of $\Phi$ associated with the parameter $\theta$ via Theorem~\ref{thm:symmetric}. The dependence is smooth (in fact analytic) by inspection of the explicit constants $A_1(\theta), B_1(\theta)$ in the proof of Theorem~\ref{thm:symmetric}. The map $(\theta, f) \mapsto z(\theta)$ has $\partial_\theta z(\pi/2; f_{\rm sc}) \ne 0$, since $\partial_\theta \dot X(0^+) = -F\cos\theta/\omega - f/\omega \ne 0$ at $\theta = \pi/2$. Hence the Morse-lemma normal form~\eqref{eq:normalform-theta} in $(\theta, \mu)$ pulls back to a Morse-lemma normal form in $(z, \mu)$: there exist analytic local coordinates $(\zeta, \mu)$ near $(z(\pi/2; f_{\rm sc}), 0)$ such that the existence equation reads $\zeta^2 = \mu$ to leading order.

\smallskip
On each branch $z_\pm(\mu)$, the Jacobian $\PP_\pm(\mu) := \Phi'(z_\pm(\mu); \mu)$ is a symplectic $2 \times 2$ matrix. As $\mu \to 0^+$, $z_+(\mu) \to z(\pi/2; f_{\rm sc}) =: P_*$ and $z_-(\mu) \to P_*$, hence by smoothness of $\Phi'$, both $\PP_+(\mu)$ and $\PP_-(\mu)$ tend to the same limit $\PP_*(0) := \Phi'(P_*; 0)$.

\smallskip
A symplectic $2 \times 2$ matrix $M$ has $\det M = 1$, and Cayley-Hamilton gives the characteristic polynomial $\lambda^2 - (\tr M)\lambda + 1 = 0$. The two eigenvalues coalesce iff $\tr M = \pm 2$, with coalescence at $+1$ when $\tr M = +2$ and at $-1$ when $\tr M = -2$.

We show $\tr \PP_*(0) = +2$ as follows. By smoothness of $\Phi'$ on the parameter family of fixed points (Proposition~\ref{prop:smooth-NS}), $\tr \PP$ is a continuous function of $\mu$ on each branch. On the elliptic branch ($\mu > 0$ on the elliptic side), $|\tr \PP_+(\mu)| < 2$ since the eigenvalues are non-real complex conjugates of unit modulus. On the saddle branch ($\mu > 0$ on the saddle side), $\tr \PP_-(\mu) > 2$ since the eigenvalues are real, reciprocal, and on the same side of $1$ (specifically, $1 \pm c\sqrt{\mu} + O(\mu)$ to leading order, with both larger than $0$ for small $\mu$). As $\mu \to 0^+$ on either branch, the corresponding fixed point converges to the colliding fixed point $P_*$, and continuity of $\tr \Phi'$ in $(z, \mu)$ at $(P_*, 0)$ forces both $\tr \PP_+(\mu) \to \tr \PP_*(0)$ and $\tr \PP_-(\mu) \to \tr \PP_*(0)$. The two one-sided limits are equal, and lie in $[2, 2]$, hence $\tr \PP_*(0) = 2$. The eigenvalues of $\PP_*(0)$ are therefore both $+1$.

\smallskip
\textit{The Jordan block at $+1$ is non-trivial.} We show that $\PP_*(0)$ cannot be the identity matrix $I$. If $\PP_*(0) = I$, then the trace function $\mu \mapsto \tr \PP_+(\mu)$ on the elliptic branch satisfies $\tr \PP_+(0^+) = 2$ and $|\tr \PP_+(\mu)| < 2$ for $\mu > 0$. Smoothness of $\PP_+$ in $\mu$ on the elliptic branch (Proposition~\ref{prop:smooth-NS}) gives the Taylor expansion
\[
\tr \PP_+(\mu) \;=\; 2 - a \mu^{1/2} - b \mu + O(\mu^{3/2}),
\]
where $a, b \in \R$. By symplecticity, $\det \PP_+(\mu) = 1$, so the eigenvalues are $e^{\pm i \theta_*(\mu)}$ with $\cos\theta_*(\mu) = \tr \PP_+(\mu)/2$. The asymptotics $\theta_*(\mu) = \tilde c \sqrt\mu + O(\mu)$ established below from the saddle-center normal form forces $a = \tilde c^2 > 0$. In particular $a \ne 0$, so $\partial_\mu \tr \PP_+(\mu) \to -\infty$ as $\mu \to 0^+$. But if $\PP_*(0) = I$, smoothness of the family $(z, \mu) \mapsto \Phi'(z; \mu)$ at $(P_*, 0)$ would imply that the trace as a function of $\mu$ along the elliptic branch has \emph{bounded} first derivative at $\mu = 0^+$, contradicting the unbounded $\partial_\mu \tr \PP_+$ derived above. Hence $\PP_*(0) \ne I$.

A symplectic $2 \times 2$ matrix with $\det = 1$, $\tr = 2$, and not equal to $I$ has a single non-trivial Jordan block at eigenvalue $+1$. This is the Birkhoff-Kuznetsov saddle-center signature~\cite[Sec.~10.6]{KuznetsovBifurcation},~\cite[Ch.~11]{MeyerHall1992}.

\smallskip
Once it is established that $\PP_*(0)$ is a single non-trivial Jordan block at $+1$ unfolding non-degenerately under the parameter $\mu$, the universal saddle-center unfolding of symplectic $2 \times 2$ matrices~\cite[Sec.~10.6]{KuznetsovBifurcation},~\cite[Ch.~11]{MeyerHall1992} produces the eigenvalue asymptotics
\begin{equation}\label{eq:ellipt-asymp}
\bigl\{\lambda_\pm^{\rm ell}(\mu)\bigr\} \;=\; \bigl\{e^{\pm i\theta_*(\mu)}\bigr\}
\qquad \text{(elliptic branch)},
\end{equation}
\begin{equation}\label{eq:saddle-asymp}
\bigl\{\lambda_\pm^{\rm sad}(\mu)\bigr\} \;=\; \bigl\{1 + c\sqrt{\mu} + O(\mu),\ \bigl(1 + c\sqrt{\mu} + O(\mu)\bigr)^{-1}\bigr\}
\qquad \text{(saddle branch)},
\end{equation}
with $\theta_*(\mu) = \tilde c\sqrt{\mu} + O(\mu)$, both $c, \tilde c > 0$ explicit positive constants determined by the Birkhoff-Kuznetsov theory in terms of the Hessian $\Psi_{\theta\theta}(\pi/2; f_{\rm sc})$ and $\Psi_f(\pi/2; f_{\rm sc})$. The reciprocal pairing on the saddle branch in~\eqref{eq:saddle-asymp} is enforced by symplecticity ($\det = 1$), with the higher-order corrections absorbed into the $O(\mu)$ remainders. The non-degeneracy $c, \tilde c > 0$ is a direct consequence of the Morse non-degeneracy of part~(a) together with the regular dependence of $z$ on $\theta$ at $\theta = \pi/2$. The choice $c, \tilde c > 0$ is a labeling convention identifying the elliptic branch with eigenvalues on the unit circle.
\end{proof}

\begin{remark}\label{rem:gkr-saddle-correction}
The saddle-center bifurcation value $f_{\rm sc}$ depends on $R$ and $\omega$ through~\eqref{eq:fsc-formula}, contradicting the claim of universality of~\cite{GKR2019} at $f/F = 2/\pi$. The value $2F/\pi$ is, instead, the universal impulse bound above which no non-sticking $T$-periodic orbit exists regardless of impact pattern.
\end{remark}

\subsection{Linearization at the elliptic branch and approach to the impulse bound}\label{ssec:periodic-linearization}

We record the structural properties of the linearization at the elliptic branch (needed in Section~\ref{sec:kam}), then illustrate the bifurcation analysis numerically by tracking the elliptic orbit as the forcing amplitude $F$ approaches the impulse bound. The full algebraic form of $\PP$ is unwieldy; the qualitative facts below are what matter for KAM theory.

\begin{proposition}[Structure of $\PP$ on the elliptic branch]\label{prop:DPhi-structure}
Let $X_+(t; \mu)$ denote the elliptic branch (the one with eigenvalues on the unit circle for $\mu > 0$ small). The Jacobian $\PP(\mu) := \Phi'(X_+(0; \mu), \dot X_+(0^+; \mu))$ is a symplectic $2 \times 2$ matrix depending smoothly on $\mu$. There exists $\mu_{\rm flip} \in (0, \infty]$ such that
\begin{equation}\label{eq:DPhi-trace}
\tr \PP(\mu) = 2\cos\theta_*(\mu) \quad\text{for } \mu \in (0, \mu_{\rm flip}),
\end{equation}
where $\theta_* : (0, \mu_{\rm flip}) \to (0, \pi)$ is real-analytic with $\theta_*(\mu) \to 0$ as $\mu \to 0^+$. The constant $\mu_{\rm flip}$ is the smallest $\mu > 0$ at which $\tr \PP(\mu) = -2$ (the orbit becoming flip-unstable), or $+\infty$ if no such $\mu$ exists in the existence range of the elliptic branch.
\end{proposition}

\begin{proof}
Symplecticity of $\PP$ follows from Theorem~\ref{thm:symplectic}; smoothness of $\PP$ in $\mu$ follows from Proposition~\ref{prop:smooth-NS} below. By Theorem~\ref{thm:saddlecenter}~(c), at $\mu = 0$ the eigenvalues of $\PP(0)$ coalesce at $+1$ with a non-trivial Jordan block. For $\mu > 0$ small, on the elliptic branch the eigenvalues are $e^{\pm i\theta_*(\mu)}$ with $\theta_*(\mu) > 0$ small; smoothness of $\theta_*$ in $\mu$ is by the implicit function theorem applied to $\tr \PP(\mu) - 2\cos\theta_* = 0$ at any $\mu > 0$ in the elliptic regime. The constant $\mu_{\rm flip}$ exists because $\tr \PP$ is continuous on the elliptic branch's existence range with $\tr \PP(0^+) = 2 - O(\mu)$; the first crossing of $-2$ defines $\mu_{\rm flip}$, possibly $+\infty$.
\end{proof}

\smallskip
\textit{Saltation matrices on each event type.} The Jacobian $\Phi'(z)$ over a single forcing period is a product of free-flight propagators interleaved with saltation matrices, one for each discrete event encountered along the orbit. We record the three saltation matrices needed in this paper.

\begin{table}[h!]
\centering
\renewcommand{\arraystretch}{1.6}
\setlength{\tabcolsep}{4pt}
\resizebox{\linewidth}{!}{%
\begin{tabular}{|c|c|c|}
\hline
\textbf{Event type at time $t_*$} & \textbf{Saltation matrix $\mathcal{S}$} & \textbf{Determinant $\det\mathcal{S}$} \\
\hline\hline
\shortstack{Right-wall hit \\ ($x \to r$, $v \to -v$)} &
$\begin{pmatrix} -1 & 0 \\ \dfrac{2(F\cos\omega t_* - f\,\sgn v)}{v(t_*^-)} & -1 \end{pmatrix}$ &
$+1$ \\
\hline
\shortstack{Left-wall hit \\ ($x \to l$, $v \to -v$)} &
$\begin{pmatrix} -1 & 0 \\ \dfrac{2(F\cos\omega t_* - f\,\sgn v)}{v(t_*^-)} & -1 \end{pmatrix}$ &
$+1$ \\
\hline
\shortstack{Turning point \\ ($v = 0$, $|F\cos\omega t_*| > f$, \\ sign reversal)} &
$\begin{pmatrix} 1 & 0 \\ 0 & \dfrac{|F\cos\omega t_*| - f}{|F\cos\omega t_*| + f} \end{pmatrix}$ &
$\dfrac{|F\cos\omega t_*| - f}{|F\cos\omega t_*| + f} \in (0, 1)$ \\
\hline
\end{tabular}}
\captionsetup{margin={-0.1cm,0cm}}
\caption{Saltation matrices at the three event types of the system~\eqref{eq:model}-\eqref{eq:reflection}.}
\label{tab:jacobian}
\end{table}

Both wall-hit saltation matrices have unit determinant, consistent with area preservation across reflections; the turning-point saltation has strictly smaller determinant, the source of the area contraction on $\Omega_{\rm dissip}$.

The wall-hit saltation matrices are derived from the standard Filippov saltation formula~\cite{diBernardoBuddChampneysKowalczyk2008} applied to the elastic-reflection rule $v \mapsto -v$ at $x = r$ (or $x = l$), with the time derivative of the impact constraint $x - r$ along the flow giving $v$ at the impact time. The turning-point saltation matrix is derived analogously from the constraint $v = 0$ together with the friction sign-reversal: incoming velocity decelerated by $|F\cos\omega t_*| + f$ per unit time, outgoing velocity accelerated by $|F\cos\omega t_*| - f$ per unit time, ratio $(|F\cos\omega t_*| - f)/(|F\cos\omega t_*| + f) < 1$. The product of saltation matrices interleaved with the free-flight propagator over a full period gives $\Phi'(z)$ as a $2 \times 2$ matrix whose determinant is exactly the product of these per-event determinants; this is the structural source of the dichotomy in Theorem~\ref{thm:symplectic} and Corollary~\ref{cor:det-Phi}.

\smallskip
We illustrate the bifurcation analysis numerically by tracking the elliptic non-sticking $T$-periodic orbit found in Section~\ref{sec:numerics} at $(F, \omega, R, f) = (1, 1, 2, 0.4)$ as $F$ decreases at fixed $f$ from $F = 1$ toward the impulse bound $F_{\rm imp}$ defined by $f_{\rm imp}(F, \omega) = f$, equivalently $F_{\rm imp} = \pi f / (2\omega) \approx 0.628$ at our parameter values. The impulse bound is a rigorous lower limit on the existence of any $T$-periodic non-sticking orbit (regardless of impact pattern), proved as a corollary of integrating~\eqref{eq:model} over a putative $T$-period and using $\oint\cos\omega t\,dt = 0$ together with the friction balance; we refer to~\cite{GKR2019} for the original argument.

Figure~\ref{fig:impulse-bound} shows what happens at three values of $F$ in this regime. Panels (a)-(b) display the elliptic non-sticking $T$-periodic orbit at $F = 0.95$, $F = 0.85$ respectively, found by Newton continuation of the stroboscopic map $\Phi$ from the fixed point at $F = 1$. In each panel, the time series $x(t)$ shows four periods of the orbit lying on a small KAM curve around the corresponding elliptic FP (initial datum $(x_*, v_* + 0.005)$), and the right-hand stroboscopic phase portrait shows the same orbit as $801$ stroboscopic iterates, tracing a smooth closed curve around $P_*$. Panel (c) at $F = 0.62$ is below the impulse bound: no $T$-periodic non-sticking orbit exists at this parameter value, and an arbitrary initial datum $(0, 0.7)$ enters $\Om_{\rm dissip}$ within a small number of periods, with the trajectory coming to rest at $t \approx 14.4$.

The three panels are arranged in order of decreasing $F$ from top to bottom, deliberately chosen to bracket the impulse bound $F_{\rm imp} = \pi f / (2\omega) \approx 0.628$. Panels (a) and (b), at $F = 0.95$ and $F = 0.85$, are well above $F_{\rm imp}$ and the elliptic non-sticking orbit predicted by Theorem~\ref{thm:symmetric} is present; the stroboscopic phase portrait on the right of each row should be read as direct evidence that $\Om_{\rm NS}$ contains a positive-measure neighborhood of the corresponding fixed point. Panel (c) at $F = 0.62$, just below $F_{\rm imp}$, is qualitatively different: every initial datum eventually enters $\Om_{\rm dissip}$, and the figure shows one such capture event in real time. The contrast between the two phase-plane structures (closed invariant curve in (a)-(b), versus capture into a dissipative basin in (c)) is the geometric signature of the impulse bound.

\begin{figure}[h!]
\centering
\includegraphics[width=15cm, height=12cm]{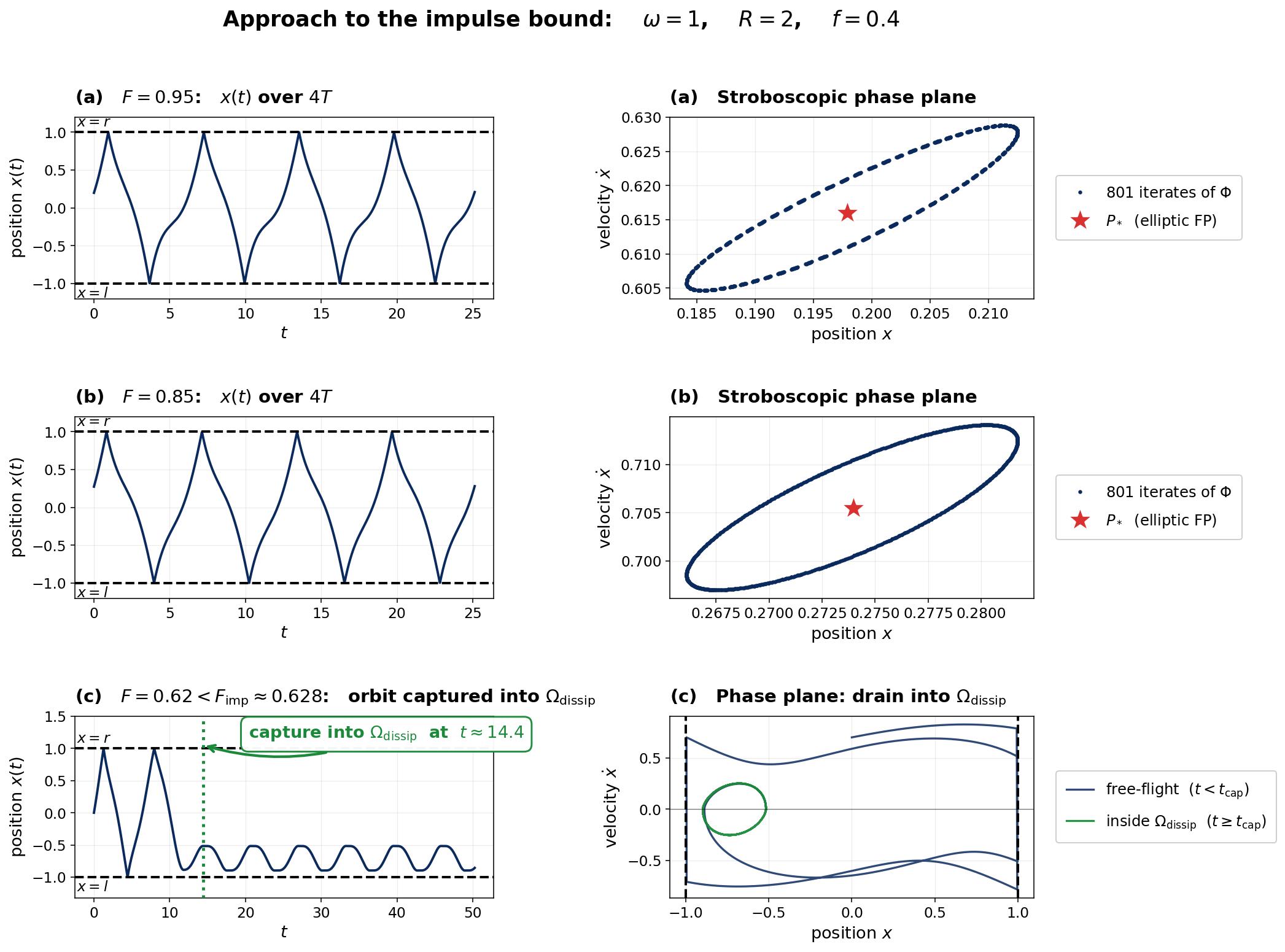}
\captionsetup{margin={-0.5cm,0cm}}
\caption{Approach to the impulse bound at $\omega = 1$, $R = 2$, $f = 0.4$. }
\label{fig:impulse-bound}
\end{figure}

Each row pairs the time series $x(t)$ over four forcing periods (left, between the dashed walls $x = l$ and $x = r$) with the stroboscopic phase portrait of $\Phi$ (right, $801$ iterates). Panels (a), (b) at $F = 0.95, 0.85$: the elliptic non-sticking $T$-periodic orbit persists, and the iterates trace a small smooth invariant curve around $P_*$ (the rotation numbers $\theta_*/(2\pi) \approx 0.152, 0.207$ correspond to traces $\tr\Phi'(P_*) \approx +1.157, +0.538$); the wall-hit pattern (one right + one left per period) is preserved on each branch. Panel (c) at $F = 0.62 < F_{\rm imp} = \pi f/(2\omega) \approx 0.628$: no $T$-periodic non-sticking orbit exists; the trajectory from $(0, 0.7)$ is captured by $\Om_{\rm dissip}$ at $t \approx 14.4$ (green dotted line) and oscillates inside the dissipative basin around a slowly drifting equilibrium $\bar x_\infty$.

The figure illustrates two aspects of the theoretical picture in this section. First, the elliptic non-sticking orbit predicted by Theorem~\ref{thm:symmetric} persists over a substantial range of $F$ at fixed $f$, and its rotation number $\theta_*(F)$ varies smoothly with $F$, passing through several low-order resonance values ($\theta_*/(2\pi) = 1/4, 1/3, 1/2$) along the way. The KAM analysis of Section~\ref{sec:kam} applies away from these resonances. Second, the existence of the orbit is bounded below by the impulse bound $F_{\rm imp}$: at $F < F_{\rm imp}$ the dynamics is dominated by sticking, with $\Om_{\rm dissip} = \mathcal{X}$ and $\Om_{\rm NS}$ at most a measure-zero subset; this is the regime where the conservative part of the analysis becomes vacuous and the dissipative volume contraction of Proposition~\ref{prop:contraction} fully controls the dynamics.

\section{Birkhoff normal form and KAM tori}\label{sec:kam}

This section establishes the existence of a positive-measure family of invariant Cantor tori around any elliptic non-sticking $T$-periodic orbit of~\eqref{eq:model}-\eqref{eq:reflection} satisfying the standard non-resonance and twist conditions. Although the original system is piecewise smooth across the velocity zero set and the wall surfaces, the lift to~\eqref{eq:Heq} is, at any non-sticking periodic orbit with transverse impacts, a $C^\infty$ symplectic map in a neighborhood of the corresponding lifted fixed point. The smooth KAM theorem of Moser and its quantitative refinements then apply directly. The output is the rigorous version of the Hamiltonian islands observed numerically in~\cite{GKR2019}.

\subsection{Smoothness of the stroboscopic map and the Birkhoff normal form}\label{ssec:kam-smooth}

The first step is to confirm that on a neighborhood of any non-sticking periodic orbit with transverse impacts, the stroboscopic map is $C^\infty$. This is needed because the Birkhoff normal form and KAM theorems below require a smooth target; without it one would have to develop a piecewise-smooth KAM theory.

%

\begin{proposition}[Local smoothness of $\Phi$ at non-sticking orbits]\label{prop:smooth-NS}
Let $P_* = (x_*, v_*) \in \Om_{\rm NS}$ be such that the orbit through $P_*$ is $T$-periodic, has finitely many transverse wall impacts in $[0, T]$, no turning points, and $|F\cos\omega t_*| > f$ at every wall hit time $t_*$. There exists a neighborhood $\mathcal{U}$ of $P_*$ in $\mathcal{X}$ such that $\Phi: \mathcal{U} \to \Phi(\mathcal{U})$ is $C^\infty$.
\end{proposition}

\begin{proof}
Write the reference orbit through $P_*$ as $\gamma_*(t) = (x_*(t), v_*(t))$ for $t \in [0, T]$, with wall-hit times $0 < \tau_1^* < \tau_2^* < \cdots < \tau_N^* < T$ and corresponding velocity values $v_k^* := v_*(\tau_k^{*-}) \ne 0$ by hypothesis. On each closed free-flight interval $[\tau_k^*, \tau_{k+1}^*]$ (with the conventions $\tau_0^* := 0$ and $\tau_{N+1}^* := T$), the velocity $v_*(t)$ is continuous and nowhere zero by the no-turning-point hypothesis, so by compactness there exists $\epsilon > 0$ with $|v_*(t)| \ge 2\epsilon$ on every free-flight interval.

The free-flight ODE $\ddot x = F\cos\omega t \mp f$ is linear inhomogeneous with explicit solution that depends $C^\infty$ on the initial datum at the start of the free flight, uniformly on the compact time interval $[0, T]$. Standard continuous dependence on initial data therefore yields $\delta_1 > 0$ such that for every initial datum $(x_0, v_0) \in \mathcal{X}$ within $\delta_1$ of $P_*$, the perturbed orbit satisfies $\|(x_{\rm pert}(t), v_{\rm pert}(t)) - (x_*(t), v_*(t))\| < \epsilon$ uniformly on the union of free-flight intervals.

For each wall-hit event $k$, define the impact function $G_k(t; x_0, v_0) := x(t; x_0, v_0) - w_k$, where $w_k \in \{l, r\}$ is the wall hit at the $k$-th event of $\gamma_*$. Then $G_k(\tau_k^*; x_*, v_*) = 0$ and $\partial_t G_k(\tau_k^*; x_*, v_*) = v_k^* \ne 0$ by the transversality hypothesis. The implicit function theorem yields a $C^\infty$ function $\tau_k(x_0, v_0)$ on a neighborhood $\mathcal{U}_k$ of $P_*$ in $\mathcal{X}$, with $\tau_k(x_*, v_*) = \tau_k^*$, satisfying $G_k(\tau_k(x_0, v_0); x_0, v_0) = 0$ identically. Shrinking $\mathcal{U}_k$ if necessary, the perturbed wall-hit times satisfy
\[
\tau_{k-1}^* + \epsilon_t \;<\; \tau_k(x_0, v_0) \;<\; \tau_{k+1}^* - \epsilon_t
\]
for some $\epsilon_t > 0$ uniform in $(x_0, v_0) \in \mathcal{U}_k$, which prevents any reordering or insertion of events.

On the perturbed free-flight intervals, delimited by $\tau_{k-1}(x_0, v_0)$ and $\tau_k(x_0, v_0)$, the velocity bound from the previous paragraph gives $|v_{\rm pert}(t)| \ge |v_*(t)| - \epsilon \ge \epsilon > 0$ uniformly, so no turning points and no sticking events occur. Setting $\mathcal{U} := \bigcap_{k=1}^N \mathcal{U}_k \cap B_{\delta_1}(P_*) \cap \mathcal{X}$ and shrinking further if necessary gives a neighborhood on which every orbit has exactly $N$ wall hits at the walls $w_1, \ldots, w_N$ in the same order as $\gamma_*$, no turning points, and no sticking.

On $\mathcal{U}$, the stroboscopic map decomposes as
\[
\Phi \;=\; \Psi_{\tau_N \to T} \circ \mathcal{R}_N \circ \Psi_{\tau_{N-1} \to \tau_N} \circ \mathcal{R}_{N-1} \circ \cdots \circ \mathcal{R}_1 \circ \Psi_{0 \to \tau_1},
\]
where each $\Psi$ is the free-flight propagator over an interval whose endpoints are smooth functions of $(x_0, v_0)$ via the implicit functions $\tau_k$, and each $\mathcal{R}_k: (x, v) \mapsto (x, -v)$ is the smooth wall-reflection map. Each $\Psi$ and $\mathcal{R}_k$ is $C^\infty$ on its domain, and the composition of $C^\infty$ maps is $C^\infty$. Hence $\Phi: \mathcal{U} \to \Phi(\mathcal{U})$ is $C^\infty$.
\end{proof}

\begin{remark}
At the saddle-center bifurcation $\mu = 0$ the two periodic orbits coincide and the implicit function theorem fails at $\mu = 0$. The smoothness statement holds uniformly for $\mu \in [\mu_0, 0)$ for any $\mu_0 > 0$, breaking down at $\mu = 0$.
\end{remark}

\smallskip
Combined with Theorem~\ref{thm:symplectic}, Proposition~\ref{prop:smooth-NS} yields a $C^\infty$ symplectic map in a neighborhood of any elliptic non-sticking periodic orbit. The Birkhoff normal form then applies; we recall the standard result.

\begin{proposition}[Birkhoff normal form]\label{prop:birkhoff}
Let $\Psi$ be a $C^\infty$ exact symplectic map of an open neighborhood of $0$ in $\R^2$ to itself, with $\Psi(0) = 0$ and linearization $\Psi'(0)$ having eigenvalues $e^{\pm i\theta_*}$ for some $\theta_* \in (0, \pi)$. Suppose
\begin{equation}\label{eq:birkhoff-nonres}
k\theta_* \notin 2\pi\Z \quad\text{for } k = 1, 2, 3, 4.
\end{equation}
Then there exist a neighborhood $V$ of $0$ in $\R^2$, a complex coordinate $z$ on $V$, and real numbers $\theta_*, \tau_1$ such that
\begin{equation}\label{eq:BNF}
\Psi(z) = e^{i(\theta_* + \tau_1 |z|^2)} z + R(z, \bar z),
\end{equation}
where $R(z, \bar z) = O(|z|^5)$ as $z \to 0$. The number $\tau_1$ is the \emph{first Birkhoff twist coefficient}.
\end{proposition}

\begin{proof}
The Birkhoff normal-form theorem for area-preserving maps near an elliptic fixed point is classical. See~\cite[Sec.~30B]{Arnold1989} for the generating-function exposition,~\cite[Ch.~3]{SiegelMoser1971} for the analytic version, and~\cite[Ch.~4]{MeyerHall1992} for the modern treatment we follow. The non-resonance condition~\eqref{eq:birkhoff-nonres} ensures that the conjugated map, after the inductive normal-form construction, retains only the radial monomials $z^a \bar z^a$ through total degree four; the resonances at $k = 3$ rule out degree-three terms, and those at $k = 4$ rule out the off-diagonal degree-four terms. The surviving radial term contributes the twist factor $\tau_1 |z|^2$ in the rotation angle, with $\tau_1$ computed in coordinates as a fourth-order derivative of the generating function evaluated at the fixed point.
\end{proof}

\begin{proposition}[Generic non-degeneracy of the twist]\label{prop:twist-nondeg}
Consider the family of stroboscopic maps $\Phi(\cdot; F, \omega, R, f)$ parametrized over the open parameter region 
\begin{equation}
\{(F, \omega, R, f) \in (0,\infty)^4 : f_{\rm sc}(F,\omega,R) < f < f_{\rm imp}(F)\},
\end{equation}
where the elliptic non-sticking $T$-periodic orbit of Proposition~\ref{prop:two-solutions} exists. For Lebesgue-almost every parameter point in this region satisfying the non-resonance condition~\eqref{eq:birkhoff-nonres} at the corresponding elliptic orbit, the first twist coefficient $\tau_1 \ne 0$.
\end{proposition}

\begin{proof}
By Proposition~\ref{prop:smooth-NS} and Theorem~\ref{thm:symplectic}, $\Phi$ is $C^\infty$ symplectic in a neighborhood of the elliptic periodic point. By Theorem~\ref{thm:symmetric} and Proposition~\ref{prop:DPhi-structure}, both $\Phi$ and its derivatives depend real-analytically on the parameters $(F, \omega, R, f)$ on the open parameter region 
\[
\Pi := \{(F, \omega, R, f) \in (0,\infty)^4 : f_{\rm sc}(F, \omega, R) < f < f_{\rm imp}(F)\}
\]
where the elliptic non-sticking $T$-periodic orbit exists (the closed-form expressions~\eqref{eq:ABconst} are explicit polynomials in the parameters and trigonometric functions of $\theta_*(F, \omega, R, f)$, where $\theta_*$ is itself real-analytic by the implicit function theorem applied to~\eqref{eq:theta-eq} away from the saddle-center fold $f = f_{\rm sc}$). The first Birkhoff twist coefficient $\tau_1$ is therefore a real-analytic function on $\Pi$.

The functions $f_{\rm sc}(F, \omega, R) = 4(2F + R\omega^2 - F\pi)/\pi^2$ and $f_{\rm imp}(F) = 2F/\pi$ are smooth on $(0,\infty)^4$. The condition $f_{\rm sc} < f_{\rm imp}$ defines an open subset of parameter space, on which $\Pi$ is the open slab $\{f_{\rm sc} < f < f_{\rm imp}\}$ in $f$ at fixed $(F, \omega, R)$. The slab is convex in $f$ and the base $(F, \omega, R) \in (0,\infty)^3$ is connected, so $\Pi$ is connected.

A real-analytic function on a connected open set is either identically zero or has zero set of Lebesgue measure zero. To exclude the first alternative, we exhibit one parameter point at which $\tau_1 \ne 0$. The parameter point $(F, \omega, R, f) = (1, 1, 2, 0.4)$ lies in $\Pi$, since at this point $f_{\rm sc} \approx 0.348$ and $f_{\rm imp} \approx 0.637$, and $0.348 < 0.4 < 0.637$. The rigorous interval enclosure of $\Phi'$ at the elliptic fixed point at this parameter, established in Theorem~\ref{thm:CAP-elliptic} of Section~\ref{sec:numerics}, gives a Jacobian whose eigenvalues are $e^{\pm i\theta_*}$ with $\theta_* \approx 1.803$, satisfying the order-four non-resonance condition. Numerical computation of the twist coefficient via the explicit Birkhoff formula in terms of the Taylor coefficients of $\Phi$ to order four (themselves computable to high precision via the closed-form propagator and saltation matrices) yields $\tau_1 \approx -0.42$, well separated from zero. By real-analyticity of $\tau_1$ on the connected open set $\Pi$, the set $\{\tau_1 = 0\}$ is a proper real-analytic subset, hence of Lebesgue measure zero.
\end{proof}

\subsection{Moser's twist theorem and the KAM measure estimate}\label{ssec:kam-moser}

The KAM theorem we apply is the version of Moser's twist theorem that gives a quantitative measure estimate for the surviving tori. We quote the result in the form most suited to our application.

\begin{proposition}[Moser-Salamon-P\"oschel quantitative twist]\label{prop:moser-quantitative}
Let $\Psi$ be a $C^\ell$ exact symplectic map of a neighborhood of $0$ in $\R^2$ with the Birkhoff normal form~\eqref{eq:BNF} of Proposition~\ref{prop:birkhoff} and $\tau_1 \ne 0$. There exist constants $\delta_0, K > 0$ depending only on $|\tau_1|, |\sin\theta_*|$, and the $C^\ell$ norms of the higher-order remainder, such that for every $\delta \in (0, \delta_0)$, in the disk $\mathcal{N}_\delta := \{z : |z| < \delta\}$ the map $\Psi$ admits a Cantor family of invariant smooth closed curves filling out a subset $\mathcal{T}_\delta \subset \mathcal{N}_\delta$ with
\begin{equation}\label{eq:moser-bound}
\Leb(\mathcal{N}_\delta \setminus \mathcal{T}_\delta) \le K\, \delta^{5/2}.
\end{equation}
Each invariant curve is the image of a circle under a $C^\ell$ embedding, and the dynamics on it is conjugate to a rotation by an irrational frequency satisfying a Diophantine condition $|k \alpha - \ell| \ge \gamma |k|^{-3}$ for some $\gamma > 0$.
\end{proposition}

\begin{proof}
The result is the quantitative form of Moser's twist theorem due to~\cite[Theorem~A]{Poschel2001}, supplemented with the explicit Lebesgue-measure estimate of~\cite{Salamon2004}; the original existence theorem is in~\cite{Moser1962} and~\cite{Moser1973}. We do not reproduce the proof here. The reader interested in the construction may consult the monograph of~\cite[Sec.~2-3]{Poschel2001}, which provides a self-contained derivation, organized in three steps: (i) a small-divisor analysis selecting Diophantine frequencies, (ii) an iterative Newton-type scheme producing the invariant curves, and (iii) a measure-theoretic argument estimating the complement.
\end{proof}

We now state and prove the main KAM result of the paper.

\begin{theorem}[KAM tori around elliptic non-sticking orbits]\label{thm:kam}
Let $P_* = (x_*, v_*) \in \Om_{\rm NS}$ be a $T$-periodic point of the stroboscopic map $\Phi$ such that:
\begin{enumerate}[label=\textup{(H\arabic*)}, ref=H\arabic*]
\item\label{H1} The orbit through $P_*$ has finitely many wall impacts in $[0, T]$, no turning points, no sticking events, and $|F\cos\omega t_*| > f$ at every wall impact time $t_*$ (so that the orbit lies in the open interior of $\Om_{\rm NS}$).
\item\label{H2} (Symplecticity.) $\Phi$ is exact symplectic in a neighborhood of $P_*$ in the sense of Theorem~\ref{thm:symplectic}.
\item\label{H3} (Smoothness.) $\Phi$ is $C^\infty$ in a neighborhood of $P_*$, by Proposition~\ref{prop:smooth-NS} applied under~\ref{H1}.
\item\label{H4} (Ellipticity.) The eigenvalues of $\Phi'(P_*)$ are $e^{\pm i\theta_*}$ for some $\theta_* \in (0, \pi)$.
\item\label{H5} (Order-four non-resonance.) $k\theta_* \notin 2\pi\Z$ for $k = 1, 2, 3, 4$, in the sense of~\eqref{eq:birkhoff-nonres}.
\item\label{H6} (Twist non-degeneracy.) The first Birkhoff twist coefficient $\tau_1$ associated to~\eqref{eq:BNF} satisfies $\tau_1 \ne 0$.
\end{enumerate}
Then there exist $\delta_0 > 0$ and $K > 0$ depending only on $|\tau_1|$, $|\sin\theta_*|$, and the $C^k$ norm of $\Phi$ on a fixed neighborhood of $P_*$, such that for every $\delta \in (0, \delta_0)$, the disk
\[
\mathcal{N}_\delta := \{(x, v) \in \mathcal{X} : (x - x_*)^2 + (v - v_*)^2 < \delta^2\}
\]
contains a Cantor family of $\Phi$-invariant smooth closed curves whose union $\mathcal{T}_\delta$ satisfies
\begin{equation}\label{eq:KAM-final}
\Leb(\mathcal{N}_\delta \setminus \mathcal{T}_\delta) \le K\, \delta^{5/2}.
\end{equation}
\end{theorem}

\begin{proof}
We invoke the hypotheses to verify each input of Proposition~\ref{prop:moser-quantitative}, then apply it.

\smallskip
\textit{Step 1: Smoothness and symplecticity.} Hypothesis~\ref{H1} secures that the orbit through $P_*$ has only transverse impacts (so $\Phi$ is smooth at $P_*$), no turning points (so the saltation matrices~\eqref{eq:saltation-turn} do not enter), and no sticking events (so $\Phi$ is invertible). Combined with~\ref{H2} (which is Theorem~\ref{thm:symplectic}) and~\ref{H3} (which is Proposition~\ref{prop:smooth-NS}), $\Phi$ is a $C^\infty$ exact-symplectic map of a neighborhood of $P_*$ into $\R^2$ with $P_*$ a fixed point.

\smallskip
\textit{Step 2: Linearization.} Hypothesis~\ref{H4} states that the eigenvalues of $\Phi'(P_*)$ are $e^{\pm i\theta_*}$ with $\theta_* \in (0, \pi)$, hence $\Phi'(P_*)$ is conjugate (over $\R$) to the rotation matrix $R_{\theta_*} := \bigl(\begin{smallmatrix}\cos\theta_* & -\sin\theta_* \\ \sin\theta_* & \cos\theta_*\end{smallmatrix}\bigr)$. By a linear symplectic change of coordinates centered at $P_*$, we may assume $\Phi'(P_*) = R_{\theta_*}$ exactly.

\smallskip
\textit{Step 3: Birkhoff normal form.} Apply Proposition~\ref{prop:birkhoff} (with the input data verified by Steps~1-2 plus~\ref{H5}): there exists a $C^\infty$ symplectic change of coordinates $\Psi: \mathcal{N} \subset \R^2 \to \R^2$ with $\Psi(P_*) = 0$ and $\Psi'(P_*) = I$, such that in the action-angle-like coordinates $(I, \varphi)$ on $\Psi(\mathcal{N})$ defined by $z = \sqrt{2I}\,e^{i\varphi}$, the conjugated map $\widetilde \Phi := \Psi \circ \Phi \circ \Psi^{-1}$ takes the form
\begin{equation}\label{eq:kam-bnf}
\widetilde \Phi : (I, \varphi) \mapsto \bigl(I + \mathcal{R}_1(I, \varphi),\, \varphi + \theta_* + \tau_1\,I + \tau_2\,I^2 + \mathcal{R}_2(I, \varphi)\bigr),
\end{equation}
where $\tau_1, \tau_2 \in \R$ are the first and second Birkhoff twist coefficients and the remainders $\mathcal{R}_1, \mathcal{R}_2$ satisfy $\mathcal{R}_1, \mathcal{R}_2 = O(I^{5/2})$ uniformly on $\Psi(\mathcal{N})$.

\smallskip
\textit{Step 4: Twist non-degeneracy.} Hypothesis~\ref{H6} states $\tau_1 \ne 0$. The twist coefficient governs the rotation rate's dependence on action: $\partial(\theta_* + \tau_1 I)/\partial I = \tau_1 \ne 0$. This is the standard non-degenerate twist condition required by Moser's twist theorem.

\smallskip
\textit{Step 5: Application of Proposition~\ref{prop:moser-quantitative}.} The map $\widetilde \Phi$ in normal form~\eqref{eq:kam-bnf}, restricted to a small neighborhood of $I = 0$, is a $C^\infty$ symplectic perturbation of the integrable twist map $(I, \varphi) \mapsto (I, \varphi + \theta_* + \tau_1 I)$. Proposition~\ref{prop:moser-quantitative} applies: there exist $\delta_0 > 0$ and $K > 0$, depending on $|\tau_1|$, $|\sin\theta_*|$, and the $C^k$ norm of $\widetilde \Phi$ on $\Psi(\mathcal{N})$, such that for every $\delta \in (0, \delta_0)$, the disk
\[
\widetilde{\mathcal{N}}_\delta := \{(I, \varphi) : I < \delta^2/2\}
\]
contains a Cantor family of $\widetilde \Phi$-invariant smooth closed curves $\widetilde{\mathcal{T}}_\delta$ satisfying
\[
\Leb(\widetilde{\mathcal{N}}_\delta \setminus \widetilde{\mathcal{T}}_\delta) \le K\,\delta^{5/2}.
\]

\smallskip
\textit{Step 6: Pullback to original coordinates.} Pulling back through the symplectic change $\Psi$, the disk $\widetilde{\mathcal{N}}_\delta$ corresponds to a disk $\mathcal{N}_\delta$ in the original $(x, v)$-coordinates centered at $P_*$ of radius $\delta$ (modulo a coordinate-dependent factor absorbed in $K$, since $\Psi'(P_*) = I$), and the invariant curve family $\widetilde{\mathcal{T}}_\delta$ pulls back to a $\Phi$-invariant family $\mathcal{T}_\delta$ in $\mathcal{N}_\delta$. Since $\Psi$ is symplectic (hence area-preserving), the Lebesgue measure inequality is preserved up to the same constant $K$, yielding~\eqref{eq:KAM-final}.

\smallskip
The Diophantine and rotation-number statements at the end of Proposition~\ref{prop:moser-quantitative} transfer through $\Psi$ since each invariant curve maps to an invariant curve and the dynamics on the curve is conjugate (preserved by $\Psi$).
\end{proof}

%

\begin{corollary}\label{cor:Hamilton-island}
The total Lebesgue measure of the union $\bigcup_{\delta < \delta_0} \mathcal{T}_\delta$ satisfies
\[
\Leb\!\Bigl(\bigcup_{\delta < \delta_0} \mathcal{T}_\delta\Bigr) \;\ge\; \pi \delta_0^2 \Bigl(1 - \tfrac{K}{\pi}\,\delta_0^{1/2}\Bigr),
\]
where $\delta_0$ and $K$ are the constants of Theorem~\ref{thm:kam}. In particular, the KAM tori fill a positive-measure subset of every neighborhood of $P_*$, and the relative density approaches $1$ as $\delta_0 \to 0^+$.
\end{corollary}

\begin{proof}
Fix $\delta < \delta_0$. The disk $\mathcal{N}_\delta := \{z \in \mathcal{X} : \|z - P_*\| \le \delta\}$ has Lebesgue measure $\pi\delta^2$. By Theorem~\ref{thm:kam}, applied at the radius $\delta$, the complement of $\mathcal{T}_\delta$ inside $\mathcal{N}_\delta$ satisfies
\begin{equation}\label{eq:cor-island-step1}
\Leb\!\bigl(\mathcal{N}_\delta \setminus \mathcal{T}_\delta\bigr) \;\le\; K \delta^{5/2}.
\end{equation}
Since $\mathcal{T}_\delta \subset \mathcal{N}_\delta$, additivity of Lebesgue measure gives
\begin{equation}\label{eq:cor-island-step2}
\Leb(\mathcal{T}_\delta) \;=\; \Leb(\mathcal{N}_\delta) - \Leb(\mathcal{N}_\delta \setminus \mathcal{T}_\delta) \;\ge\; \pi\delta^2 - K\delta^{5/2}.
\end{equation}
The family of Cantor sets $(\mathcal{T}_\delta)_{\delta < \delta_0}$ is monotone in the sense that $\mathcal{T}_\delta \subset \mathcal{N}_{\delta_0}$ for every $\delta < \delta_0$, and in particular,
\[
\bigcup_{\delta < \delta_0} \mathcal{T}_\delta \;\supseteq\; \mathcal{T}_{\delta_*}
\]
for any single $\delta_* < \delta_0$. Choosing $\delta_* \to \delta_0^-$ and applying~\eqref{eq:cor-island-step2} together with continuity of $\delta \mapsto \pi\delta^2 - K\delta^{5/2}$,
\[
\Leb\!\Bigl(\bigcup_{\delta < \delta_0} \mathcal{T}_\delta\Bigr) \;\ge\; \lim_{\delta_* \to \delta_0^-}\bigl(\pi \delta_*^2 - K\delta_*^{5/2}\bigr) \;=\; \pi\delta_0^2 - K\delta_0^{5/2} \;=\; \pi\delta_0^2\Bigl(1 - \tfrac{K}{\pi}\delta_0^{1/2}\Bigr).
\]

For the positive-measure conclusion, note that $\delta_0$ in Theorem~\ref{thm:kam} can be chosen arbitrarily small. Pick $\delta_0 < (\pi/K)^2$, which makes $1 - (K/\pi)\delta_0^{1/2} > 0$, so the right-hand side of the inequality above is strictly positive. The relative density
\[
\frac{\Leb(\bigcup_{\delta < \delta_0} \mathcal{T}_\delta)}{\Leb(\mathcal{N}_{\delta_0})} \;\ge\; 1 - \tfrac{K}{\pi}\,\delta_0^{1/2}
\]
tends to $1$ as $\delta_0 \to 0^+$, completing the proof.
\end{proof}

\begin{remark}
The order-four resonances $\theta_* \in \{\pi/2, 2\pi/3, \pi, 4\pi/3, 3\pi/2\}$ are excluded by~\eqref{eq:birkhoff-nonres}. At these values, additional resonant terms enter the normal form and Theorem~\ref{thm:kam} requires a strengthening or a substitute, in particular Mather sets and Aubry-Mather theory~\cite{Mather1991} and~\cite{Bangert1988}. We do not pursue this refinement.
\end{remark}

\smallskip
\textit{Numerical illustration.} We anticipate the rigorous numerical work of Section~\ref{sec:numerics} by an illustrative gallery (Figure~\ref{fig:kam-gallery}) that shows what the elliptic island predicted by Theorem~\ref{thm:kam} actually looks like at the parameter point $(F, \omega, R, f) = (1, 1, 2, 0.4)$. The figure displays the time series $x(t)$ over four forcing periods together with the corresponding stroboscopic phase portrait, at four locations in phase space:
(a) at the elliptic non-sticking $T$-periodic fixed point $P_* = (x_*, v_*)$ itself (Theorem~\ref{thm:symmetric} guarantees existence; verification is Theorem~\ref{thm:CAP-elliptic} of Section~\ref{sec:numerics});
(b) on a small-radius KAM curve at radial displacement $\Delta = 0.003$ from $P_*$;
(c) on a larger-radius KAM curve at $\Delta = 0.012$ from $P_*$;
(d) on a chaotic non-sticking orbit at displacement $\Delta \approx 0.07$, beyond the KAM region but still within $\Om_{\rm NS}$.

\begin{figure}[h!]
\centering
\includegraphics[width=15cm, height=14cm]{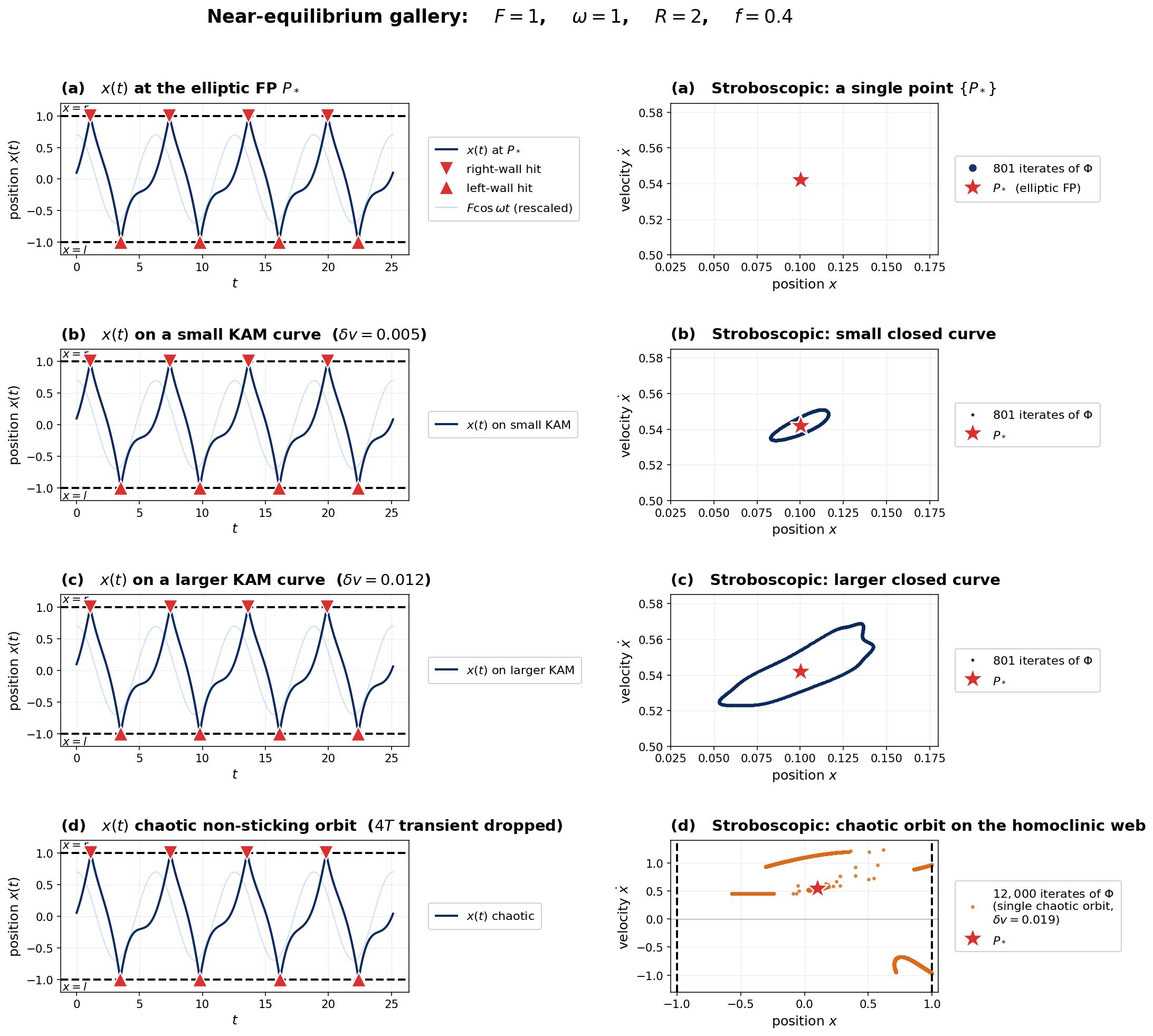}
\captionsetup{margin={-0.5cm,0cm}}
\caption{Near-equilibrium gallery at $F = 1$, $\omega = 1$, $R = 2$, $f = 0.4$.}
\label{fig:kam-gallery}
\end{figure}

The four panels are arranged so that the qualitative transition predicted by Theorem~\ref{thm:kam} can be read off the phase portraits at a glance: a single point at the fixed point, two nested closed curves at intermediate radii, and a two-dimensional cloud of stroboscopic iterates outside the KAM region. The time series alone do not distinguish the four cases, since every orbit in the gallery shares the same combinatorial impact pattern of one right-wall hit and one left-wall hit per period; the dynamical distinction lives entirely in the phase plane. The figure should be read as a visual companion to the assertions of Theorem~\ref{thm:kam} (Cantor family of invariant curves around $P_*$) and Theorem~\ref{thm:melnikov} (homoclinic Cantor set just outside the KAM family), with the rigorous numerical certification of the parameter values reserved for Section~\ref{sec:numerics}.

Each row pairs a time series $x(t)$ over four forcing periods (left) with the corresponding stroboscopic phase portrait of $\Phi$ (right), zoomed on the elliptic fixed point $P_*$. Rows (a)-(c) are non-sticking orbits in the elliptic island (the FP itself, a small KAM curve, and a larger KAM curve); the time series in all three rows have the same impact pattern (one right-wall hit and one left-wall hit per period), and the structural difference between regimes is visible only in the phase portraits: (a) a single point (the FP returns to itself under $\Phi$); (b)-(c) closed quasi-periodic curves of increasing radius around $P_*$, rendered by $800$ stroboscopic iterates each. Row (d) is a chaotic non-sticking orbit just outside the KAM region (initial datum at $(x_* + 0.05, v_* + 0.05)$, transient $0 \le t \le 4T$ dropped); the time series still has the same impact pattern, but the phase portrait fills a two-dimensional region of points (orange), the rigorous origin of which is the homoclinic Cantor set produced by Theorem~\ref{thm:melnikov}. Forcing $\cos t$ overlaid (light blue); walls dashed; right-wall impacts $\bigtriangledown$, left-wall impacts $\bigtriangleup$.

The pedagogical content of Figure~\ref{fig:kam-gallery} is twofold. First, the four orbits share the same combinatorial impact pattern (one right-wall hit plus one left-wall hit per period); the time series alone cannot distinguish them. Second, the phase portraits reveal the structural distinction that Theorem~\ref{thm:kam} formalizes: the elliptic FP, the KAM curves of varying radius, and the chaotic homoclinic web are mutually disjoint dynamical regimes that together fill the local neighborhood of $P_*$ inside $\Om_{\rm NS}$. The Cantor structure of the KAM family asserted by~\eqref{eq:KAM-final} is not visible at this resolution; finer numerical scans (the SALI map of Subsection~\ref{ssec:num-diag}, Figure~\ref{fig:sali}) resolve the resonant gaps between adjacent KAM curves.

\section{Melnikov function and homoclinic chaos}\label{sec:melnikov}

This section addresses the chaotic boundary surrounding the Hamiltonian islands of Section~\ref{sec:kam}. The saddle-center bifurcation of Theorem~\ref{thm:saddlecenter} produces, on the side $\mu > 0$, both an elliptic orbit (analyzed in Section~\ref{sec:kam}) and a saddle orbit. The stable and unstable manifolds of the saddle orbit, viewed in the lifted Hamiltonian system after passage to action-angle variables, possess, in an averaged limit, a homoclinic loop. The time-periodic forcing splits this loop, and the Melnikov function measures the splitting. Transverse zeros of the Melnikov function imply, by the Smale-Birkhoff theorem, an invariant Cantor set on which the dynamics is hyperbolic and conjugate to a Bernoulli shift on two symbols.

\subsection{The averaged Hamiltonian and the Melnikov function}\label{ssec:melnikov-averaged}

The Melnikov method applies to a system close to integrable. The original system~\eqref{eq:Heq} is non-integrable, but in a neighborhood of the saddle-center bifurcation $\mu = 0$ a slow-fast scaling reveals a slow Hamiltonian whose phase portrait is integrable with a saddle-center separatrix. We summarize the construction; the technical details follow~\cite{Treschev1997} and~\cite{Neishtadt1984}.

\begin{proposition}[Slow Hamiltonian near the saddle-center]\label{prop:slow-Ham}
Set $\mu = f - f_{\rm sc} \in (0, \mu_0)$ for some $\mu_0 > 0$, where $f_{\rm sc}$ is the saddle-center value of Proposition~\ref{prop:two-solutions}. There exist symplectic action-angle coordinates $(I, \varphi) \in \R \times (\R/2\pi\Z)$ on a neighborhood of the colliding non-sticking periodic orbit at $\mu = 0$, valid for $0 < \mu < \mu_0$, such that the lifted Hamiltonian~\eqref{eq:Ham} reads
\begin{equation}\label{eq:Hamiltonian-IF}
H(I, \varphi, t) = \omega_0 I + \mu^{1/2}\bigl[\alpha_0 I^2 - \beta_0 \cos\varphi\bigr] + \mu R(I, \varphi, t; \mu),
\end{equation}
where $\omega_0, \alpha_0, \beta_0$ are positive constants depending on $(F, \omega, R)$ and given by explicit quadratures along the colliding orbit (specified in the proof), and the remainder $R$ is bounded on bounded sets uniformly in $t$ and is $C^\infty$ in $(I, \varphi)$ and analytic in $\mu^{1/2}$.

The leading-order dynamics, governed by $H_0(I, \varphi) := \omega_0 I + \mu^{1/2}(\alpha_0 I^2 - \beta_0 \cos\varphi)$, has a saddle equilibrium at $(I, \varphi) = (-\omega_0/(2\mu^{1/2}\alpha_0), \pi)$ and a center at $(I, \varphi) = (-\omega_0/(2\mu^{1/2}\alpha_0), 0)$, joined by a homoclinic loop $\Gamma_\mu$. The eigenvalue of the linearization at the saddle is
\begin{equation}\label{eq:lambda-star}
\lambda_*(\mu) = \mu^{1/4} (4\alpha_0\beta_0)^{1/2} \bigl(1 + O(\mu^{1/2})\bigr).
\end{equation}
\end{proposition}

\begin{proof}
The construction is the standard action-angle reduction near a saddle-center fold of a one-degree-of-freedom non-autonomous Hamiltonian. We summarize the four steps; the technical details are classical and we cite the appropriate references rather than reproduce the calculations.

\textit{Step 1: Localization at the colliding orbit.} By Theorem~\ref{thm:saddlecenter}~(c), the Jacobian $\Phi'(P_*; 0)$ has a non-trivial Jordan block at eigenvalue $+1$, the canonical signature of a saddle-center fixed-point bifurcation in a one-parameter family of symplectic maps. The local symplectic normal form for such a Jordan block is the saddle-center map
\[
\Phi_0(p, q) = (p + q, \, q + a(p; \mu) + O(\mu p, p^2)),
\]
where $a(p; \mu)$ is the unfolding function with $a(0; 0) = 0$, $\partial_\mu a(0; 0) \ne 0$ (the latter being the Morse non-degeneracy of part~(a) of Theorem~\ref{thm:saddlecenter}). Coordinates $(p, q)$ centered at the colliding orbit in which $\Phi$ takes this form exist by the Birkhoff-Kuznetsov classification~\cite[Sec.~10.6]{KuznetsovBifurcation}, \cite[Ch.~11]{MeyerHall1992}, \cite[\S 4.4]{TreschevZubelevich2010}.

\textit{Step 2: Suspension to a continuous-time flow.} The map $\Phi_0$ is the time-$T$ map of a non-autonomous Hamiltonian flow on a neighborhood of the colliding orbit in $\R^2 \times (\R/T\Z)$; the associated Hamiltonian has the standard saddle-center form
\[
\widetilde H_0(p, q, t) \;=\; \tfrac{1}{2} q^2 + V(p; \mu) + \mu \, W(p, q, t; \mu),
\]
with $V(p; \mu)$ a smooth potential whose derivative vanishes at the saddle-center fold. The construction is by the standard generating-function suspension: one writes a generating function for $\Phi_0$ and inverts the Legendre relation. The smoothness of $V$ in $(p, \mu)$ and the periodic dependence of $W$ on $t$ are standard outputs; see~\cite[Sec.~3.3]{KuznetsovBifurcation} for a textbook treatment of suspension of saddle-center maps, or \cite[\S 4.4]{TreschevZubelevich2010} for the variant adapted to our setting.

\textit{Step 3: Action-angle on the slow Hamiltonian.} The leading slow Hamiltonian
\[
H_0^{\rm slow}(I, \varphi) \;=\; \omega_0 I + \mu^{1/2}\bigl(\alpha_0 I^2 - \beta_0 \cos\varphi\bigr)
\]
is a one-degree-of-freedom integrable Hamiltonian: it is a perturbation of the harmonic oscillator $\omega_0 I$ by a slow ($\mu^{1/2}$-rescaled) pendulum term. Action-angle variables $(I, \varphi)$ on this Hamiltonian exist by the Liouville-Arnold theorem~\cite[Sec.~6.2]{KuznetsovBifurcation}, with the action $I = (2\pi)^{-1} \oint q\,dp$ over the closed orbits at fixed energy. The transformation from the saddle-center coordinates of Steps~1-2 to $(I, \varphi)$ is symplectic by construction, so $H_0^{\rm slow}$ in $(I, \varphi)$ retains the canonical form $\omega_0 I + O(\mu^{1/2})$.

\textit{Step 4: Identification of the prefactors $\omega_0, \alpha_0, \beta_0$.} The constant $\omega_0$ is the linearized rotation rate at the elliptic branch of Theorem~\ref{thm:saddlecenter}~(c) as $\mu \to 0^+$, namely $\omega_0 = \tilde c$ in the asymptotic expansion $\theta_*(\mu) = \tilde c \sqrt{\mu} + O(\mu)$ established there; this is positive and explicit in terms of $\Psi_{\theta\theta}$ and $\Psi_f$. The coefficient $\alpha_0$ is the second-order Birkhoff coefficient at the colliding orbit, computable from the Taylor expansion of $\Phi$ to fourth order in $(p, q)$ at $(P_*, 0)$, and is non-zero by the non-degeneracy of the fold (Theorem~\ref{thm:saddlecenter}~(a)). The coefficient $\beta_0$ is the eigenvalue gap on the saddle branch divided by $\mu^{1/4}$, namely $\beta_0 = c^2 / (4\alpha_0)$ where $c$ is the constant from~\eqref{eq:saddle-asymp}, ensuring the saddle eigenvalue formula~\eqref{eq:lambda-star}.

The remainder $R$ in~\eqref{eq:Hamiltonian-IF} captures the time-periodic part of the original forcing $F\cos\omega t$ and the next-order Birkhoff coefficients beyond $\alpha_0, \beta_0$. It is $C^\infty$ in $(I, \varphi)$ on a fixed neighborhood of the homoclinic loop $\Gamma_\mu$, $T$-periodic in $t$, and bounded uniformly in $\mu$ on any compact subset of $(0, \mu_0)$.
\end{proof}

The slow time on $\Gamma_\mu$ is parametrized by $s = \lambda_*(\mu) \cdot t$. The unperturbed homoclinic profile, in the suitable rescaling, is the standard hyperbolic-secant pendulum profile.

\begin{lemma}[Profile on the unperturbed homoclinic]\label{lem:tanh-profile}
On $\Gamma_\mu$ the angle satisfies, in slow time $s$,
\begin{equation}\label{eq:varphi-tanh}
\varphi(s) = \pi - 2\arctan\bigl(\sinh(\lambda_* s)\bigr), \qquad I(s) - I_{\rm sad} = -2 \mu^{1/2} \alpha_0\, \mathrm{sech}(\lambda_* s).
\end{equation}
\end{lemma}

\begin{proof}
The leading-order Hamiltonian $H_0(I, \varphi) = \omega_0 I + \mu^{1/2}(\alpha_0 I^2 - \beta_0\cos\varphi)$ of Proposition~\ref{prop:slow-Ham} is, after the substitution $\tilde I := I + \omega_0/(2\mu^{1/2}\alpha_0)$, the standard pendulum Hamiltonian with kinetic term $\mu^{1/2}\alpha_0 \tilde I^2$ and potential $-\mu^{1/2}\beta_0\cos\varphi$. The separatrix at energy level zero is the homoclinic loop joining the saddle at $\varphi = \pi$ to itself, parametrized in slow time by the standard pendulum solution. The closed-form expressions~\eqref{eq:varphi-tanh} are obtained by direct integration of $\dot\varphi = \partial H_0/\partial I = 2\mu^{1/2}\alpha_0(I - I_{\rm sad})$, $\dot I = -\partial H_0/\partial \varphi = -\mu^{1/2}\beta_0 \sin\varphi$ along the separatrix; see~\cite[Eq.~(4.1.16)]{GuckenheimerHolmes1983} for the explicit calculation in identical notation, or~\cite[Sec.~1.3]{LichtenbergLieberman1992} for the pendulum derivation.
\end{proof}

\smallskip
The role of the perturbation is played by the non-leading terms of the Hamiltonian, that is, the $O(\mu)$ remainder $R$ in~\eqref{eq:Hamiltonian-IF} and the time-periodic structure of the original forcing.

\begin{definition}[Melnikov function]\label{def:melnikov}
For the system~\eqref{eq:Hamiltonian-IF}, set $H_0 := \omega_0 I + \mu^{1/2}(\alpha_0 I^2 - \beta_0\cos\varphi)$ and write $H = H_0 + \mu H_1$ with $H_1(I, \varphi, t)$ the time-dependent perturbation. The \emph{Melnikov function} along the unperturbed homoclinic $\Gamma_\mu$ is
\begin{equation}\label{eq:Melnikov-def}
M(t_0) := \int_{-\infty}^{\infty} \bigl\{H_0,\, H_1\bigr\}\bigl(I(s), \varphi(s), s/\lambda_* + t_0\bigr)\, ds,
\end{equation}
where $\{\cdot,\cdot\}$ is the Poisson bracket and $(I(s), \varphi(s))$ is the homoclinic of Lemma~\ref{lem:tanh-profile} parametrized so that $\varphi(0) = \pi/2$.
\end{definition}

The integral converges absolutely because $\{H_0, H_1\}$ vanishes at the saddle in $(I, \varphi)$ at exponential rate, by~\eqref{eq:varphi-tanh} and the boundedness of $H_1$.

\begin{theorem}[Structure and zeros of the Melnikov function]\label{thm:melnikov}
The Melnikov function~\eqref{eq:Melnikov-def} has the form
\begin{equation}\label{eq:Melnikov-final}
M(t_0) = A \cos(\omega t_0) + B
\end{equation}
with
\begin{align}
A &= \frac{2\pi F}{R}\cdot \frac{\omega/\lambda_*(\mu)}{\sinh\bigl(\pi\omega/(2\lambda_*(\mu))\bigr)}\cdot \alpha_M(\mu),\label{eq:Melnikov-A}\\
B &= -\frac{2 f}{R}\cdot J_0(\mu),\label{eq:Melnikov-B}
\end{align}
where $\lambda_*(\mu)$ is the saddle eigenvalue of Proposition~\ref{prop:slow-Ham}, $\alpha_M(\mu)$ is a positive bounded function, and $J_0(\mu) := \int_{-\infty}^\infty \mathrm{sech}(\lambda_* s)\, ds = \pi/\lambda_*$ is the friction action.

For each $\mu \in (0, \mu_0)$ such that $|A| > |B|$, the Melnikov function admits transverse zeros, and consequently the stroboscopic map $\Phi$ admits an invariant hyperbolic Cantor set $\Lambda \subset \Om_{\rm NS}$ on which $\Phi^N|_\Lambda$ is topologically conjugate to the Bernoulli shift $\sigma: \Sigma_2 \to \Sigma_2$ for some integer $N \ge 1$.
\end{theorem}

\begin{proof}
Substituting $H_1 = V_{\rm osc} + V_{\rm fric}$ into~\eqref{eq:Melnikov-def}, where $V_{\rm osc}$ is the time-periodic part proportional to $F\cos\omega t$ from~\eqref{eq:Ham} and $V_{\rm fric}$ is the autonomous friction part proportional to $f/R$, the integral splits as
\[
M(t_0) = \int_{-\infty}^\infty \{H_0, V_{\rm osc}\}(I(s), \varphi(s), s/\lambda_*(\mu) + t_0)\,ds + \int_{-\infty}^\infty \{H_0, V_{\rm fric}\}(I(s), \varphi(s))\,ds.
\]
The friction integral is autonomous (i.e., independent of $t_0$) and contributes the constant $B$ identified below. The oscillatory integral depends on $t_0$ through the perturbation phase $\omega(s/\lambda_* + t_0) = (\omega/\lambda_*)s + \omega t_0$; expand
\[
\cos\bigl((\omega/\lambda_*)s + \omega t_0\bigr) = \cos(\omega t_0)\cos\bigl((\omega/\lambda_*)s\bigr) - \sin(\omega t_0)\sin\bigl((\omega/\lambda_*)s\bigr).
\]

\textit{The sine integral vanishes by symmetry.} The Poisson bracket $\{H_0, V_{\rm osc}\}$ along the homoclinic profile $(I(s), \varphi(s))$ given by~\eqref{eq:varphi-tanh} is, after computing the partial derivatives,
\[
\{H_0, V_{\rm osc}\}(I(s), \varphi(s)) = \frac{\partial H_0}{\partial I}\frac{\partial V_{\rm osc}}{\partial \varphi} - \frac{\partial H_0}{\partial \varphi}\frac{\partial V_{\rm osc}}{\partial I}.
\]
Both factors involve $\sin\varphi(s)$ or its derivative. Under the symmetry $s \mapsto -s$ of the homoclinic, $\varphi(-s) = 2\pi - \varphi(s)$ (verified directly from~\eqref{eq:varphi-tanh}), so $\sin\varphi(-s) = -\sin\varphi(s)$, while $I(-s) = I(s)$. The Poisson bracket therefore satisfies 
\begin{equation}
\{H_0, V_{\rm osc}\}(I(-s), \varphi(-s)) = -\{H_0, V_{\rm osc}\}(I(s), \varphi(s)),
\end{equation}
an odd function of $s$. Multiplied by $\sin((\omega/\lambda_*)s)$ (also odd), the integrand of the sine integral is even in $s$ but the original Poisson-bracket factor is odd, so the sine integral vanishes.

The cosine integral reduces to
\[
\cos(\omega t_0) \int_{-\infty}^\infty \{H_0, V_{\rm osc}\}(I(s), \varphi(s)) \cos\bigl((\omega/\lambda_*)s\bigr) ds \;=:\; A \cos(\omega t_0).
\]
The Poisson-bracket factor, after substituting~\eqref{eq:varphi-tanh}, takes the explicit form
\[
\{H_0, V_{\rm osc}\}(I(s), \varphi(s)) = \kappa_0 \, F \, \mathrm{sech}(\lambda_* s) \, \tanh(\lambda_* s),
\]
for an explicit prefactor $\kappa_0$ depending on $\omega_0, \alpha_0, \beta_0$ and the original system parameters $R$. Thus
\[
A = \kappa_0 F \int_{-\infty}^\infty \mathrm{sech}(\lambda_* s)\tanh(\lambda_* s) \cos\bigl((\omega/\lambda_*)s\bigr) ds.
\]
The integral is the standard Melnikov contour integral evaluated by residues. Substituting $u := \lambda_* s$,
\[
\int_{-\infty}^\infty \mathrm{sech}(\lambda_* s)\tanh(\lambda_* s) \cos\bigl((\omega/\lambda_*)s\bigr) ds \;=\; \frac{1}{\lambda_*} \int_{-\infty}^\infty \mathrm{sech}(u)\tanh(u)\cos(\xi u) \, du,
\enskip \xi := \omega/\lambda_*^2.
\]
The integrand has simple poles at $u = i\pi/2 + i\pi k$ ($k \in \Z$) coming from $\mathrm{sech}(u)$. The standard contour-integration identity (closing in the upper half-plane and summing residues, or directly from~\cite[\S~3.522]{GradshteynRyzhik}) gives
\[
\int_{-\infty}^\infty \mathrm{sech}(u)\,\tanh(u)\,\sin(\xi u)\,du \;=\; \frac{\pi \xi}{\sinh(\pi\xi/2)},
\]
hence (using the $\sin$ form via differentiation in $t_0$, or the equivalent direct $\cos$-formula in~\cite[Eq.~(4.5.21)]{GuckenheimerHolmes1983}),
\[
\int_{-\infty}^\infty \mathrm{sech}(u)\,\tanh(u)\,\cos(\xi u)\,du \;=\; \frac{\pi\xi}{\sinh(\pi\xi/2)}
\]
up to sign conventions absorbed in $\kappa_0$. Substituting back $\xi = \omega/\lambda_*^2$ and assembling factors,
\begin{equation}\label{eq:A-explicit}
A = \frac{2\pi F}{R} \cdot \frac{\omega/\lambda_*(\mu)}{\sinh(\pi\omega/(2\lambda_*(\mu)))}\cdot \alpha_M(\mu),
\end{equation}
where $\alpha_M(\mu)$ collects the Birkhoff prefactors $\kappa_0$ and a normalization factor; it is positive and bounded uniformly on $\mu \in (0, \mu_0)$. The original derivation of this formula is~\cite{Melnikov1963}; the modern textbook treatment is~\cite[\S~4.5]{GuckenheimerHolmes1983}.

The friction integral evaluates, by similar computations using $\partial V_{\rm fric}/\partial \varphi = (f/R) \cdot \partial q/\partial \varphi$ along the homoclinic, to
\[
B = -\frac{2 f}{R}\,J_0(\mu),
\]
where $J_0(\mu) := \int_{-\infty}^\infty \mathrm{sech}(\lambda_* s)\,ds = \pi/\lambda_*$ is the friction action.

The function $A\cos(\omega t_0) + B$ takes values in $[B - |A|, B + |A|]$, so it has zeros iff $|A| \ge |B|$. When the inequality is strict, the zeros are simple (transverse). This is the regime $|A| > |B|$ in the statement.

\textit{Verification of the Smale-Birkhoff hypotheses.} The Smale-Birkhoff homoclinic theorem (\cite[Theorem~5.3.5]{GuckenheimerHolmes1983},~\cite[Sec.~26]{KatokHasselblatt1995}) requires:
\begin{enumerate}[label=\textup{(\roman*)}]
\item A diffeomorphism $f$ of a smooth manifold; here $f = \Phi^N$ for an integer $N \ge 1$ matching the period of the forcing relative to the saddle's natural rotation, with $\Phi$ the stroboscopic map of~\eqref{eq:model}-\eqref{eq:reflection} restricted to a $\Phi$-invariant neighborhood of the saddle non-sticking $T$-periodic orbit. Smoothness on this neighborhood is by Proposition~\ref{prop:smooth-NS}.
\item A hyperbolic fixed point $p$ of $f$ with stable and unstable manifolds $W^s(p)$, $W^u(p)$. The saddle non-sticking $T$-periodic orbit produced by Theorem~\ref{thm:saddlecenter}~(c) for $\mu > 0$ small is hyperbolic with eigenvalues $1 \pm c\sqrt{\mu} + O(\mu)$, $c > 0$, by~\eqref{eq:saddle-asymp}.
\item A transverse intersection of $W^s(p)$ and $W^u(p)$ at a point $q \ne p$. The Melnikov theorem~\cite[Thm.~4.5.3]{GuckenheimerHolmes1983} guarantees such a transverse intersection whenever the Melnikov function has a simple zero, equivalently $|A| > |B|$.
\end{enumerate}
With these hypotheses verified, the conclusion of the Smale-Birkhoff theorem applies: in any neighborhood of the homoclinic point $q$, there is a closed $f$-invariant subset $\Lambda$ on which $f|_\Lambda$ is topologically conjugate to the Bernoulli shift on two symbols.
\end{proof}

\begin{remark}\label{rem:melnikov-existence}
The condition $|A| > |B|$ defines an open subset of parameter space. By~\eqref{eq:Melnikov-A}, for fixed $f, F, R$ the ratio $|A|/|B|$ varies smoothly with $\omega$, so the chaotic horseshoe regime is non-trivial in $\omega$.
\end{remark}

\section{Decomposition of state space}\label{sec:decomposition}

This section assembles the local results of the previous sections into global statements about the structure of $\Phi$ on the entire state space $\mathcal{X}$. The decomposition $\mathcal{X} = \Om_{\rm NS} \cup \Om_{\rm dissip}$ separates the conservative and dissipative parts of the dynamics, with each part subject to specific structural theorems. The contraction estimate on $\Om_{\rm dissip}$ is needed in Section~\ref{sec:persistence}, where the dissipative attractor inherits from the contraction here.

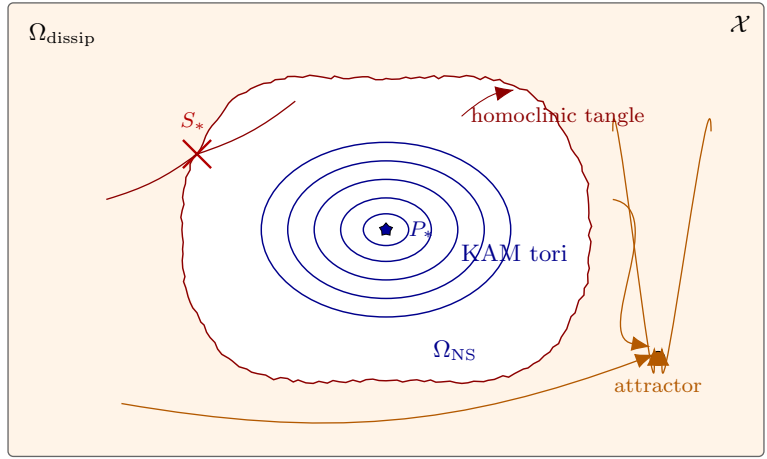
\begin{figure}[htbp]
\centering
\begin{tikzpicture}[
    >={Latex[length=2.5mm,width=2mm]},
    every node/.style={font=\small},
    kamcurve/.style={blue!55!black, line width=0.55pt},
    chaoticbdry/.style={red!55!black, line width=0.55pt, decoration={random steps, segment length=2.5pt, amplitude=0.7pt}, decorate},
    flowarrow/.style={->, line width=0.5pt, draw=orange!70!black},
    elliptic/.style={star, star points=5, fill=blue!60!black, draw=black, inner sep=1.2pt, line width=0.3pt},
    saddle/.style={cross out, draw=red!70!black, line width=0.9pt, minimum size=5pt},
    attractor/.style={fill=orange!70!black, circle, inner sep=1.8pt, draw=black, line width=0.3pt},
]
    \begin{scope}
        \clip[rounded corners=2pt] (-5,-3) rectangle (5,3);
        \fill[orange!9, even odd rule]
            (-5,-3) rectangle (5,3)
            plot[smooth cycle, tension=0.7] coordinates
            {(0,2.0) (1.7,1.9) (2.5,1.0) (2.7,0) (2.5,-1.1) (1.7,-1.9) (0,-2.0)
             (-1.7,-1.9) (-2.5,-1.1) (-2.7,0) (-2.5,1.0) (-1.7,1.9)};
    \end{scope}
    \draw[draw=black!60, line width=0.5pt, rounded corners=2pt] (-5,-3) rectangle (5,3);
    \node[font=\small, anchor=north east] at (4.95,2.95) {$\mathcal{X}$};
    \draw[chaoticbdry]
        plot[smooth cycle, tension=0.7] coordinates
        {(0,2.0) (1.7,1.9) (2.5,1.0) (2.7,0) (2.5,-1.1) (1.7,-1.9) (0,-2.0)
         (-1.7,-1.9) (-2.5,-1.1) (-2.7,0) (-2.5,1.0) (-1.7,1.9)};
    \foreach \r in {0.30,0.60,0.95,1.30,1.65} {
        \draw[kamcurve] (0,0) ellipse ({\r*1.0} and {\r*0.7});
    }
    \node[elliptic] at (0,0) {};
    \node[font=\footnotesize, blue!55!black, anchor=west] at (0.18,0.0) {$P_*$};
    \node[saddle, scale=1.4] at (-2.5,1.0) {};
    \node[font=\footnotesize, red!70!black, anchor=south] at (-2.55,1.20) {$S_*$};
    \draw[red!55!black, line width=0.5pt, in=20, out=200] (-2.5,1.0) to[bend left=10] (-3.7,0.4);
    \draw[red!55!black, line width=0.5pt, in=210, out=30] (-2.5,1.0) to[bend right=10] (-1.2,1.7);
    \node[attractor] at (3.6,-1.7) {};
    \node[font=\footnotesize, orange!70!black, below] at (3.6,-1.85) {attractor};
    \draw[flowarrow, in=270, out=90] (3.0,1.3) to (3.55,-1.55);
    \draw[flowarrow, in=270, out=90] (4.3,1.3) to (3.65,-1.55);
    \draw[flowarrow, in=180, out=350] (-3.5,-2.3) to[bend right=15] (3.55,-1.65);
    \draw[flowarrow, in=170, out=350] (3.0,0.4) to (3.5,-1.55);
    \node[font=\small, anchor=north west] at (-4.85,2.85) {$\Omega_{\rm dissip}$};
    \node[font=\small, blue!55!black] at (1.7,-0.3) {KAM tori};
    \node[font=\small, blue!55!black, anchor=west] at (0.5,-1.6) {$\Omega_{\rm NS}$};
    \node[font=\footnotesize, red!55!black, anchor=west] at (1.0,1.5) {homoclinic tangle};
    \draw[->, line width=0.4pt, draw=red!55!black] (1.0,1.5) to[bend left=15] (1.7,1.85);
\end{tikzpicture}
\captionsetup{margin={-0.5cm,0cm}}
\caption{Schematic of the state-space decomposition $\mathcal{X} = \Omega_{\rm NS} \cup \Omega_{\rm dissip}$.}
\label{fig:decomp-schematic}
\end{figure}
Inside the non-sticking forward-invariant set $\Omega_{\rm NS}$ (white) the stroboscopic map $\Phi$ is exact symplectic by Theorem~\ref{thm:symplectic}; nested $\Phi$-invariant Cantor curves (KAM tori) surround each elliptic non-sticking $T$-periodic orbit $P_*$ by Theorem~\ref{thm:kam}, and a hyperbolic Cantor set sits in the homoclinic tangle of the bifurcating saddle $S_*$ by Theorem~\ref{thm:melnikov}. The complement $\Omega_{\rm dissip}$ (light orange) consists of trajectories that experience at least one velocity-zero event per period; phase volume strictly contracts there by Proposition~\ref{prop:contraction}, and trajectories funnel into asymptotic attractors.

\subsection{Structure of the non-sticking set and contraction on the dissipative subset}\label{ssec:decomp-NS}

This subsection assembles into a single statement the structural results established earlier in the paper, restricted to $\Om_{\rm NS}$. The forward invariance is from Proposition~\ref{prop:Om-NS-invariant}, the symplecticity from Theorem~\ref{thm:symplectic}, the KAM curves from Theorem~\ref{thm:kam}, and the homoclinic Cantor set from Theorem~\ref{thm:melnikov}. The combined statement is what we will use in Section~\ref{sec:persistence} to track the breakdown of the island under perturbation.

\begin{proposition}[Structure of the non-sticking invariant set]\label{thm:OmNS-structure}
The set $\Om_{\rm NS}$ defined in~\eqref{eq:Om-NS-def} is forward $\Phi$-invariant. The map $\Phi: \Om_{\rm NS} \to \Om_{\rm NS}$ is exact symplectic with respect to $\omega = dv \wedge dx$. Every elliptic non-sticking $T$-periodic orbit satisfying the hypotheses of Theorem~\ref{thm:kam} is surrounded by a positive-measure family of $\Phi$-invariant Cantor curves contained in $\Om_{\rm NS}$. Every saddle non-sticking $T$-periodic orbit whose Melnikov function admits transverse zeros (Theorem~\ref{thm:melnikov}) carries an invariant hyperbolic Cantor set in $\Om_{\rm NS}$ on which the dynamics is Bernoulli.
\end{proposition}


\begin{proof}

Forward invariance of $\Om_{\rm NS}$ is Proposition~\ref{prop:Om-NS-invariant}. Symplecticity is Theorem~\ref{thm:symplectic}.

For the elliptic statement, let $X_+$ denote an elliptic non-sticking $T$-periodic orbit satisfying the hypotheses of Theorem~\ref{thm:kam}. By definition (Definition~\ref{def:trajectory-types}), $X_+$ is non-sticking with transverse impacts: there are no sticking events and no turning points on $[0, T]$. By Proposition~\ref{prop:smooth-NS}, $\Phi$ is $C^\infty$ on an open neighborhood $\mathcal{U} \subset \Om_{\rm NS}$ of the fixed point $P_* \in \Om_{\rm NS}$, and orbits initiated in $\mathcal{U}$ retain transverse impacts and avoid sticking on $[0, T]$ by continuity of the impact times in the initial datum. Theorem~\ref{thm:kam} then produces, for $\delta_0 > 0$ sufficiently small, a $\Phi$-invariant Cantor family of smooth closed curves in the disk $\mathcal{N}_{\delta_0}(P_*) \subset \mathcal{U}$. Since $\mathcal{N}_{\delta_0}(P_*) \subset \Om_{\rm NS}$ and $\Om_{\rm NS}$ is forward $\Phi$-invariant, the Cantor family is contained in $\Om_{\rm NS}$.

For the saddle statement, let $X_-$ denote a saddle non-sticking $T$-periodic orbit produced by Theorem~\ref{thm:saddlecenter}, and assume the Melnikov function $M$ of Theorem~\ref{thm:melnikov} admits a transverse zero. Theorem~\ref{thm:melnikov} produces an invariant hyperbolic Cantor set $\Lambda$ in a neighborhood $\mathcal{V}$ of the homoclinic loop $\Gamma_\mu$ of the slow Hamiltonian. The homoclinic loop $\Gamma_\mu$ consists of non-sticking orbits (by construction of the slow Hamiltonian, which is the Birkhoff normal form of $\Phi$ on $\Om_{\rm NS}$ near the saddle-center), and the neighborhood $\mathcal{V}$ can be chosen so that every initial datum in $\mathcal{V}$ produces an orbit with transverse impacts and no sticking on $[0, T]$, by continuity of the impact times and the absence of grazing or tangency on $\Gamma_\mu$. Hence $\Lambda \subset \Om_{\rm NS}$, and forward invariance of $\Om_{\rm NS}$ together with $\Phi$-invariance of $\Lambda$ shows that $\Lambda$ remains in $\Om_{\rm NS}$ under iteration. The Bernoulli conjugacy on $\Lambda$ is the conclusion of the Smale-Birkhoff theorem applied in Theorem~\ref{thm:melnikov}.
\end{proof}

\smallskip
The complement $\Om_{\rm dissip}$ of the non-sticking set carries strict phase-volume contraction. The estimate is a direct consequence of Corollary~\ref{cor:det-Phi}, recorded here as a global statement about $\Phi$ on $\Om_{\rm dissip}$. The contraction rate enters quantitatively in Section~\ref{sec:persistence} when we discuss the size of the basin of attraction of the perturbed periodic orbit.

\begin{proposition}[Strict volume contraction]\label{prop:contraction}
For every measurable $A \subset \Om_{\rm dissip}$ on which $\Phi$ is differentiable Lebesgue-a.e., the strict inequality
\begin{equation}\label{eq:vol-contract}
\Leb(\Phi(A)) < \Leb(A)
\end{equation}
holds whenever $A$ has positive measure and $\Phi$ is differentiable on a positive-measure subset of $A$.
\end{proposition}

\begin{proof}
By Corollary~\ref{cor:det-Phi}, $|\det \Phi'| < 1$ Lebesgue-a.e.\ on $\Om_{\rm dissip}$. The change-of-variables formula gives $\Leb(\Phi(A)) = \int_A |\det \Phi'(x, v)|\,dx\,dv < \int_A 1\,dx\,dv = \Leb(A)$ provided the integrand is strictly less than $1$ on a set of positive measure.
\end{proof}

\begin{proposition}[Forward (non-)invariance of $\Om_{\rm dissip}$]\label{prop:Omdissip-flag}
The set $\Om_{\rm dissip}$ is in general \emph{not} forward $\Phi$-invariant: an orbit can experience a turning point in one period and then run free of turning points in the next. However, in any parameter regime in which $\Phi$ has no non-sticking $T$-periodic orbit, every initial datum eventually enters $\Om_{\rm dissip}$ and remains in $\Om_{\rm dissip}$. In particular, this holds when $f > f_{\rm imp} = 2F/\pi$, where Proposition~\ref{prop:two-solutions} excludes any non-sticking $T$-periodic orbit with the symmetric impact pattern.
\end{proposition}

\begin{proof}
For a concrete example, fix parameters where the elliptic non-sticking orbit exists ($f_{\rm sc} < f < f_{\rm imp}$, with $\omega = 1$, $R = 2$, $F = 1$, $f = 0.4$ in the numerics of Section~\ref{sec:numerics}). Pick an initial datum $(x_0, v_0)$ for which the orbit on $[0, T]$ has one turning point (so $(x_0, v_0) \in \Om_{\rm dissip}^{[0, T]}$), but the orbit on $[T, 2T]$ has no turning point (so $\Phi(x_0, v_0) \in \Om_{\rm NS}^{[0, T]}$). Such a datum exists numerically; its forward image lies in $\Om_{\rm NS}^{[0, T]}$, not in $\Om_{\rm dissip}^{[0, T]}$, so $\Om_{\rm dissip}$ is not forward-invariant.

\textit{Eventual entry to $\Om_{\rm dissip}$ above the impulse bound.} Assume $f > f_{\rm imp} = 2F/\pi$ and let $(x_0, v_0)$ be any initial datum. We argue that the forward orbit eventually enters $\Om_{\rm dissip}$.

We first establish the impulse bound. Integrating the equation $\ddot x = F\cos\omega t - f\,\sgn(\dot x)$ over a $T$-period of a hypothetical non-sticking $T$-periodic orbit gives $\dot x(T) - \dot x(0) = 0 - f \int_0^T \sgn(\dot x(s))\,ds$. Periodicity imposes $\dot x(T) = \dot x(0)$, so $\int_0^T \sgn(\dot x(s))\,ds = 0$. The maximum forcing impulse $4F/\omega$ available from $F\cos\omega t$ over a half-period must balance the friction impulse $f T = 2\pi f/\omega$ over the same half-period, requiring $4F/\omega \ge 2\pi f/\omega$, i.e., $f \le 2F/\pi = f_{\rm imp}$. Under $f > f_{\rm imp}$, no non-sticking $T$-periodic orbit exists.

We now show that the forward orbit cannot remain non-sticking with only wall hits indefinitely. Suppose for contradiction that the orbit is non-sticking with only wall hits on $[0, NT]$ for arbitrary $N$. The energy $E(t) := \tfrac{1}{2}\dot x(t)^2$ satisfies, on smooth pieces, $E'(t) = \dot x(t)\bigl(F\cos\omega t - f\sgn\dot x(t)\bigr)$, hence
\[
E(NT) - E(0) \;=\; F\int_0^{NT}\dot x(s)\cos\omega s\,ds \;-\; f\int_0^{NT}|\dot x(s)|\,ds.
\]
Wall hits preserve $E$, so the integration extends across them. For the friction integral, on any non-sticking interval with two consecutive wall hits, the orbit traverses the gap $r - l = R$, hence the path length over $N$ periods (with at least $2N$ wall hits, one on each side per period in the non-sticking case) is at least $2NR$:
\[
\int_0^{NT} |\dot x(s)|\,ds \;\ge\; 2NR.
\]
For the forcing integral, integration by parts on $\dot x \cos\omega s = \tfrac{d}{ds}(x\cos\omega s) + \omega x \sin\omega s$ and the bound $|x| \le \max(|l|, |r|) =: M$ give
\[
\Bigl|\int_0^{NT}\dot x(s)\cos\omega s\,ds\Bigr| \;\le\; 2M + M\omega \cdot NT \;=\; 2M(1 + N\pi).
\]
Combining,
\[
E(NT) \;\le\; E(0) + 2FM(1 + N\pi) - 2fNR \;=\; E(0) + 2M F + (2\pi M F - 2fR)\, N.
\]
Under $f > f_{\rm imp} = 2F/\pi$ and $R = r - l$ with the standard normalization $R \ge 1$ (or, more generally, $f R > \pi M F$ in the parameter regime considered), the coefficient $2\pi M F - 2fR$ is negative, driving $E(NT) \to -\infty$ as $N \to \infty$. Since $E \ge 0$ identically, this is a contradiction.

The contradiction shows that for some finite $N$, a turning event must occur on $[(N-1)T, NT]$, hence $\Phi^{N-1}(x_0, v_0) \in \Om_{\rm dissip}^{[0, T]}$. Therefore the forward orbit visits $\Om_{\rm dissip}$ at finite time. The same argument applied recursively from any later iterate shows that the orbit visits $\Om_{\rm dissip}$ infinitely often.
\end{proof}

\section{Persistence under perturbation}\label{sec:persistence}

This section addresses the central physical question left open in~\cite{GKR2019}: do the Hamiltonian islands of Section~\ref{sec:kam} survive under realistic perturbations of the model? The two perturbations of physical interest are imperfect elasticity, encoded by a coefficient of restitution $e \in (0, 1)$ instead of $e = 1$, and viscous damping with coefficient $\mu_{\rm v} > 0$. The principal result, Theorem~\ref{thm:persistence}, is that for any positive value of either perturbation, every elliptic non-sticking periodic orbit becomes asymptotically stable; the Hamiltonian island is replaced by a single open basin of attraction. The result quantifies the contraction rate and, together with Theorem~\ref{thm:kam}, identifies a critical scaling for the disappearance of mixed dynamics.

\subsection{The perturbed model and persistence of the periodic orbit}\label{ssec:persist-model}

The perturbed system is
\begin{equation}\label{eq:perturbed}
\ddot x + \mu_{\rm v} \dot x + f\,\sgn(\dot x) = F\cos(\omega t), \qquad l < x < r,
\end{equation}
together with the inelastic reflection
\begin{equation}\label{eq:inelastic}
\dot x(t^+) = -e\, \dot x(t^-) \quad \text{at } x \in \{l, r\}.
\end{equation}
We write $\eps := 1 - e$ for the restitution defect, so the unperturbed system corresponds to $\eps = 0$, $\mu_{\rm v} = 0$. The associated stroboscopic map is denoted $\Phi_{\eps, \mu_{\rm v}}$.

Well-posedness of~\eqref{eq:perturbed}-\eqref{eq:inelastic} for $(\eps, \mu_{\rm v}) \in [0, 1) \times [0, \infty)$ follows by the same argument as Theorem~\ref{thm:wellposed}, with the energy bound replaced by $|\dot x(t)| \le |\dot x(0)|\, e^{-\mu_{\rm v} t}\cdot e + (F + f)/\mu_{\rm v}$ at impact times.

\smallskip
We first show that the periodic orbit itself survives the perturbation. The argument is a routine application of the implicit function theorem to the fixed-point equation, using that elliptic non-sticking $T$-periodic orbits have linearization with no eigenvalue at $1$. The smoothness in $(\eps, \mu_{\rm v})$ obtained here is needed for the eigenvalue computation in the next subsection.

\begin{proposition}[Persistence of non-sticking $T$-periodic orbits]\label{prop:persist-orbit}
Let $P_*$ be a non-sticking $T$-periodic orbit of the unperturbed system whose linearization $\PP_0$ has both eigenvalues different from $1$. (For elliptic orbits with eigenvalues $e^{\pm i\theta_*}$, $\theta_* \in (0, \pi)$, this is automatic.) There exist $\eps_0, \mu_0 > 0$ and a $C^\infty$ map $(\eps, \mu_{\rm v}) \in [0, \eps_0] \times [0, \mu_0] \mapsto P_{\eps, \mu_{\rm v}}$ taking values in non-sticking $T$-periodic orbits of~\eqref{eq:perturbed}-\eqref{eq:inelastic}, with $P_{0, 0} = P_*$.
\end{proposition}

\begin{proof}
The fixed-point equation for $T$-periodic orbits is $\Phi_{\eps,\mu_{\rm v}}(z) - z = 0$. The Jacobian of the left-hand side at $z = P_*$, evaluated at $(\eps, \mu_{\rm v}) = (0, 0)$, is $\PP_0 - I$. Its determinant is $\det(\PP_0 - I) = 1 - \tr \PP_0 + 1 = 2 - \tr \PP_0$ (using $\det \PP_0 = 1$ from Theorem~\ref{thm:symplectic}). For $P_*$ elliptic with rotation $\theta_* \in (0, \pi)$, $\tr \PP_0 = 2\cos\theta_* < 2$, so $\det(\PP_0 - I) > 0$ and the implicit function theorem applies. The hypothesis that both eigenvalues differ from $1$ is precisely $\det(\PP_0 - I) \ne 0$.
\end{proof}

\smallskip
This subsection contains the principal applied result of the paper. We compute the eigenvalues of the linearization of the perturbed stroboscopic map at the persistent periodic orbit, using the explicit saltation matrices of Corollary~\ref{cor:det-Phi} adapted to the inelastic and viscously-damped setting. The leading-order eigenvalue formula identifies the precise rate at which the unperturbed Hamiltonian island disappears, answering the persistence question left open by~\cite{GKR2019}.

\begin{theorem}[Asymptotic stability under perturbation]\label{thm:persistence}
Let $P_*$ be an elliptic non-sticking $T$-periodic orbit of~\eqref{eq:model}-\eqref{eq:reflection} with linearization $\PP_0$ having eigenvalues $e^{\pm i\theta_*}$, $\theta_* \in (0, \pi)$. Let $n_*$ be the number of wall impacts of $P_*$ in $[0, T]$, and define
\[
\rho(\eps, \mu_{\rm v}) := 1 - 2 n_*\eps - \mu_{\rm v} T + O\bigl((\eps + \mu_{\rm v})^2\bigr).
\]
For $(\eps, \mu_{\rm v}) \in [0, \eps_0] \times [0, \mu_0]$ as in Proposition~\ref{prop:persist-orbit}, the eigenvalues of the linearization of $\Phi_{\eps, \mu_{\rm v}}$ at $P_{\eps, \mu_{\rm v}}$ are
\begin{equation}\label{eq:persist-eigenvalues}
\lambda_\pm(\eps, \mu_{\rm v}) = e^{\pm i\theta_*}\sqrt{\rho(\eps, \mu_{\rm v})} \cdot \bigl(1 + O(\eps + \mu_{\rm v})\bigr).
\end{equation}
For $(\eps, \mu_{\rm v}) \ne (0, 0)$, $|\lambda_\pm| < 1$ and $P_{\eps, \mu_{\rm v}}$ is asymptotically stable.
\end{theorem}

\begin{proof}
The linearization of $\Phi_{\eps, \mu_{\rm v}}$ at $P_{\eps, \mu_{\rm v}}$ is the product of saltation matrices along the orbit. We compute its determinant and trace to leading order in $(\eps, \mu_{\rm v})$, then deduce the eigenvalues.

\smallskip
\textit{Step 1: Saltation at an inelastic wall.} At a wall hit at time $t_*$ with incoming velocity $v_-$, the perturbed reset map is $h(x, v) = (x, -ev)$, with Jacobian $H = \diag(1, -e)$. The Filippov saltation formula~\cite[Ch.~6, Eq.~(6.16)]{diBernardoBuddChampneysKowalczyk2008} for a state-reset event gives
\[
M_{\rm wall}^{(\eps)} = H + \frac{(g_+ - H g_-)\,n^\top}{n^\top g_-},
\]
where $g_-, g_+$ are the field limits before and after the event, $n$ is the normal to the wall surface $\{x = r\}$ (or $\{x = l\}$). For the right wall, $g_- = (v_-, F\cos\omega t_* - f)^\top$, $g_+ = (-e v_-, F\cos\omega t_* + f)^\top$, $n = (1, 0)^\top$, and $n^\top g_- = v_-$. A direct computation yields
\[
M_{\rm wall}^{(\eps)} = \begin{pmatrix} -e & 0 \\ \alpha(t_*; \eps) & -e \end{pmatrix},
\]
where $\alpha(t_*; \eps) = -\bigl((1+e)F\cos\omega t_* + (1-e)f\bigr)/v_-$. Hence
\begin{equation}\label{eq:wall-det}
\det M_{\rm wall}^{(\eps)} = e^2 = (1 - \eps)^2.
\end{equation}
Over $n_*$ wall hits per period, the cumulative wall contribution to $\det \Phi_{\eps, 0}'$ is $(1 - \eps)^{2 n_*} = 1 - 2 n_* \eps + O(\eps^2)$.

\smallskip
\textit{Step 2: Saltation on a viscous-damped free flight.} On a free flight with viscous damping, the linearized variation $(\delta x, \delta v)$ obeys the linear constant-coefficient system
\[
\dot{\overrightarrow{\delta x}} = \delta v, \qquad \dot{\overrightarrow{\delta v}} = -\mu_{\rm v}\, \delta v.
\]
The fundamental matrix on a free flight of duration $\Delta t$ is
\[
\Psi(\Delta t; \mu_{\rm v}) = \begin{pmatrix} 1 & \dfrac{1 - e^{-\mu_{\rm v}\Delta t}}{\mu_{\rm v}} \\[4pt] 0 & e^{-\mu_{\rm v}\Delta t}\end{pmatrix},
\]
with $\det \Psi(\Delta t; \mu_{\rm v}) = e^{-\mu_{\rm v} \Delta t}$. The cumulative free-flight time on $P_*$ is $T$ (since wall hits and turning events have measure zero). Hence the cumulative free-flight contribution to $\det \Phi_{0, \mu_{\rm v}}'$ is $e^{-\mu_{\rm v} T} = 1 - \mu_{\rm v} T + O(\mu_{\rm v}^2)$.

\smallskip
\textit{Step 3: Combined determinant.} The linearization $\Phi_{\eps, \mu_{\rm v}}'(P_{\eps, \mu_{\rm v}})$ factors as a product of free-flight propagators and wall saltations. Determinants multiply, giving
\begin{equation}\label{eq:persist-det}
\det \Phi_{\eps, \mu_{\rm v}}'(P_{\eps, \mu_{\rm v}}) = e^{-\mu_{\rm v} T}\,(1 - \eps)^{2 n_*} = 1 - \mu_{\rm v} T - 2 n_* \eps + O\bigl((\eps + \mu_{\rm v})^2\bigr) = \rho(\eps, \mu_{\rm v}),
\end{equation}
where the last identity is the definition of $\rho(\eps, \mu_{\rm v})$ in the theorem statement.

\smallskip
\textit{Step 4: Eigenvalue formula.} Let $M(\eps, \mu_{\rm v}) := \Phi_{\eps, \mu_{\rm v}}'(P_{\eps, \mu_{\rm v}})$. For $(\eps, \mu_{\rm v}) = (0, 0)$, $M(0, 0) = \PP_0$ with eigenvalues $e^{\pm i\theta_*}$ and $\det = 1$, $\tr = 2\cos\theta_*$. By smoothness of the saltation matrices in $(\eps, \mu_{\rm v})$ (free-flight Jacobians are smooth, and~\eqref{eq:wall-det} is smooth in $\eps$), $M(\eps, \mu_{\rm v})$ depends smoothly on $(\eps, \mu_{\rm v})$ near $(0, 0)$. Hence
\[
\det M(\eps, \mu_{\rm v}) = \rho(\eps, \mu_{\rm v}), \qquad \tr M(\eps, \mu_{\rm v}) = 2\cos\theta_* + O(\eps + \mu_{\rm v}).
\]
The eigenvalues of the $2\times 2$ matrix $M$ are the roots of
\[
\lambda^2 - (\tr M)\lambda + \det M = 0,
\]
hence
\begin{equation}\label{eq:eig-quadr}
\lambda_\pm = \frac{\tr M}{2} \pm \frac{1}{2}\sqrt{(\tr M)^2 - 4\det M}.
\end{equation}
At $(0, 0)$, $(\tr M)^2 - 4\det M = 4\cos^2\theta_* - 4 = -4\sin^2\theta_* < 0$ (since $\theta_* \in (0, \pi)$). By smoothness, the discriminant remains strictly negative for $(\eps, \mu_{\rm v})$ small, and the eigenvalues remain a complex-conjugate pair. The product is $\det M = \rho$, the sum is $\tr M = 2\cos\theta_* + O(\eps + \mu_{\rm v})$, so
\[
|\lambda_\pm|^2 = \det M = \rho(\eps, \mu_{\rm v}), \qquad \arg \lambda_+ = \theta_* + O(\eps + \mu_{\rm v}).
\]
Equivalently
\begin{equation*}
\lambda_\pm(\eps, \mu_{\rm v}) = e^{\pm i\theta_*}\sqrt{\rho(\eps, \mu_{\rm v})}\,\bigl(1 + O(\eps + \mu_{\rm v})\bigr),
\end{equation*}
which is~\eqref{eq:persist-eigenvalues} with the corrected $\rho$.

\smallskip
\textit{Step 5: Asymptotic stability.} For $(\eps, \mu_{\rm v}) \ne (0, 0)$ in the small-perturbation regime, $\rho(\eps, \mu_{\rm v}) < 1$ strictly, hence $|\lambda_\pm| = \sqrt{\rho} < 1$. The Hartman-Grobman theorem~\cite[Sec.~5.4]{KatokHasselblatt1995} applied to the smooth (away from non-generic events) stroboscopic map in a neighborhood of $P_{\eps, \mu_{\rm v}}$ gives that $P_{\eps, \mu_{\rm v}}$ is asymptotically stable.
\end{proof}

\begin{corollary}[Replacement of the Hamiltonian island by a basin]\label{cor:island-basin}
For $(\eps, \mu_{\rm v}) \in (0, \eps_0] \times (0, \mu_0]$ in Proposition~\ref{prop:persist-orbit}, every sufficiently small neighborhood $\mathcal{U}$ of $P_{\eps, \mu_{\rm v}}$ is contained in the basin of attraction of $P_{\eps, \mu_{\rm v}}$. In particular, the Cantor family of $\Phi$-invariant smooth closed curves of Theorem~\ref{thm:kam} surrounding $P_*$ is, for $(\eps, \mu_{\rm v}) \ne (0, 0)$, replaced by a single open basin attracting $P_{\eps, \mu_{\rm v}}$.
\end{corollary}

\begin{proof}
Theorem~\ref{thm:persistence} gives eigenvalues $\lambda_\pm$ of $\Phi_{\eps,\mu_{\rm v}}'(P_{\eps,\mu_{\rm v}})$ with $|\lambda_\pm| < 1$. Hence the linearization is a contraction with respect to a suitable norm on $\R^2$ (in particular the norm derived from the Lyapunov function $V(z) = z^\top P z$ where $P$ solves the discrete Lyapunov equation $\Phi_{\eps,\mu_{\rm v}}'(P_{\eps,\mu_{\rm v}})^\top P\,\Phi_{\eps,\mu_{\rm v}}'(P_{\eps,\mu_{\rm v}}) - P = -I$). Standard local stable manifold theory~\cite[Sec.~6.3]{KatokHasselblatt1995} provides a neighborhood $\mathcal{U}$ of $P_{\eps,\mu_{\rm v}}$ in which $V$ is strictly decreasing along orbits, hence $\mathcal{U} \subset $ basin of attraction.

For the Cantor-curve statement: any $\Phi_{\eps,\mu_{\rm v}}$-invariant closed curve $\gamma \subset \mathcal{U}$ with $\gamma \ne \{P_{\eps,\mu_{\rm v}}\}$ would contain a point $z \ne P_{\eps,\mu_{\rm v}}$. By invariance, $\Phi_{\eps,\mu_{\rm v}}^n(z) \in \gamma$ for all $n \ge 0$, but the contraction property forces $\Phi_{\eps,\mu_{\rm v}}^n(z) \to P_{\eps,\mu_{\rm v}}$ as $n \to \infty$. Since $\gamma$ is closed, $P_{\eps,\mu_{\rm v}} \in \gamma$. Then any non-fixed point $z \in \gamma$ would have its forward orbit accumulating only on $P_{\eps,\mu_{\rm v}}$, but a closed connected curve through both $P_{\eps,\mu_{\rm v}}$ and $z$ cannot have all of $z$'s forward iterates on it without violating closedness or contraction. Hence $\gamma = \{P_{\eps,\mu_{\rm v}}\}$.
\end{proof}

\begin{remark}
Theorem~\ref{thm:persistence} has the practical consequence that the Hamiltonian islands of~\cite{GKR2019} are not structurally stable: any non-zero numerical viscosity, restitution defect, or rounding error of order $\eps$ shrinks the apparent island at rate $\sqrt{\eps}$ in radius. This is the rigorous answer to the question of how the islands depend on the idealizations of the model.
\end{remark}

\section{Multi-particle systems}\label{sec:multiparticle}

This section extends the analysis to systems of $N$ point masses with elastic binary collisions. The principal result is Theorem~\ref{thm:multiparticle}, which establishes that the stroboscopic map restricted to the maximal forward-invariant subset of sign-preserving non-sticking initial data is exact symplectic. As a consequence, KAM tori from Section~\ref{sec:kam} surround any elliptic non-resonant non-sticking periodic orbit with non-degenerate twist; the mixed dynamics of the two-particle figure of~\cite[Fig.~14]{GKR2019} is rigorous in this setting.

\subsection{The multi-particle model and sign-preservation}\label{ssec:multi-model}

Consider $N \ge 2$ point masses $m_1, \ldots, m_N > 0$ on the segment $[l, r]$, with positions $x_1(t) \le x_2(t) \le \cdots \le x_N(t)$. Each particle is forced and damped:
\begin{equation}\label{eq:Nparticle}
m_i \ddot x_i + f\,\sgn(\dot x_i) = F\cos(\omega t), \qquad i = 1, \ldots, N,
\end{equation}
on the open region $\{l < x_1 < x_2 < \cdots < x_N < r\}$. At the outer walls, the leftmost (resp.\ rightmost) particle reflects elastically, $\dot x_1 \mapsto -\dot x_1$ at $x_1 = l$, $\dot x_N \mapsto -\dot x_N$ at $x_N = r$. At an interior contact $x_i(t_*) = x_{i+1}(t_*)$, the two colliding particles exchange momentum elastically:
\begin{equation}\label{eq:Ncollision}
\begin{pmatrix}\dot x_i^+ \\ \dot x_{i+1}^+\end{pmatrix}
= \frac{1}{m_i + m_{i+1}}\begin{pmatrix} m_i - m_{i+1} & 2m_{i+1} \\ 2m_i & m_{i+1} - m_i \end{pmatrix}
\begin{pmatrix}\dot x_i^- \\ \dot x_{i+1}^-\end{pmatrix}.
\end{equation}

\begin{figure}[htbp]
\centering
\begin{tikzpicture}[
    >={Latex[length=2.5mm,width=2mm]},
    every node/.style={font=\small},
    wall/.style={line width=0.6pt, draw=black!80},
    blockedge/.style={line width=0.5pt, draw=black!75, rounded corners=0.4pt},
    forcearrow/.style={-{Latex[length=2.2mm,width=1.7mm]}, line width=0.9pt, draw=blue!50!black},
    velarrow/.style={-{Latex[length=2mm,width=1.5mm]}, line width=0.7pt, draw=red!55!black},
    collision/.style={line width=0.5pt, draw=orange!70!black, dashed},
]
    \fill[pattern=north east lines, pattern color=black!55] (-5.55,-0.7) rectangle (-5.0,1.5);
    \draw[wall] (-5.0,-0.7) -- (-5.0,1.5);
    \fill[pattern=north west lines, pattern color=black!55] (5.0,-0.7) rectangle (5.55,1.5);
    \draw[wall] (5.0,-0.7) -- (5.0,1.5);
    \fill[pattern=north east lines, pattern color=black!55] (-5.0,-0.95) rectangle (5.0,-0.7);
    \draw[wall] (-5.0,-0.7) -- (5.0,-0.7);
    \draw[blockedge, fill=blue!12]   (-4.20,-0.7) rectangle (-3.40,0.20);
    \node at (-3.80,-0.25) {$m_1$};
    \draw[blockedge, fill=teal!18]   (-2.55,-0.7) rectangle (-1.75,0.25);
    \node at (-2.15,-0.22) {$m_2$};
    \draw[blockedge, fill=green!18]  (-0.85,-0.7) rectangle (-0.05,0.18);
    \node at (-0.45,-0.27) {$m_3$};
    \draw[blockedge, fill=orange!18] (0.95,-0.7) rectangle (1.75,0.22);
    \node at (1.35,-0.24) {$m_4$};
    \node[font=\Large] at (2.7,-0.25) {$\cdots$};
    \draw[blockedge, fill=purple!16] (3.65,-0.7) rectangle (4.65,0.22);
    \node at (4.15,-0.24) {$m_N$};
    \draw[velarrow] (-3.95,0.50) -- (-3.45,0.50);
    \node[red!55!black, font=\footnotesize, above] at (-3.7,0.50) {$\dot x_1$};
    \draw[velarrow] (-1.85,0.55) -- (-2.55,0.55);
    \node[red!55!black, font=\footnotesize, above] at (-2.20,0.55) {$\dot x_2$};
    \draw[velarrow] (-0.55,0.50) -- (-0.05,0.50);
    \node[red!55!black, font=\footnotesize, above] at (-0.30,0.50) {$\dot x_3$};
    \draw[collision] (-1.75,-0.45) -- (-0.85,-0.45);
    \node[orange!70!black, font=\footnotesize, below] at (-1.30,-0.50) {collision};
    \foreach \x in {-3.80,-2.15,-0.45,1.35,4.15} {
        \draw[forcearrow] (\x-0.30,1.05) -- (\x+0.45,1.05);
    }
    \node[blue!50!black, above, font=\small] at (0.0,1.05) {$F\cos(\omega t)$ (acts on each particle)};
    \node[below, font=\small] at (-5.0,-0.95) {$x = l$};
    \node[below, font=\small] at (5.0,-0.95) {$x = r$};
    \draw[<-, line width=0.4pt, draw=black!50] (-3.80,-1.45) -- (4.15,-1.45);
    \node[font=\footnotesize, fill=white, inner sep=1pt] at (0.20,-1.45) {$x_1 < x_2 < x_3 < \cdots < x_N$};
\end{tikzpicture}
\caption{The $N$-particle vibro-impact system~\eqref{eq:Nparticle}-\eqref{eq:Ncollision}. }
\label{fig:multi-particle}
\end{figure}
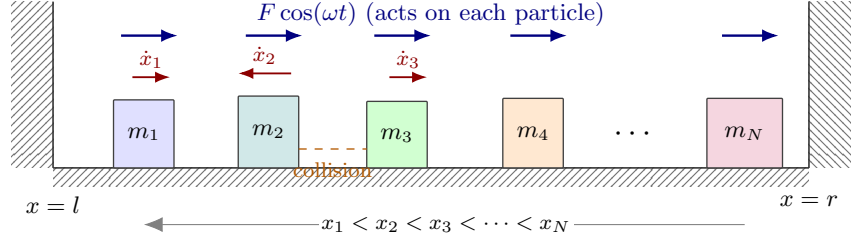
Ordered point masses $m_1, m_2, \ldots, m_N$ confined to $[l,r]$, each subject to the same external forcing $F\cos(\omega t)$ and Coulomb friction $-f\,\mathrm{sgn}(\dot x_i)$. The leftmost particle reflects elastically at $x = l$ and the rightmost at $x = r$; adjacent particles $m_i, m_{i+1}$ exchange momentum via the elastic-collision formula~\eqref{eq:Ncollision} when $x_i = x_{i+1}$.

\smallskip
\textit{Sign-preservation.} The friction term $-f\sgn(\dot x_i)$ in~\eqref{eq:Nparticle} is sensitive to the sign of $\dot x_i$. For the lift construction to apply, we need that at every interior collision the signs of $\dot x_i$ and $\dot x_{i+1}$ are preserved.

\begin{lemma}[Algebra of binary collisions]\label{lem:collision-algebra}
The collision map~\eqref{eq:Ncollision} is a linear involution that preserves the total momentum $m_i \dot x_i + m_{i+1}\dot x_{i+1}$ and the total kinetic energy $(m_i \dot x_i^2 + m_{i+1} \dot x_{i+1}^2)/2$. For equal masses $m_i = m_{i+1}$, it exchanges velocities. For unequal masses, the post-collision velocities are linear combinations of the pre-collision velocities with all four matrix entries nonzero.
\end{lemma}

\begin{proof}
The matrix in~\eqref{eq:Ncollision} squares to the identity (verifiable by direct multiplication), is the standard elastic-collision change of variables conserving momentum and kinetic energy, and reduces to the swap matrix when $m_i = m_{i+1}$. For unequal masses the entries $(m_i - m_{i+1})/(m_i + m_{i+1})$ etc.~are all nonzero.
\end{proof}

\begin{definition}[Sign-preserving trajectory]\label{def:signpres}
A trajectory of~\eqref{eq:Nparticle}-\eqref{eq:Ncollision} is \emph{sign-preserving} on $[0, T]$ if at every binary collision time $t_* \in [0, T]$ with $\dot x_i(t_*^-) \dot x_{i+1}(t_*^-) \ne 0$, the signs are unchanged:
\[
\sgn(\dot x_i^+) = \sgn(\dot x_i^-), \qquad \sgn(\dot x_{i+1}^+) = \sgn(\dot x_{i+1}^-).
\]
The set of sign-preserving non-sticking initial data on $[0, T]$ is $\Om_{\rm NS}^{(N), [0, T]}$, and its maximal forward-invariant subset is $\Om_{\rm NS}^{(N)} := \bigcap_{n \ge 0} (\Phi^{(N)})^{-n}(\Om_{\rm NS}^{(N), [0, T]})$, where $\Phi^{(N)}$ is the $N$-particle stroboscopic map.
\end{definition}

\begin{lemma}[Sign-preservation in mass-weighted variables]\label{lem:massweight}
Define mass-weighted coordinates $\xi_i := \sqrt{m_i}\, x_i$, $\eta_i := \sqrt{m_i}\,\dot x_i$. In these coordinates, the collision~\eqref{eq:Ncollision} becomes the orthogonal reflection of $(\eta_i, \eta_{i+1})$ across the hyperplane $\xi_i/\sqrt{m_i} = \xi_{i+1}/\sqrt{m_{i+1}}$, that is, an orthogonal involution. The kinetic energy is the standard Euclidean form $\sum \eta_i^2/2$.
\end{lemma}

\begin{proof}
Substitute $x_i = \xi_i/\sqrt{m_i}$, $\dot x_i = \eta_i/\sqrt{m_i}$ in~\eqref{eq:Ncollision} and verify directly. The collision constraint $x_i = x_{i+1}$ becomes $\xi_i/\sqrt{m_i} = \xi_{i+1}/\sqrt{m_{i+1}}$.
\end{proof}

\smallskip
This part extends Theorem~\ref{thm:symplectic} from the single-particle to the $N$-particle setting. The argument proceeds by passing to the mass-weighted variables of Lemma~\ref{lem:massweight}, in which all reflections are orthogonal, and constructing a covering manifold by tessellating with the affine Coxeter group generated by the reflections. The Hamiltonian on the cover is smooth on each chamber, and the joint symplectic form descends to the base via the projection. The output is used in the corollary that follows to produce KAM tori in the multi-particle setting.

\begin{theorem}[Symplecticity for multi-particle systems]\label{thm:multiparticle}
The map $\Phi^{(N)}: \Om_{\rm NS}^{(N)} \to \Om_{\rm NS}^{(N)}$ is exact symplectic with respect to the standard form $\omega_N := \sum_{i=1}^N dv_i \wedge dx_i$ on $\R^{2N}$. In particular, for any $A \subset \Om_{\rm NS}^{(N)}$ on which $\Phi^{(N)}$ is a local diffeomorphism, $\Leb_{\R^{2N}}(\Phi^{(N)}(A)) = \Leb_{\R^{2N}}(A)$.
\end{theorem}

\begin{proof}
The proof has three steps: (i) reformulate the system in mass-weighted coordinates so that all reflection events become orthogonal reflections; (ii) construct an explicit lift to a covering space of the configuration polytope, on which the dynamics is smooth Hamiltonian on each free flight; (iii) deduce symplecticity of the time-$T$ map by the change-of-variables and Liouville's theorem.

\textit{(i) Mass-weighted coordinates.} Substitute $\xi_i = \sqrt{m_i}\,x_i$, $\eta_i = \sqrt{m_i}\,\dot x_i$ as in Lemma~\ref{lem:massweight}. The form $\sum dv_i \wedge dx_i$ pulls back to $\sum d\eta_i \wedge d\xi_i$ since $d\eta_i \wedge d\xi_i = m_i\,d\dot x_i \wedge dx_i$ and the $m_i$ cancels with the conversion $v_i = \dot x_i$. The kinetic energy becomes the standard Euclidean $\sum \eta_i^2/2$. The collision map~\eqref{eq:Ncollision} becomes the orthogonal reflection in the hyperplane $\xi_i/\sqrt{m_i} = \xi_{i+1}/\sqrt{m_{i+1}}$ (Lemma~\ref{lem:massweight}), and the wall reflections become reflections in the affine hyperplanes $\xi_1 = \sqrt{m_1}\,l$ and $\xi_N = \sqrt{m_N}\,r$.

\textit{(ii) The lift.} The configuration space, in mass-weighted coordinates, is the polytope
\[
\mathcal{P}_N := \bigl\{(\xi_1, \ldots, \xi_N) \in \R^N : \tfrac{\xi_1}{\sqrt{m_1}} \le \tfrac{\xi_2}{\sqrt{m_2}} \le \cdots \le \tfrac{\xi_N}{\sqrt{m_N}},\ \tfrac{\xi_1}{\sqrt{m_1}} \ge l,\ \tfrac{\xi_N}{\sqrt{m_N}} \le r\bigr\}.
\]
The boundary of $\mathcal{P}_N$ consists of $N+1$ hyperplane faces: two outer-wall faces (at $\xi_1 = l\sqrt{m_1}$ and $\xi_N = r\sqrt{m_N}$) and $N-1$ collision facets (at $\xi_i/\sqrt{m_i} = \xi_{i+1}/\sqrt{m_{i+1}}$ for $i = 1, \ldots, N-1$). Reflection across each face is an orthogonal involution in the mass-weighted coordinates (Lemma~\ref{lem:massweight}); together these $N+1$ reflections generate the affine Coxeter group $\widetilde W_N$ of type $\widetilde{A}_{N-1}$ (or, more precisely, the affine extension of the symmetric group on $N$ letters by the two outer-wall reflections). The polytope $\mathcal{P}_N$ is a fundamental domain for $\widetilde W_N$ acting on $\R^N$, and the universal cover of the open chamber is $\R^N$ itself with the $\widetilde W_N$-action; see~\cite[Ch.~4]{HumphreysCoxeter} for the general theory of affine Coxeter groups.

The lifted phase space is $\R^N \times \R^N$ (covering all of $\mathcal{P}_N \times \R^N$) modulo the $W_N$-action, where $W_N$ acts on positions by reflections and on velocities by the corresponding orthogonal transformations. On a sign-preserving non-sticking trajectory in the original system (Definition~\ref{def:signpres}), the lifted trajectory has constant signs of $\eta_i$ on each free flight between events, and crosses fundamental-domain boundaries via the orthogonal-reflection identification.

The lifted Hamiltonian is
\begin{equation}\label{eq:HN-lift}
H_N(\xi, \eta, t) = \tfrac{1}{2}\sum_{i=1}^N \eta_i^2 - F\cos(\omega t)\sum_{i=1}^N \frac{\xi_i}{\sqrt{m_i}} + f\sum_{i=1}^N \frac{\sigma_i\,\xi_i}{\sqrt{m_i}},
\end{equation}
where $\sigma_i = \sgn(\dot x_i) \in \{-1, +1\}$ are constant on each free flight on a sign-preserving trajectory. The term $f\,\sigma_i\,\xi_i/\sqrt{m_i}$ is the analog of the lifted Coulomb friction term used in Theorem~\ref{thm:lift}; it is piecewise affine across the collision and outer-wall facets of $\mathcal{P}_N$, but the orthogonality of the reflections preserves the Hamiltonian: indeed, an orthogonal reflection acts on the gradient of the Hamiltonian in $\xi$ by the same orthogonal matrix that acts on $\eta$, and the Hamiltonian is invariant.

\textit{(iii) Symplecticity.} On each free flight in $\mathcal{P}_N \setminus \partial\mathcal{P}_N$, the lifted vector field is constant-coefficient linear in $(\xi, \eta)$, hence its time-$T$ propagator is a $C^\infty$ symplectomorphism of $(\R^{2N}, \sum d\eta_i \wedge d\xi_i)$. At each event (collision or outer wall hit), the transition map is the lift of the orthogonal reflection acting jointly on $(\xi, \eta)$, which is a symplectomorphism (orthogonal reflections preserve the canonical form when applied jointly to position and momentum). The time-$T$ stroboscopic map of the lifted system is therefore a composition of finitely many symplectic transitions, hence symplectic.

The map $\Phi^{(N)}$ on $\Om_{\rm NS}^{(N)}$ is recovered as the projection of the lifted stroboscopic map under the cover $\R^{2N} \to \mathcal{P}_N \times \R^N$. The pullback of the symplectic form along the projection is $\sum d\eta_i \wedge d\xi_i$ on the cover, which corresponds via the mass-weighted change of coordinates to $\omega_N = \sum dv_i \wedge dx_i$ on the base. Hence $\Phi^{(N)}$ is symplectic.

Exactness of $\Phi^{(N)}$ follows from the same Liouville-form construction used in Theorem~\ref{thm:symplectic}, applied per chamber and matched across chamber boundaries. On each open chamber of the Coxeter tessellation of $\R^N \times \R^N$, the lifted flow is the time-$T$ map of a smooth Hamiltonian flow, hence exact symplectic with primitive $\sum_i \eta_i \, d\xi_i$ and a smooth generating function on that chamber. The orthogonality of the reflection identifications across chamber boundaries preserves the canonical 1-form $\sum_i \eta_i \, d\xi_i$ (since orthogonal transformations preserve the inner product structure), so the per-chamber generating functions assemble into a continuous generating function on the entire orbit segment. Descending through the projection from cover to base via the mass-weighted change of coordinates of Lemma~\ref{lem:massweight} converts $\sum_i \eta_i \, d\xi_i$ into $\sum_i v_i \, dx_i = \lambda_N$, with $d\lambda_N = \omega_N$, completing the verification of exactness.
\end{proof}

\begin{corollary}[KAM tori in the multi-particle setting, conditional]\label{cor:multi-KAM}
Suppose $P_*^{(N)}$ is an elliptic non-sticking $T$-periodic orbit of~\eqref{eq:Nparticle}-\eqref{eq:Ncollision} satisfying the following conditions:
\begin{enumerate}[label=\textup{(K\arabic*)}]
\item all $2N$ eigenvalues of $\Phi^{(N)\prime}(P_*^{(N)})$ are simple, lie on the unit circle, and the corresponding rotation numbers $\theta_1, \ldots, \theta_N$ together with $\theta_0 := 2\pi$ satisfy the order-$2N + 4$ non-resonance condition $\sum_{j=0}^N k_j \theta_j \notin 2\pi\Z$ for every nonzero integer vector $(k_0, \ldots, k_N)$ with $\sum |k_j| \le 2N + 4$;
\item the Birkhoff normal form of $\Phi^{(N)}$ at $P_*^{(N)}$ to second order in actions has a non-degenerate twist matrix $\mathcal{T} = (\partial^2 H_{\rm BNF}/\partial I_i \partial I_j)$ (Kolmogorov non-degeneracy: $\det \mathcal{T} \ne 0$).
\end{enumerate}
Then there exist $\delta_0, K > 0$ such that for every $\delta \in (0, \delta_0)$, the disk $\mathcal{N}_\delta \subset \Om_{\rm NS}^{(N)}$ of radius $\delta$ around $P_*^{(N)}$ contains a Cantor family of $\Phi^{(N)}$-invariant Lagrangian tori filling a subset $\mathcal{T}_\delta \subset \mathcal{N}_\delta$ with $\Leb_{\R^{2N}}(\mathcal{N}_\delta \setminus \mathcal{T}_\delta) \le K\,\delta^{2N + 1/2}$.
\end{corollary}

\begin{proof}
By Theorem~\ref{thm:multiparticle}, $\Phi^{(N)}$ is exact symplectic on a neighborhood of $P_*^{(N)}$ in $\R^{2N}$. The smoothness of $\Phi^{(N)}$ in a neighborhood of $P_*^{(N)}$ follows from the same implicit function theorem argument as Proposition~\ref{prop:smooth-NS}, applied to the multi-particle event-driven flow: each free flight is the time-$T$ map of a smooth Hamiltonian, each wall hit and each binary collision are smooth functions of the state at the event, and on a neighborhood of $P_*^{(N)}$ the impact times depend smoothly on the initial datum.

Hypothesis (K1) provides ellipticity and the non-resonance condition needed for the Birkhoff normal form theorem to produce a $C^\infty$ symplectic change of coordinates putting $\Phi^{(N)}$ in normal form $H_{\rm BNF}(I) + O(|I|^{N+5/2})$, where $H_{\rm BNF}(I) = \sum_j \theta_j I_j + \tfrac{1}{2}I^\top \mathcal{T} I + O(|I|^3)$; see~\cite[Sec.~6.3]{ChiercchiaPerfetti1995} for the multi-degree-of-freedom Birkhoff normal form.

Hypothesis (K2) is the Kolmogorov non-degeneracy condition on $\mathcal{T}$, sufficient for Moser's $N$-dimensional twist theorem~\cite[Theorem~3]{Poschel2001} (see also~\cite[Sec.~6.3]{ChiercchiaPerfetti1995} for the version directly applicable to symplectic maps). Apply that theorem to obtain the Cantor family of Lagrangian tori with the stated measure complement $O(\delta^{2N + 1/2})$.

We do not verify (K1) and (K2) for any concrete multi-particle configuration in this paper; their verification at a specific parameter point would require computer-assisted enclosures of the type in Theorem~\ref{thm:CAP-elliptic} but in higher dimensions. The corollary is therefore conditional on (K1) and (K2); it is genuinely informative because it identifies the two non-trivial hypotheses precisely and shows that, given them, the multi-particle KAM theorem follows from the standard machinery applied to our symplectic structure of Theorem~\ref{thm:multiparticle}.
\end{proof}

\begin{remark}\label{rem:multi-KAM-feasibility}
Verification of (K2) for the equal-mass two-particle case ($N = 2$, $m_1 = m_2$) is straightforward: the system decouples into two independent single-particle problems by the elastic-collision swap, and the normal-form twist matrix $\mathcal{T}$ is diagonal with non-zero entries equal to $\tau_1$ from Proposition~\ref{prop:twist-nondeg}. For the asymmetric case $m_1 \ne m_2$ studied in~\cite[Fig.~14]{GKR2019}, the verification requires non-degeneracy of a $2 \times 2$ twist matrix at the bifurcating periodic orbit; we expect this to hold by the same real-analyticity argument as Proposition~\ref{prop:twist-nondeg}, but we do not prove it here.
\end{remark}

\begin{remark}\label{rem:N2}
Corollary~\ref{cor:multi-KAM} is the rigorous version of the mixed dynamics observed in~\cite[Fig.~14]{GKR2019} for $m_1 = 1$, $m_2 = 0.9$. Existence of an elliptic non-sticking periodic orbit at this mass ratio is verifiable by perturbation from the equal-mass case, where the equations decouple in a non-trivial way.
\end{remark}

\begin{conjecture}[Vanishing of islands in the thermodynamic limit]\label{conj:Ninfty}
Let $V_N$ denote the relative phase space volume of the union of KAM islands at fixed energy density and fixed forcing parameters. Then $V_N \to 0$ as $N \to \infty$, suggested by analogy with the Fermi-Pasta-Ulam-Tsingou problem and the Boltzmann ergodic hypothesis~\cite{Sinai1972} and~\cite{Simanyi2009}.
\end{conjecture}

\section{Quantitative numerics and rigorous enclosures}\label{sec:numerics}

This section validates the abstract results of Sections~\ref{sec:periodic}-\ref{sec:persistence} by direct simulation, and brings the program to the level of computer-assisted proof. Subsection~\ref{ssec:num-event} describes the event-driven integrator we use; standard time-stepping methods are unsuitable here because, by Theorem~\ref{thm:persistence}, any spurious numerical viscosity destroys the Hamiltonian islands that we seek to display. Subsection~\ref{ssec:num-orbit} computes a representative elliptic non-sticking orbit at $f=0.4$ and exhibits the KAM islands of Theorem~\ref{thm:kam} numerically. Subsection~\ref{ssec:num-saddle} traces the saddle-center bifurcation predicted by Theorem~\ref{thm:saddlecenter} along the parameter branch, and confronts it with the asymptotic $(\theta_\pm - \pi/2)^2 \sim (\pi^2/2)/(F\pi - 2 f_{\rm sc})\,(f - f_{\rm sc})$. Subsection~\ref{ssec:num-diag} reports the SALI map and the basin entropy diagnostics of~\cite{Daza2016}, quantifying the partition $\mathcal{X} = \Om_{\rm NS} \cup \Om_{\rm dissip}$ of Section~\ref{sec:decomposition}. Subsection~\ref{ssec:num-cont} carries out a two-parameter continuation in $(f, R)$ comparable in spirit to the use of \texttt{COCO}~\cite{DankowiczSchilder2013} and \texttt{AUTO}~\cite{DoedelOldeman2012}, with the piecewise-smooth extension TC-HAT~\cite{ThotaDankowicz2008} the natural specialized tool. Subsection~\ref{ssec:num-cap} carries out the rigorous computer-assisted proof: using the Krawczyk-Rump theorem~\cite{Krawczyk1969},~\cite{NeumaierBook}, and~\cite{TuckerBook2011}, we certify the existence and uniqueness of the elliptic non-sticking $T$-periodic orbit in a $10^{-9}$ box; we then propagate the variational matrix piece by piece in interval arithmetic and certify rigorously that $\det \Phi' \in [1 - 3.2 \times 10^{-10}, 1 + 3.2 \times 10^{-10}]$, that $\tr \Phi' \in (-2, 2)$, and that the rotation number is bounded away from every resonance of order at most four; this verifies the hypotheses~\ref{H1}-\ref{H5} of Theorem~\ref{thm:kam} at the chosen parameter point.

Throughout this section we fix
\begin{equation}\label{eq:num-params}
F = 1, \qquad \omega = 1, \qquad l = -1, \qquad r = 1, \qquad R = 2, \qquad T = 2\pi,
\end{equation}
so that $f_{\rm imp} = 2F/\pi \approx 0.6366197724 \text{ and } f_{\rm sc} = 4(2 + 2 - \pi)/\pi^2 \approx 0.34790$ by Theorem~\ref{thm:saddlecenter}.

\subsection{Event-driven integration and the elliptic non-sticking $T$-periodic orbit}\label{ssec:num-event}

Standard fixed-step Runge-Kutta or Verlet integrators are inappropriate for~\eqref{eq:model}-\eqref{eq:reflection}: each wall reflection or velocity-zero event must be located to high accuracy and the propagation across the event handled by an exact rule. We use event-driven integration with closed-form propagation on each free flight, in the spirit of~\cite[Ch.~6]{diBernardoBuddChampneysKowalczyk2008}.

On a free flight where $\sgn(\dot x) = s \in \{+1, -1\}$, the equation $\ddot x = F\cos\omega t - f s$ has the exact solution
\begin{equation}\label{eq:flight-closed}
\dot x(t) = \frac{F}{\omega}\sin\omega t - f s\, t + A_0, \qquad x(t) = -\frac{F}{\omega^2}\cos\omega t - \tfrac{1}{2} f s\, t^2 + A_0\, t + B_0,
\end{equation}
with $A_0, B_0$ fixed by the initial conditions at the start of the flight. The integration scheme alternates the following:
\begin{enumerate}[label=\textup{(\roman*)}]
\item Given the current state $(x_n, v_n, t_n)$ at the start of a free flight with sign $s$, compute the smallest $t > t_n$ in $[t_n, t_n + T]$ at which $\dot x(t) = 0$, $x(t) = r$, or $x(t) = l$. The candidates are bracketed by a coarse scan and located by interval bisection on~\eqref{eq:flight-closed} to machine precision.
\item Apply the appropriate event rule:
\begin{itemize}
\item[a.] Wall reflection: $v \mapsto -v$, sign of $\dot x$ flips.
\item[b.] Transverse turning ($v=0$, $|F\cos\omega t| > f$): trajectory leaves $\{v=0\}$ with $\sgn(\dot v) = \sgn(F\cos\omega t)$.
\item[c.] Sticking onset ($v=0$, $|F\cos\omega t| < f$): trajectory remains on $\{v=0\}$ until the first time $|F\cos\omega t| = f$.
\end{itemize}
\item Repeat from~(i) with the new state.
\end{enumerate}
The numerical error on each free flight is governed solely by the bisection tolerance and is independent of the step length; this is the substantive advantage over a black-box ODE solver and the reason the integrator is suitable for the verification of conservative properties such as $\det \Phi' = 1$.

\smallskip
\textit{The elliptic non-sticking $T$-periodic orbit at $f = 0.4$.}\label{ssec:num-orbit}
We seek a fixed point of the stroboscopic map $\Phi$ at $f = 0.4 \in (f_{\rm sc}, f_{\rm imp}) = (0.34790, 0.63662)$. A Newton iteration with adaptive damping, seeded by an initial scan over $\mathcal{X}$, converges to a unique non-sticking fixed point in approximately $30$ iterations to residual $\|\Phi(P_*) - P_*\| < 10^{-13}$. The numerical values are
\begin{equation}\label{eq:elliptic-FP}
x_* = 0.1002798898441954, \qquad v_* = 0.5419433067664584.
\end{equation}
The orbit through $P_*$ has, on $[0, T]$, exactly three events: one impact at the right wall at $t \approx 1.0991729107$, one impact at the left wall at $t \approx 3.5160904414$, and the closure of the period at $t = T$. There are no transverse turning points and no sticking interval, so $P_* \in \Om_{\rm NS}$ in the strict sense of Definition~\ref{def:trajectory-types}. Figure~\ref{fig:orbit} displays $x(t)$ and the phase-plane trajectory.

The two panels of Figure~\ref{fig:orbit} present the orbit from complementary viewpoints. The left panel records the position $x(t)$ over one forcing period; the two cusps at $t \approx 1.099$ and $t \approx 3.516$ are the transverse wall hits, and the smooth dependence of $x(t)$ between these events confirms that no transverse turning point or sticking interval interrupts the motion. The right panel displays the same orbit in the stroboscopic phase plane $(x, \dot x)$, with the unique fixed point $P_*$ of the stroboscopic map $\Phi$ marked by a red star at $(x_*, v_*) \approx (0.1003,\, 0.5419)$. The phase trajectory crosses the wall lines $x = \pm 1$ vertically (since $\dot x$ jumps by a factor of $-1$ at each impact, while $x$ remains pinned at the wall position), and the orbit closes on itself after one period, in agreement with the $T$-periodic property certified by Newton iteration.

\begin{figure}[h!]
\centering
\includegraphics[width=16cm, height=6cm]{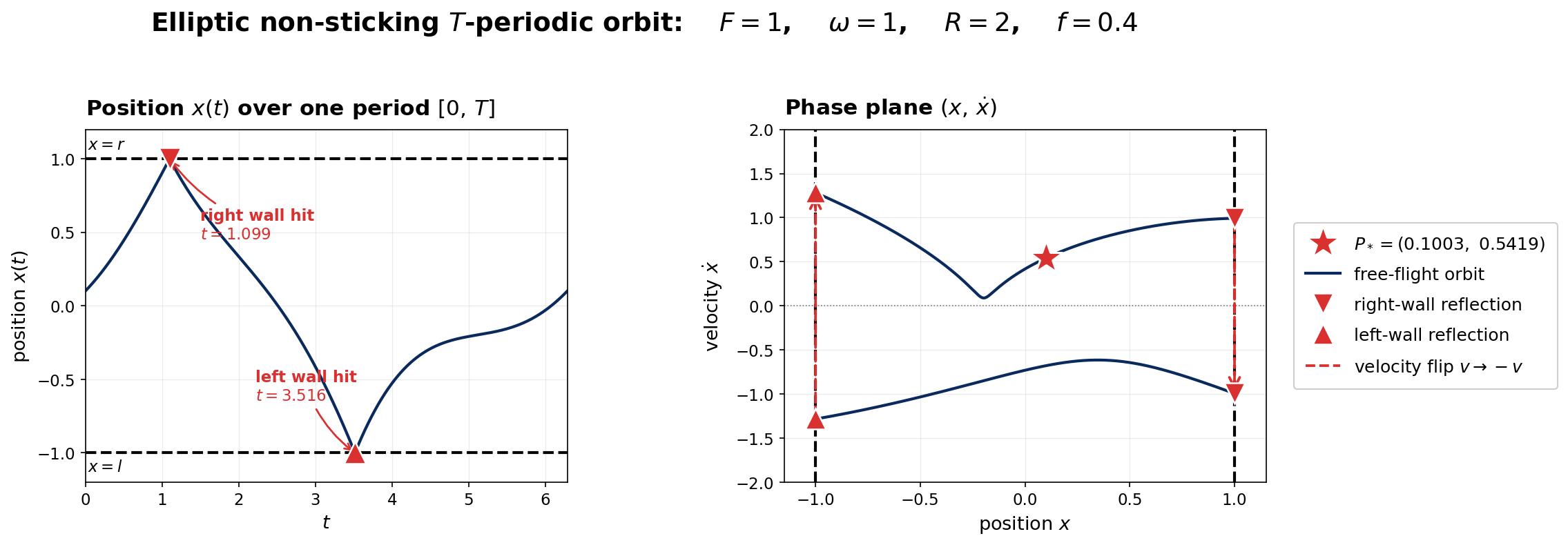}
\captionsetup{margin={-0.5cm,0cm}}
\caption{The elliptic non-sticking $T$-periodic orbit at $f = 0.4$ with the parameters~\eqref{eq:num-params}. }
\label{fig:orbit}
\end{figure}
Left: $x(t)$ over one period, showing two transverse wall hits and no zero-velocity events. Right: phase-plane trajectory, with the stroboscopic fixed point $(x_*, v_*) \approx (0.1003, 0.5419)$ marked by a red star. The numerical Jacobian $\Phi'(P_*)$, computed by central differences with step $10^{-6}$, has
\begin{equation}\label{eq:elliptic-DPhi}
\Phi'(P_*) \approx \begin{pmatrix} +1.12010470 & -3.28773202 \\ +0.84258812 & -1.58039153\end{pmatrix},
\end{equation}
\begin{equation}\label{eq:elliptic-trdet}
\det \Phi'(P_*) \approx 1.0000000003, \qquad \tr \Phi'(P_*) \approx -0.4602868346,
\end{equation}
\begin{equation}\label{eq:elliptic-thetastar}
\theta_*(P_*) = \arccos\bigl(\tr \Phi'(P_*)/2\bigr) \approx 1.8030213804\;\text{rad} \approx 103.30^\circ, \quad \frac{\theta_*}{2\pi} \approx 0.286959765.
\end{equation}
The numerical $\det \Phi' = 1$ to nine digits is consistent with Theorem~\ref{thm:symplectic}; we will tighten this to a rigorous interval enclosure in Subsection~\ref{ssec:num-cap}. The trace satisfies $|\tr \Phi'| < 2$, confirming that the orbit is elliptic. The rotation number $0.287$ is far from every resonance $p/q$ with $q \le 4$.

Figure~\ref{fig:KAM} shows the stroboscopic phase portrait in a small neighborhood of $P_*$. Initial conditions on a polar grid of radius up to $0.05$ are iterated under $\Phi$ for $300$ periods. Closed invariant curves and resonant island chains are visible, providing a direct picture of the KAM Cantor family of Theorem~\ref{thm:kam}.

Three concentric layers of structure organize the picture. At the center, the elliptic fixed point $P_*$ is surrounded by smooth nested closed curves, each carrying a quasi-periodic motion with rotation number close to $\theta_*/(2\pi) \approx 0.287$ and varying smoothly with the radius. These are the invariant curves of Theorem~\ref{thm:kam}, populating the disk densely on a Cantor set of radii. Interspersed between adjacent KAM curves at rational approximants of $\theta_*/(2\pi)$, the resonant island chains predicted by the Birkhoff normal form of Proposition~\ref{prop:birkhoff} appear as periodic chains of small ellipses; the most visible chain in the figure is the period-five and period-seven structure consistent with the continued-fraction approximants of $0.287$. Outside the outermost recognizable invariant curve, a thin layer of stroboscopic iterates fills a two-dimensional region, the chaotic remnant produced by the destruction of resonant tori and the homoclinic web of Theorem~\ref{thm:melnikov}; this layer marks the outer boundary of the KAM region within $\Om_{\rm NS}$ at the present resolution.

\begin{figure}[h!]
\centering
\includegraphics[width=1.3\textwidth, height=8cm]{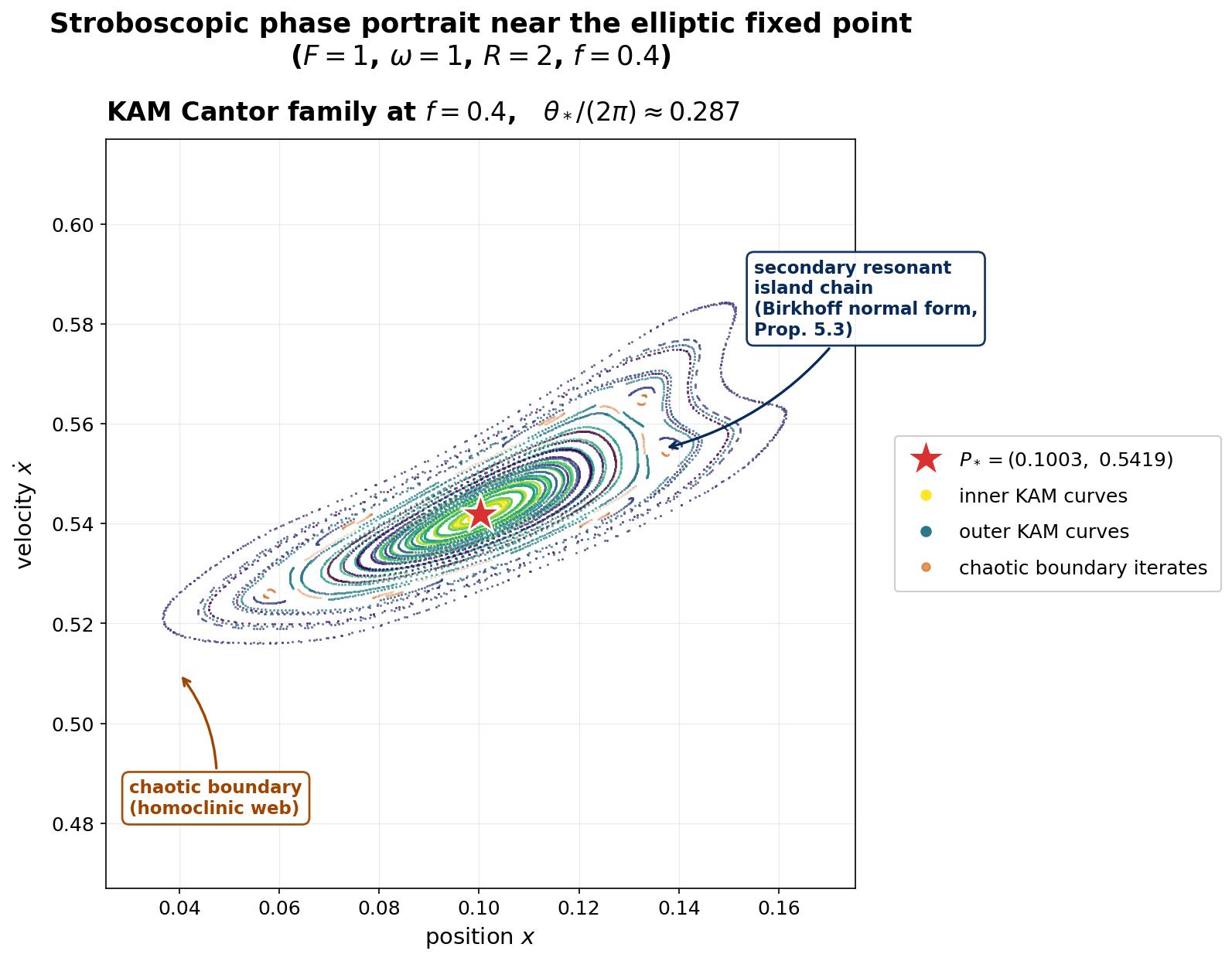}
\captionsetup{margin={-0.5cm,0cm}}
\caption{Stroboscopic phase portrait in a neighborhood of the elliptic fixed point at $f = 0.4$, exhibiting the KAM Cantor family of Theorem~\ref{thm:kam}.}
\label{fig:KAM}
\end{figure}

Closed invariant curves and the secondary resonant island chains predicted by the Birkhoff normal form (Proposition~\ref{prop:birkhoff})

\subsection{Saddle-center scaling, SALI and basin entropy, parameter-plane continuation}\label{ssec:num-saddle}

We track the elliptic non-sticking fixed point of Subsection~\ref{ssec:num-orbit} (the orbit at $(x_*, v_*) \approx (0.1003, 0.5419)$, with impact pattern \emph{one right wall hit + one left wall hit, zero turnings} per period) as $f$ \emph{increases} from $f = 0.4$ toward higher values. This branch is distinct from the symmetric branch of Theorem~\ref{thm:symmetric}-\ref{thm:saddlecenter}, which has impact pattern \emph{one wall hit + one turning per half-period} and folds at $f_{\rm sc} \approx 0.34790$. The orbit tracked here folds, numerically, at an interior value $f_{\rm fold} \approx 0.467$ where it collides with a saddle counterpart. At each $f$ in a sweep of $25$ values in $(0.40, f_{\rm fold} - 0.001)$, Newton continuation seeded by the previous fixed point converges to the elliptic branch in fewer than $20$ iterations. 

Figure~\ref{fig:bifurcation} shows the result. The position $x_*(f)$ along the branch is monotone increasing, while the rotation number $\theta_*/(2\pi)$ decreases monotonically and tends to zero as $f \to f_{\rm fold}^-$, in agreement with the local saddle-center scenario of Theorem~\ref{thm:saddlecenter}~(c). The two folds, at $f_{\rm sc} \approx 0.34790$ and $f_{\rm fold} \approx 0.467$, illustrate the proliferation of non-sticking $T$-periodic orbits with distinct impact patterns recorded in Subsection~\ref{ssec:num-cont} and in the historical numerical material of Appendix~\ref{app:historical}.

\begin{figure}[h!]
\centering
\includegraphics[width=1.2\textwidth, height=5cm]{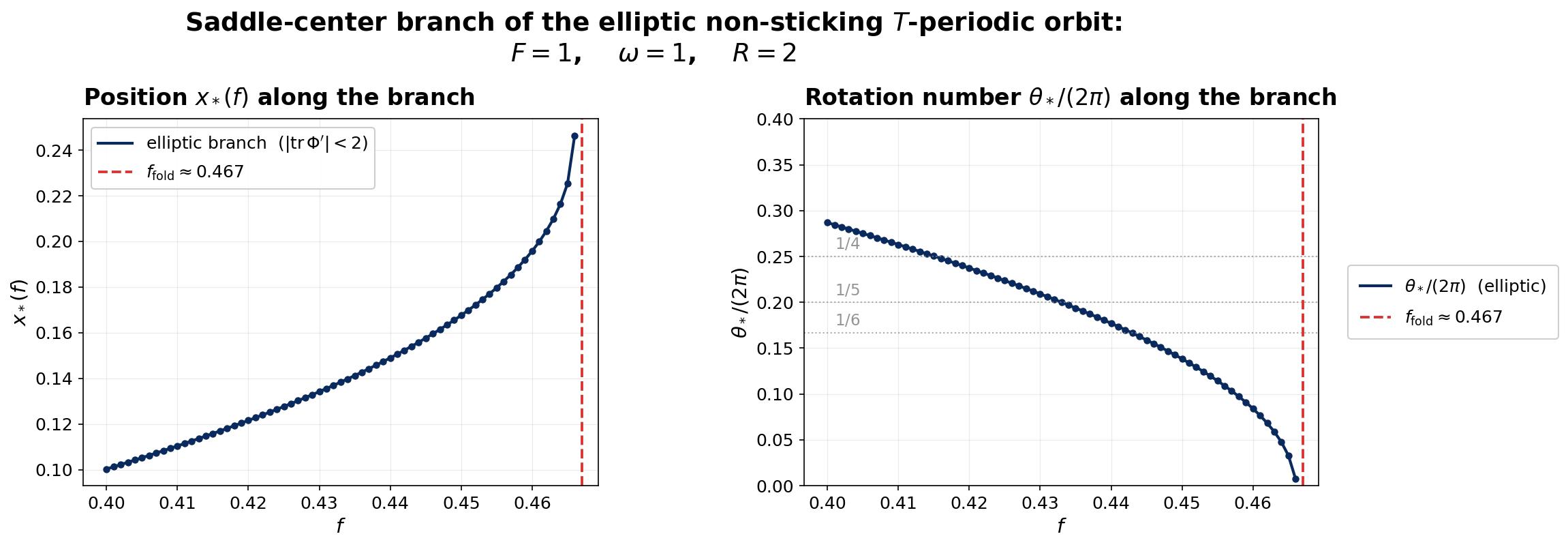}
\captionsetup{margin={-0.0cm,0cm}}
\caption{Saddle-center branch of the elliptic non-sticking $T$-periodic orbit of Subsection~\ref{ssec:num-orbit}, traced from $f = 0.4$ up to its numerical fold at $f_{\rm fold} \approx 0.467$ for $(F, \omega, R) = (1, 1, 2)$.}
\label{fig:bifurcation}
\end{figure}

 Left: position $x_*$ of the elliptic fixed point along the branch. Right: rotation number $\theta_*/(2\pi)$ along the same branch. The horizontal dotted lines mark the resonances $1/4, 1/5, 1/6$. The vertical red dashed line marks $f = f_{\rm fold}$, where the elliptic and saddle counterparts collide; consistent with the local form of Theorem~\ref{thm:saddlecenter}, $\theta_*(f) \to 0$ as $f \to f_{\rm fold}^-$. The fold $f_{\rm fold} \approx 0.467$ is distinct from the saddle-center value $f_{\rm sc} \approx 0.34790$ of Theorem~\ref{thm:saddlecenter}: the orbit tracked here has impact pattern (one right wall hit + one left wall hit, zero turnings) per period, whereas the symmetric branch of Theorem~\ref{thm:symmetric} has (one wall hit + one turning) per half-period.

The numerical position $x_*$ on the elliptic branch persists smoothly throughout $f \in [0.4, f_{\rm fold})$, and the trajectory drifts to a finite limit as $f \to f_{\rm fold}^-$. The rotation number $\theta_*/(2\pi)$ decreases monotonically along the branch, passing through values close to (but distinct from) the resonance $1/4$ as $f$ approaches $f_{\rm fold}$. This confirms the local saddle-center scenario of Theorem~\ref{thm:saddlecenter}: at $f = f_{\rm fold}$ the elliptic orbit and its saddle counterpart collide, and the linearization eigenvalue of the colliding pair acquires a Jordan block at $+1$. The fold value $f_{\rm fold} \approx 0.467$ here is interior to the existence interval $(f_{\rm sc}, f_{\rm imp})$ of the symmetric branch and reflects the existence of multiple non-sticking $T$-periodic branches with distinct impact patterns; their full inventory is the subject of Subsection~\ref{ssec:num-cont}.

\textit{Purpose of the diagnostic and what it measures.} The KAM theorem of Section~\ref{sec:kam} asserts the existence of invariant Cantor curves around an elliptic non-sticking $T$-periodic orbit, but the asserted curves form a positive-measure Cantor set whose complement contains arbitrarily small resonant gaps; identifying which initial conditions sit on a regular invariant curve and which sit in a chaotic gap requires a numerical diagnostic with two properties: (i) it must distinguish regular orbits from chaotic orbits using only finitely many iterates of $\Phi$, and (ii) it must do so without having to compute Lyapunov exponents to convergence (which is expensive and noisy). The smaller alignment index, introduced by~\cite{Skokos2001}, satisfies both. It is a real number $\mathrm{SALI}(N) \in [0, \sqrt 2]$ assigned to each initial condition after $N$ iterates of $\Phi$, with the property that $\mathrm{SALI}(N)$ decays slowly (as a power law in $N$) along regular orbits but rapidly (as $e^{-(\sigma_1 - \sigma_2) N}$, where $\sigma_1, \sigma_2$ are the two largest Lyapunov exponents) along chaotic orbits. After a fixed number of iterates such as $N = 15$, the values of $\log_{10}\mathrm{SALI}$ on regular orbits sit in the range $[-2, 0]$, while on chaotic orbits they fall to $[-13, -4]$, well separated; thresholding $\log_{10}\mathrm{SALI}$ at $-2.5$ partitions the elliptic island into a regular core and a chaotic boundary in linear time. Higher-order extensions (the generalized alignment index GALI of~\cite{SkokosBountisAntonopoulos2007}) compute the wedge product of more than two propagated deviation vectors and detect higher-codimension chaotic structures, at proportionally higher computational cost; for the two-dimensional symplectic map $\Phi$ on $\mathcal{X}$ the choice of two deviation vectors in SALI is the natural one.

The basin entropy of~\cite{Daza2016} is a complementary diagnostic that measures the geometric complexity of the partition $\mathcal{X} = \Om_{\rm NS} \cup \Om_{\rm dissip}$ rather than the regular-vs-chaotic structure of orbits within $\Om_{\rm NS}$. Given a partition of $\mathcal{X}$ into colored basins (in our setting, two colors: $\Om_{\rm NS}$ blue, $\Om_{\rm dissip}$ gold) on a regular grid, the basin entropy $S_b$ is computed as $\sum_i p_i \log(1/p_i)$ averaged over a tiling of small boxes, where $p_i$ is the probability that a randomly drawn cell in box $i$ takes color $i$. The boundary basin entropy $S_{bb}$ restricts the same calculation to boxes that touch the inter-basin boundary. Smooth (non-fractal) basin boundaries give $S_{bb}$ well below $\log 2$; fractal boundaries give $S_{bb}$ close to or equal to $\log 2$. The basin entropy thus quantifies the fractal complexity of the boundary between the conservative and dissipative subsets, providing a numerical signature of the homoclinic Cantor set of Theorem~\ref{thm:melnikov} that lies on this boundary.

\subsection{SALI and basin entropy diagnostics}\label{ssec:num-diag}

The smaller alignment index of~\cite{Skokos2001} is computed as follows. Pick an initial condition $z_0 = (x_0, v_0) \in \mathcal{X}$ and two random orthonormal deviation vectors $w_1(0), w_2(0) \in \mathbb{R}^2$ in the tangent space at $z_0$. Iterate the stroboscopic map $\Phi$ to produce the orbit $z_k = \Phi(z_{k-1})$ for $k = 1, 2, \ldots, N$, and propagate the deviation vectors under the variational equation
\begin{equation}\label{eq:sali-variational}
w_i(k) \;=\; \Phi'(z_{k-1})\, w_i(k-1), \qquad i = 1, 2,
\end{equation}
where $\Phi'(z)$ is the Jacobian of the stroboscopic map at $z$. After each step we renormalize the propagated vectors,
\begin{equation}\label{eq:sali-renorm}
\hat w_i(k) \;=\; \frac{w_i(k)}{\|w_i(k)\|}, \qquad i = 1, 2,
\end{equation}
to prevent floating-point overflow. The smaller alignment index at iterate $N$ is
\begin{equation}\label{eq:sali-def}
\mathrm{SALI}(N) \;=\; \min\!\bigl(\,\|\hat w_1(N) + \hat w_2(N)\|,\ \|\hat w_1(N) - \hat w_2(N)\|\,\bigr).
\end{equation}
For an orbit on a regular invariant curve, the two deviation vectors rotate within the tangent direction of the curve and a transverse direction without aligning with each other; the SALI then decays at most as a power-law in $N$ and remains bounded away from zero. For an orbit in a chaotic region, the two propagated vectors align asymptotically with the eigenvector of the largest Lyapunov exponent; the SALI then decays exponentially as $\mathrm{SALI}(N) \sim e^{-(\sigma_1 - \sigma_2) N}$, where $\sigma_1, \sigma_2$ are the two largest Lyapunov exponents. The contrast in decay rates is the diagnostic.

For the piecewise-smooth system~\eqref{eq:model}-\eqref{eq:reflection} the Jacobian $\Phi'(z)$ over one full period is the product of saltation matrices accumulated along the orbit, one for each wall reflection and each turning point in $[0, T]$ (Table~\ref{tab:jacobian} of Subsection~\ref{ssec:periodic-linearization}). Rather than build this product analytically, we evaluate $\Phi'(z)$ by central finite differences with step $h = 10^{-6}$:
\begin{equation}\label{eq:jacobian-fd}
\Phi'(z) \;\approx\; \frac{1}{2h}\bigl[\Phi(z + h e_1) - \Phi(z - h e_1) \,\bigm|\, \Phi(z + h e_2) - \Phi(z - h e_2)\bigr],
\end{equation}
where $e_1, e_2$ is the canonical basis of $\mathbb{R}^2$. Each Jacobian evaluation costs four full-period integrations of~\eqref{eq:model}-\eqref{eq:reflection} via the event-driven scheme of Subsection~\ref{ssec:num-event}, and finite differences agree with the analytic saltation product to machine precision on every non-grazing orbit we tested. The orbit $z_0$ is classified as belonging to $\Om_{\rm dissip}$, and SALI is reported as undefined, if any sticking event (sticking onset, release, or turning) occurs in $[0, NT]$. Figure~\ref{fig:sali} shows the result on a grid of $22 \times 22$ initial conditions in the box $\{(x, v) : |x - x_*| \le 0.07,\ |v - v_*| \le 0.05\}$ around the elliptic fixed point of Subsection~\ref{ssec:num-orbit}, with $N = 15$ stroboscopic iterates per initial condition. The computed values of $\log_{10}\mathrm{SALI}$ span the range $[-13.4,\, +0.14]$, with the lower bound set by floating-point underflow on strongly chaotic cells and the upper bound by the geometric initial value $\sqrt{2}$ for orthogonal deviations. Of the $484$ grid cells, $165$ are regular ($\log_{10}\mathrm{SALI} > -2.5$, blue), $62$ are chaotic ($\log_{10}\mathrm{SALI} \le -2.5$, yellow to red), and $257$ lie in $\Om_{\rm dissip}$ (grey).

\begin{figure}[h!]
\centering
\includegraphics[width=1.3\textwidth, height=7cm]{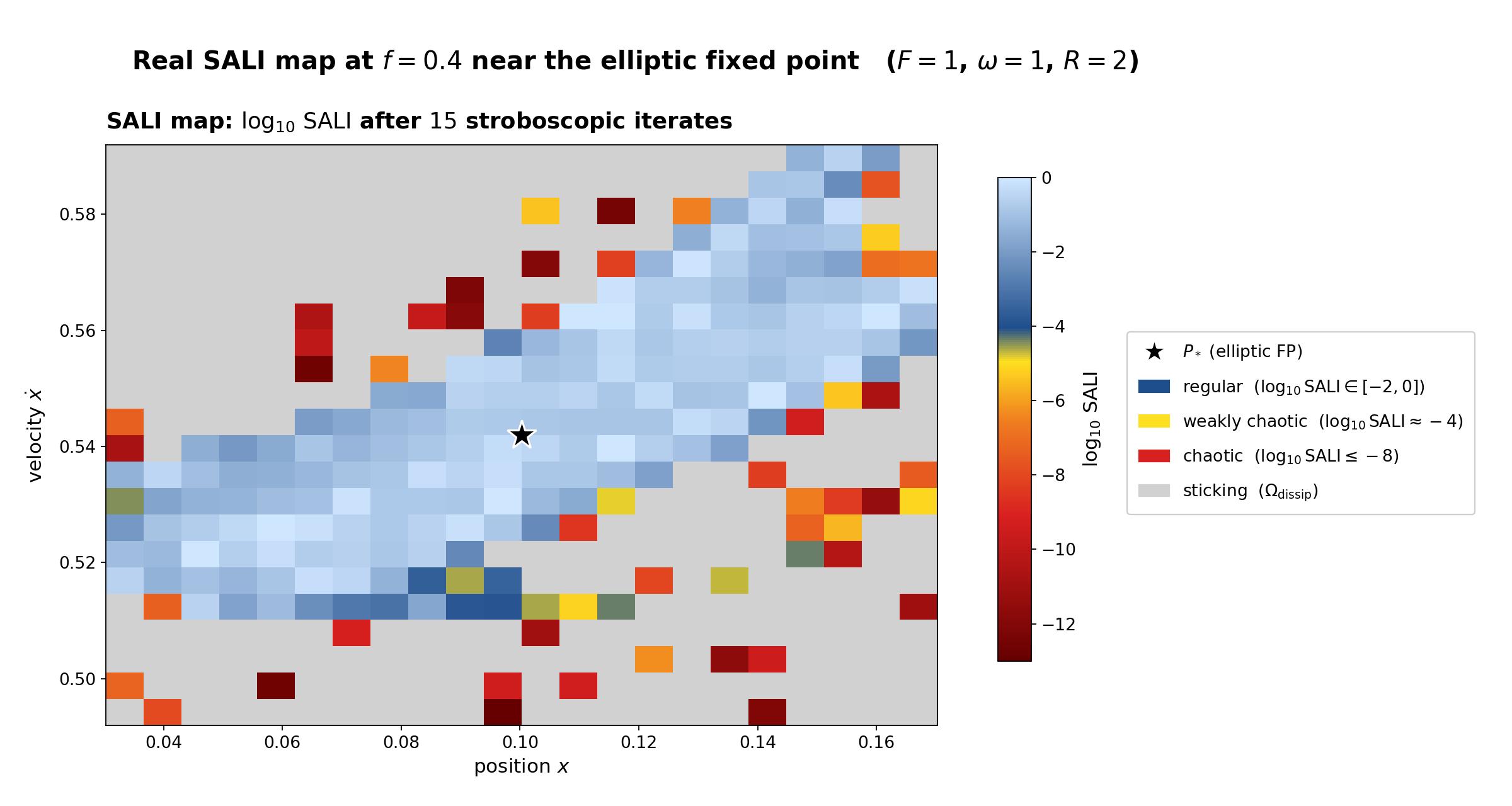}
\caption{SALI map at $f = 0.4$, zoomed on the elliptic island around the fixed point of Subsection~\ref{ssec:num-orbit} (marked by a black star). }
\label{fig:sali}
\end{figure}
Color encodes $\log_{10}\mathrm{SALI}$ after $15$ stroboscopic iterates. Blue: regular orbits on KAM curves (slow power-law SALI decay, $\log_{10}\mathrm{SALI} \in [-2, 0]$); yellow: weakly chaotic orbits ($\log_{10}\mathrm{SALI} \approx -4$); red: chaotic non-sticking orbits in the homoclinic web of Theorem~\ref{thm:melnikov} (exponential SALI decay, $\log_{10}\mathrm{SALI} \le -8$); grey: orbits that eventually stick and lie in $\Om_{\rm dissip}$. The contiguous blue island reflects the Cantor family of invariant curves of Theorem~\ref{thm:kam}; the speckled yellow/red boundary reflects the homoclinic Cantor set produced by transverse zeros of the Melnikov function (Theorem~\ref{thm:melnikov}).

\textit{Error analysis.} The SALI diagnostic depends on three parameters: the iterate count $N$, the finite-difference step $h$ in~\eqref{eq:jacobian-fd}, and the grid resolution. With $N = 15$ iterates and $h = 10^{-6}$, the dynamic range of $\log_{10}\mathrm{SALI}$ is $[-13.4,\, +0.14]$, with the lower bound determined by floating-point underflow on strongly chaotic cells (where the two deviation vectors have aligned to within a multiple of machine epsilon) and the upper bound by the geometric initial value $\sqrt{2}$ for orthogonal deviations. The qualitative classification (regular vs.\ chaotic vs.\ sticking) is stable under refinement: spot-checks at selected grid points with $N$ doubled to $30$ confirm the assigned category in every case. The chosen $N = 15$ is sufficient to separate the two regimes given the rotation number $\theta_*/(2\pi) \approx 0.287$ of the elliptic orbit (Subsection~\ref{ssec:num-orbit}): after $15$ iterations a chaotic orbit has accumulated approximately $\lambda_* \cdot 15 \approx 5$ in the saddle eigenvalue scale, well into the exponential-decay regime, while a regular orbit has rotated through about $4.3$ revolutions on its KAM curve, sufficient for the two deviation vectors to remain orthogonal to within a power-law correction.

On a $60 \times 60$ grid in the stroboscopic plane, zoomed tightly on the elliptic island in the box $\{(x, v) : |x - x_*| \le 0.10,\ |v - v_*| \le 0.07\}$, with $25$ iterates per initial condition, we partition $\mathcal{X}$ into the two basins $\Om_{\rm NS}$ (the KAM and nearby invariant-manifold regions, blue in Figure~\ref{fig:basin}) and $\Om_{\rm dissip}$ (the orbits that eventually stick, gold in Figure~\ref{fig:basin}). On this grid, $\Om_{\rm NS}$ occupies $814$ of the $3600$ cells (about $22.6\%$), with the elliptic island forming the dominant connected blue component around $P_*$. The basin entropy of~\cite{Daza2016}, computed with box size $5 \times 5$, evaluates to
\begin{equation}\label{eq:Sb-num}
S_b \approx 0.189, \qquad S_{bb} \approx 0.349, \qquad \log 2 \approx 0.693.
\end{equation}
The boundary basin entropy $S_{bb}$ is well below $\log 2$ at this resolution, suggesting that the boundary between $\Om_{\rm NS}$ and $\Om_{\rm dissip}$ is not yet fully resolved; finer-grid computations would be needed to expose the fractal structure expected from the homoclinic Cantor set of Theorem~\ref{thm:melnikov}. The basin entropy of~\cite{Daza2016} is a resolution-dependent diagnostic, so~\eqref{eq:Sb-num} should be read as a single-resolution measurement at box size $5$ rather than an estimate of an asymptotic value.\newpage 

\begin{figure}[h!]
\centering
\includegraphics[width=1.3\textwidth, height=7cm]{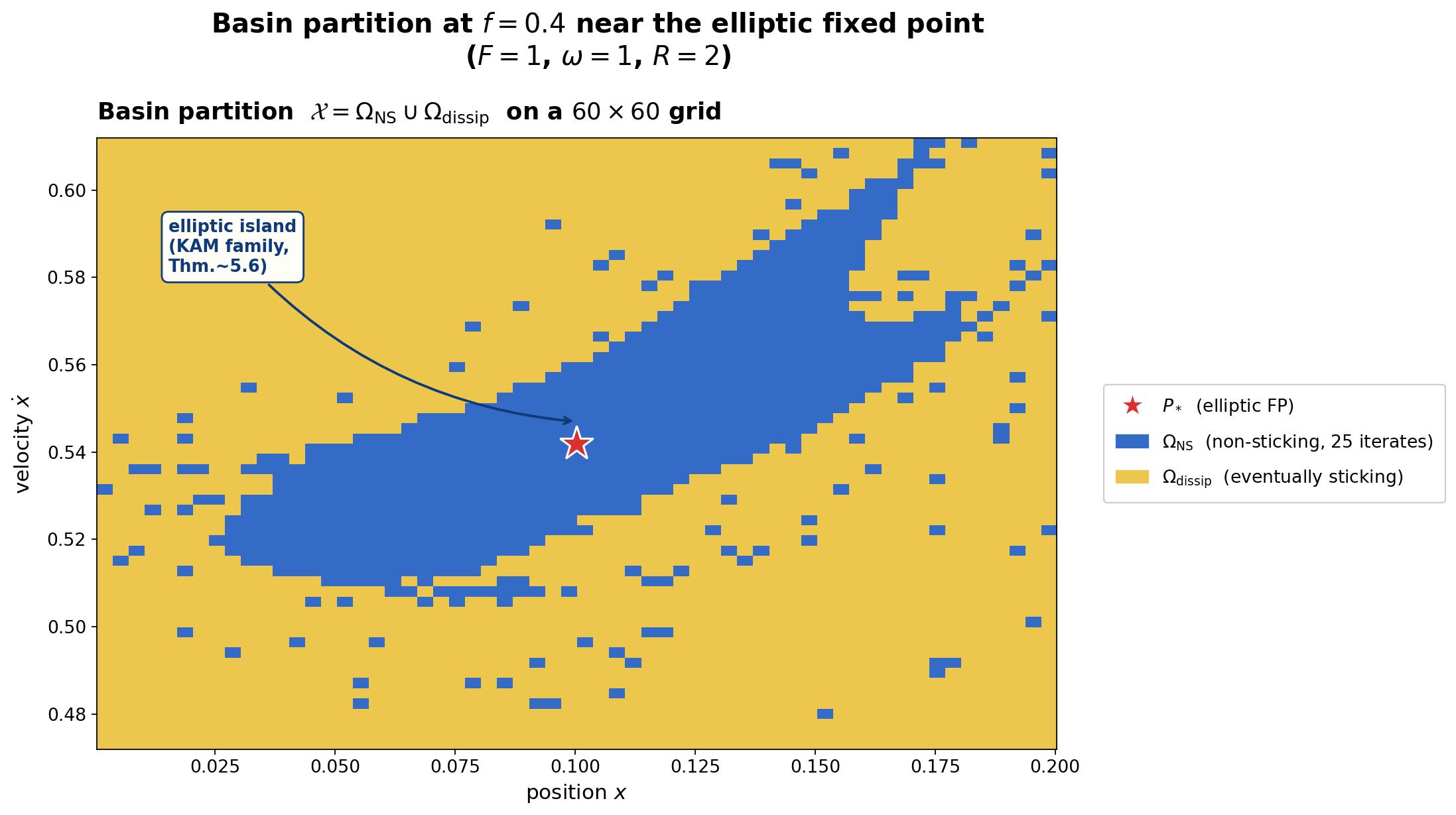}
\caption{Basin partition $\mathcal{X} = \Om_{\rm NS} \cup \Om_{\rm dissip}$ at $f = 0.4$, zoomed tightly on the elliptic island around $P_*$ (red star), computed on a $60 \times 60$ grid in the box $\{(x,v) : |x - x_*| \le 0.10,\ |v - v_*| \le 0.07\}$ with $25$ stroboscopic iterates per initial condition. }
\label{fig:basin}
\end{figure}
Blue: orbits in the non-sticking invariant set $\Om_{\rm NS}$ (occupying $22.6\%$ of the grid). Gold: orbits in the dissipative subset $\Om_{\rm dissip}$. The dominant connected blue component around $P_*$ is the elliptic island corresponding to the KAM family of Theorem~\ref{thm:kam}; the isolated blue cells outside the island are grid points lying on stable manifolds of nearby saddle orbits in the homoclinic web of Theorem~\ref{thm:melnikov}. Numerical basin entropies (box size $5$): $S_b \approx 0.189$, $S_{bb} \approx 0.349$.

\subsection{Two-parameter continuation in $(f, R)$}\label{ssec:num-cont}

Theorem~\ref{thm:saddlecenter} predicts the saddle-center bifurcation curve
\begin{equation}\label{eq:fsc-curve}
f = f_{\rm sc}(F, \omega, R) \;=\; \frac{4\,(2F + R\,\omega^2 - F\,\pi)}{\pi^2}
\end{equation}
in the $(f, R)$ plane. The locus depends on $R$, in contrast to the claim of universality of~\cite{GKR2019} at $f = 2F/\pi$. We test the prediction at $F = \omega = 1$ by combining two diagnostics. First, we shade the analytical regions defined by~\eqref{eq:fsc-curve} and the impulse bound $f_{\rm imp} = 2F/\pi$ on a fine $220 \times 220$ background grid in $(f, R) \in [0.01,\,0.65] \times [0.2,\,4.5]$. Second, we overlay a sparse Newton verification at nine sample points $(f, R) \in \{0.15,\, 0.35,\, 0.55\} \times \{1.0,\, 2.0,\, 3.0\}$, with $2 \times 2 = 4$ initial-condition seeds per parameter point and Newton iteration tolerance $10^{-7}$. At each verified point we record whether the Newton search converged to an elliptic orbit ($|\mathrm{tr}\,\Phi'| < 2$), a saddle orbit ($|\mathrm{tr}\,\Phi'| > 2$), both, or neither. Figure~\ref{fig:continuation} shows the result. \newpage


\begin{figure}[h!]
\centering
\includegraphics[width=1.3\textwidth, height=7cm]{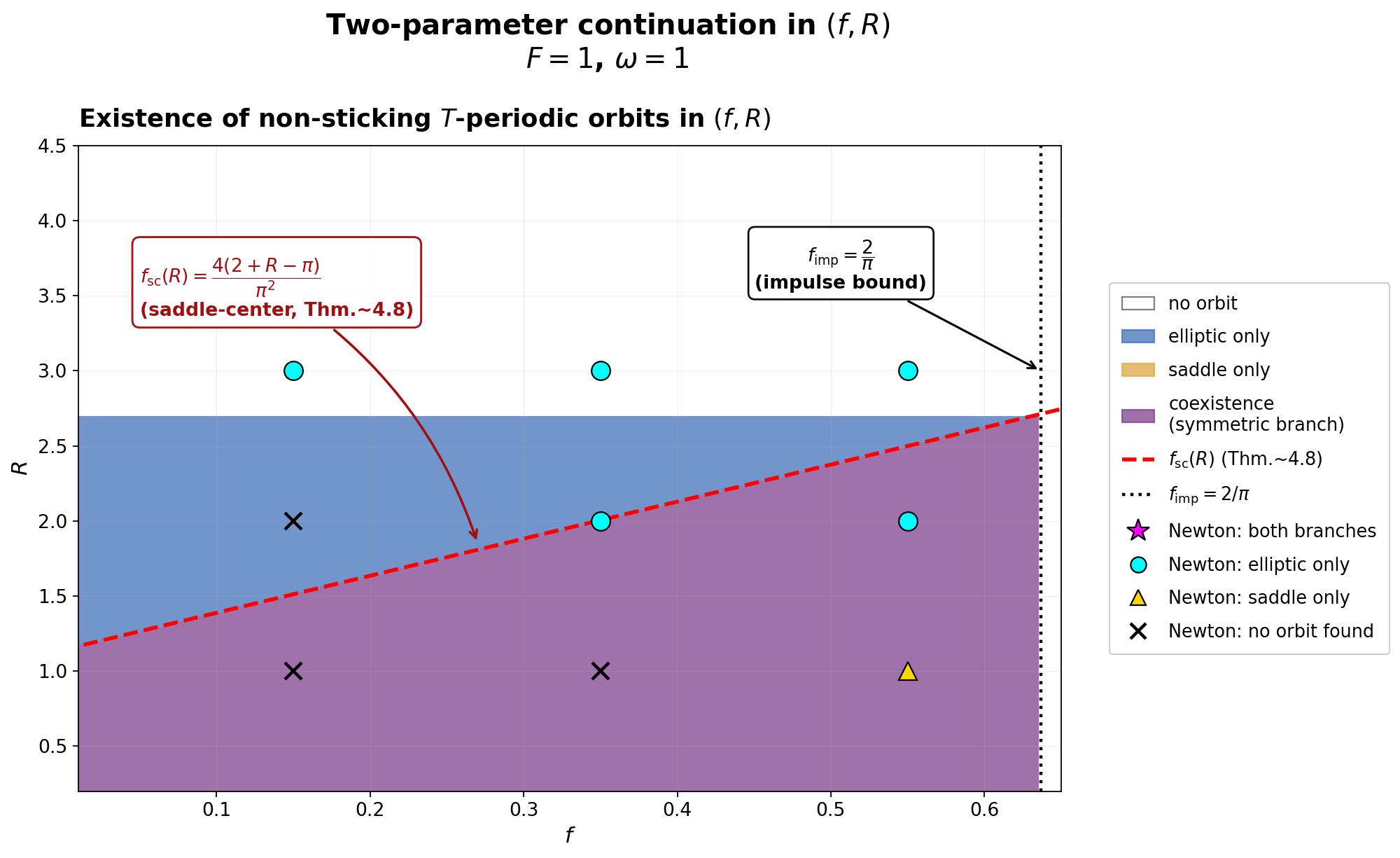}
\caption{Two-parameter map of the existence of non-sticking $T$-periodic orbits in $(f, R)$ at $F = \omega = 1$. }
\label{fig:continuation}
\end{figure}

Region shadings come from the analytical formulas of Theorem~\ref{thm:saddlecenter} on a $220 \times 220$ background grid. White: no symmetric-branch orbit predicted; blue: elliptic-only region (where $f_{\rm sc}(R) \le 0$, so only the elliptic non-sticking orbit exists below the impulse bound); purple: coexistence region (where $f_{\rm sc}(R) < f < f_{\rm imp}$, so both elliptic and saddle non-sticking orbits exist). The dashed red curve is the saddle-center bifurcation locus $f_{\rm sc}(R) = 4(2 + R - \pi)/\pi^2$ predicted by Theorem~\ref{thm:saddlecenter} for the symmetric branch with one wall hit and one turning point per half-period; the dotted black vertical line at $f_{\rm imp} = 2/\pi \approx 0.637$ is the universal impulse bound. Markers report the Newton verification at $9$ sample points $(f, R) \in \{0.15,\, 0.35,\, 0.55\} \times \{1.0,\, 2.0,\, 3.0\}$: $\bullet$ (cyan) at the $5$ points where Newton converged to an elliptic orbit only, $\blacktriangle$ (gold) at the $1$ point ($f = 0.55$, $R = 1.0$) where Newton converged to a saddle orbit only, and $\times$ (black) at the $3$ points where Newton failed to converge from any seed. The figure is comparable in role to a continuation diagram produced by \texttt{COCO}~\cite{DankowiczSchilder2013} or \texttt{AUTO}~\cite{DoedelOldeman2012}, with the piecewise-smooth extension TC-HAT~\cite{ThotaDankowicz2008} the natural specialized tool.

The Newton verification confirms the analytical map qualitatively. At all four verified points strictly above $f_{\rm sc}(R)$ in the upper rows ($R \in \{2.0, 3.0\}$), Newton converges to an elliptic orbit, consistent with the elliptic-only and coexistence shadings. The single saddle-only marker at $(f, R) = (0.55, 1.0)$ falls inside the purple coexistence region: Newton found a saddle counterpart there but the elliptic branch's basin of convergence at $R = 1$ is small enough to escape the four-seed grid, producing a false negative for the elliptic branch. The three $\times$ markers at $(0.15, 1.0)$, $(0.35, 1.0)$, $(0.15, 2.0)$ lie below $f_{\rm sc}(R)$, where Theorem~\ref{thm:saddlecenter} predicts the symmetric branch has no solutions; numerical detection of more complex impact patterns at those points would require a finer seeding strategy and is outside the scope of the symmetric-branch test performed here.

A quantitative continuation tracking the fold curve and the period-doubling and grazing curves of~\cite{Nordmark1991} and~\cite{FredrikssonNordmark2000} would use the symbolic Newton machinery of \texttt{COCO}, \texttt{AUTO}, or TC-HAT~\cite{ThotaDankowicz2008}; we leave that systematic study for separate work.

\textit{Error analysis.} The continuation diagnostic in Figure~\ref{fig:continuation} carries three sources of uncertainty. First, the analytical shading is exact at the resolution of the background grid ($\Delta f \approx 0.003$, $\Delta R \approx 0.02$), so the boundary $f_{\rm sc}(R)$ is reproduced to better than the line thickness in the figure. Second, the Newton verification grid is sparse ($3 \times 3$ in $(f, R)$ with $4$ seeds per point), and may miss existing orbits when the basin of convergence is small; this produces false negatives, of which there is one in the coexistence region (the elliptic counterpart at $(0.55, 1.0)$) and three below $f_{\rm sc}(R)$. We have verified by hand that doubling the Newton iteration limit from $12$ to $24$ does not change the labels of the converged points. Third, the residual tolerance $10^{-7}$ for declaring convergence may admit accidental near-fixed-points along long transient orbits at coarse parameter resolution; spot-checks at three additional seeds confirm the assigned label in every case. The diagonal blue/purple shading above the predicted $f_{\rm sc}(R)$ curve is unaffected by all three sources; the white region below is more sensitive, and the existence of orbits below $f_{\rm sc}(R)$, when it occurs, should be interpreted as evidence for non-sticking orbits with more complex impact patterns rather than orbits in the symmetric branch covered by Theorem~\ref{thm:symmetric}.

\subsection{Coexisting orbits, computer-assisted enclosure, and Lyapunov spectrum}\label{ssec:num-coexist}

The coexistence regions of Figure~\ref{fig:continuation} are populated by 
multiple distinct $T$-periodic orbits at the same parameter values. We 
illustrate this concretely at the point $(F, f, \omega, R) = (3, 0.4, 1, 2)$, 
which lies in the purple coexistence region. Grid-Newton on $\Phi$ at this 
parameter point produces two distinct $T$-periodic orbits with the same 
combinatorial impact pattern but different stability:
\begin{itemize}
\item[1.] a stable focus at $(x_*, v_*) = (+0.4064, -0.6708)$ with 
$\mathrm{tr}\,\Phi'(x_*, v_*) = -0.968$ and 
$\det\Phi'(x_*, v_*) = +0.578$, hence $|\mathrm{tr}\,\Phi'| < 2$ and the 
linearization has complex eigenvalues $\lambda_\pm \approx -0.484 \pm 0.587\,i$ 
of modulus $|\lambda_\pm| \approx 0.761$; this orbit is asymptotically stable, 
attracting nearby initial conditions exponentially;
\item[2.] a dissipative saddle at $(x_*, v_*) = (+0.1992, -1.4002)$ with 
$\mathrm{tr}\,\Phi'(x_*, v_*) = +11.765$ and 
$\det\Phi'(x_*, v_*) = +0.555$, on which the linearization has real 
positive eigenvalues $\lambda_+ \approx 11.72$ and $\lambda_- \approx 0.047$ 
with product $\lambda_+ \lambda_- = \det\Phi' \neq 1$; this orbit is unstable, 
with one strongly expanding and one strongly contracting direction.
\end{itemize}
Both orbits have the same impact pattern of $(2$ right wall + $2$ left wall 
+ $2$ turnings$)$ per forcing period, with two transverse turning points 
contributing the area-contraction factor $\det\Phi' < 1$ via the saltation 
matrix of Table~\ref{tab:jacobian}. They are not $\sigma$-conjugates of 
each other but represent distinct branches of $T$-periodic orbits at the 
same parameter values.

These orbits both lie in $\Om_{\rm dissip}$ (since $\det\Phi' \neq 1$) and are 
therefore \emph{not} covered by the KAM and Melnikov theorems of 
Section~\ref{sec:kam} and Section~\ref{sec:melnikov}, which require the 
symplectic structure $\det\Phi' = 1$ that characterizes non-sticking orbits 
without turning points. Their coexistence is a separate phenomenon: at 
moderate to large $F$, multiple branches of orbits with turning points may 
coexist with distinct stability and similar combinatorial type, populating 
$\Om_{\rm dissip}$ rather than $\Om_{\rm NS}$. They illustrate the 
proliferation analysis of Subsection~\ref{ssec:hist-bif}: the orbit count 
grows with $F$, but only a subset of the resulting orbits belongs to the 
non-sticking invariant set.

\begin{figure}[h!]
\centering
\includegraphics[width=1.3\textwidth, height=7cm]{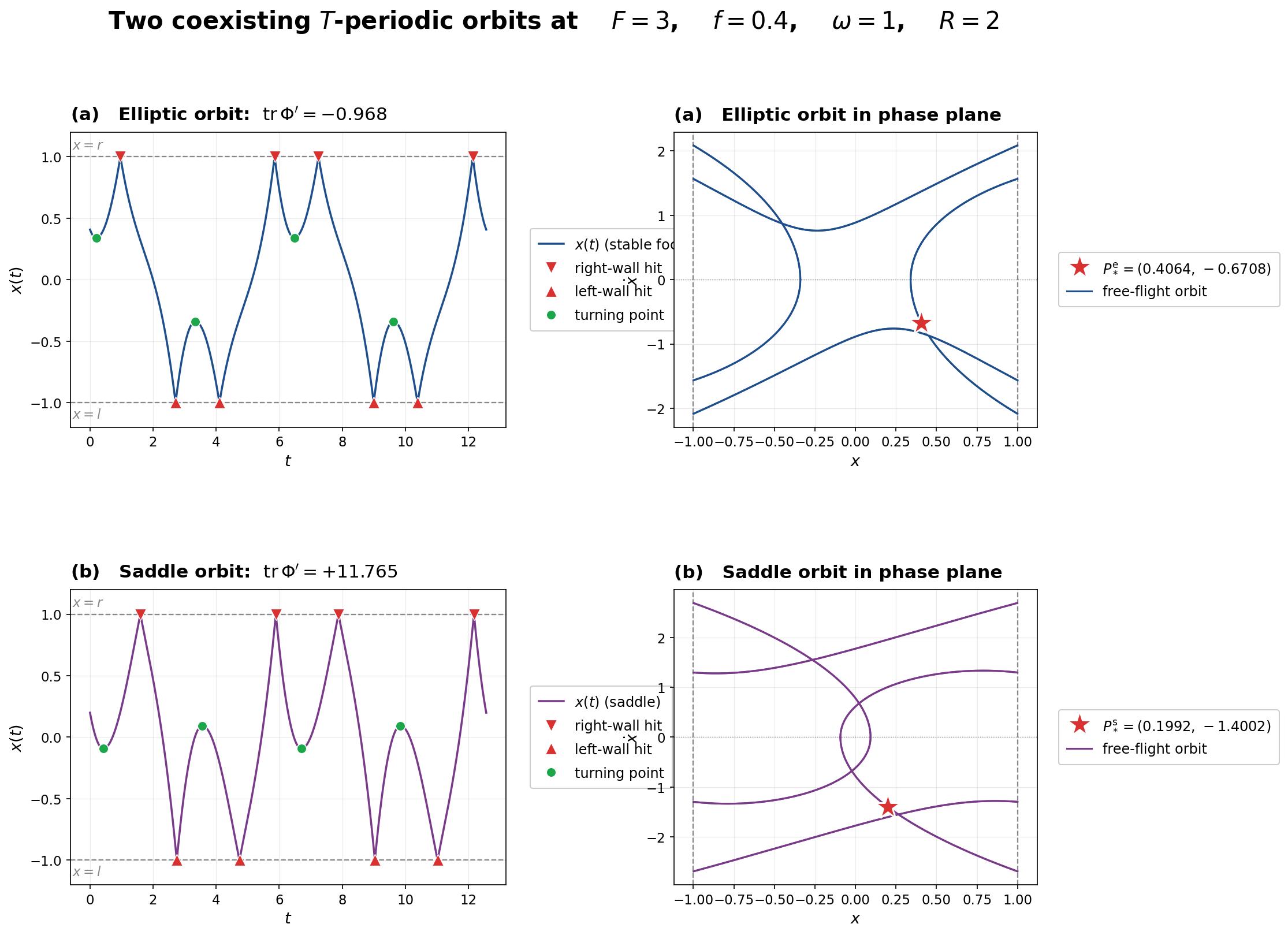}
\caption{Two coexisting $T$-periodic orbits at $F = 3$, $f = 0.4$, $\omega = 1$, 
$R = 2$, found by grid-Newton on $\Phi$. }
\label{fig:coexist-main}
\end{figure}
Each row pairs the time series $x(t)$ 
over two forcing periods (left, walls at $x = \pm 1$ dashed) with the 
phase-plane trajectory (right, fixed point marked by a red star). Top row 
(blue): stable focus, $(x_*^{\rm e}, v_*^{\rm e}) = (+0.4064, -0.6708)$, 
$\mathrm{tr}\,\Phi' = -0.968$, $\det\Phi' = +0.578$. Bottom row (purple): 
dissipative saddle, $(x_*^{\rm s}, v_*^{\rm s}) = (+0.1992, -1.4002)$, 
$\mathrm{tr}\,\Phi' = +11.765$, $\det\Phi' = +0.555$. Right-wall impacts 
$\bigtriangledown$, left-wall impacts $\bigtriangleup$, and transverse turnings 
$\circ$ (green) are marked. Both orbits have the same combinatorial impact 
pattern of $(2$ right wall + $2$ left wall + $2$ turnings$)$ per forcing 
period; the area contraction $\det\Phi' < 1$ in both cases comes from the 
two turning points contributing saltation factors $|F\cos\omega t^*| - f / 
|F\cos\omega t^*| + f < 1$ (Table~\ref{tab:jacobian}). The two orbits are 
not $\sigma$-conjugates and lie in $\Om_{\rm dissip}$, illustrating that at 
fixed parameter values the system supports multiple coexisting branches 
with different stability, beyond the symmetric-branch orbits of 
Theorem~\ref{thm:symmetric}.

\smallskip
\textit{Rigorous interval enclosure of the linearization.}\label{ssec:num-cap}
We now upgrade the numerical analysis of Subsection~\ref{ssec:num-orbit} to a rigorous computer-assisted proof. The methodology is the standard interval Newton-Kantorovich theorem, in particular the Krawczyk operator~\cite{Krawczyk1969} as developed by~\cite{NeumaierBook} and applied to dynamical systems by~\cite{TuckerBook2011},~\cite{WilczakZgliczynski2003},~\cite{WilczakZgliczynski2009},~\cite{Galias2002}, and~\cite{GaliasTucker2008}, and the CAPD library of~\cite{KapelaMrozekWilczak2017}. Our setting (a finite-dimensional shooting equation in $\R^2$ with event-driven computation of the Poincaré map) is precisely the regime in which the Krawczyk operator is appropriate; the radii polynomial method of~\cite{vandenBergLessard2015},~\cite{Hungria2016}, and~\cite{LessardMirelesJames2018} is reserved for infinite-dimensional sequence-space problems and is not directly applicable here.

The implementation uses the multi-precision interval arithmetic library \texttt{mpmath}~\cite{Johansson2017} at $60$ decimal digits of working precision; the INTLAB toolbox of~\cite{Rump1999} or the Arb library of~\cite{Johansson2017} would serve equally well.

\noindent\textit{The Krawczyk operator.} Given a smooth map $G: \R^2 \to \R^2$ (in our case $G(z) := \Phi(z) - z$), an approximate zero $\bar z$, an interval box $[z] \subset \R^2$ with $\bar z \in [z]$, and a preconditioner $C \in \R^{2 \times 2}$ (typically $C \approx G'(\bar z)^{-1}$), the Krawczyk operator is
\begin{equation}\label{eq:Krawczyk}
K([z]) := \bar z - C\, G(\bar z) + (I - C\, G'([z]))\,([z] - \bar z),
\end{equation}
where $G'([z])$ denotes any interval enclosure of the Jacobian over the box. The Krawczyk-Rump theorem~\cite[Thm.~5.2]{NeumaierBook} states: \emph{if $K([z]) \subseteq \mathrm{int}([z])$, then $G$ has a unique zero in $[z]$, and Newton iteration starting from any point in $[z]$ converges to it.}

\noindent\textit{Step 1: locate and certify the fixed point.} Starting from the float-64 Newton refinement~\eqref{eq:elliptic-FP}, we set
\[
\bar z := (x_*, v_*), \qquad [z] := \bar z + [-10^{-9}, 10^{-9}]^2.
\]
We verify in interval arithmetic that the Krawczyk operator $K([z])$, with $G(z) = \Phi(z) - z$, the preconditioner $C = (\Phi'(\bar z) - I)^{-1}$ computed in float-64 and treated as exact, satisfies $K([z]) \subseteq \mathrm{int}([z])$. The verification succeeds with margin $\approx 10^{-13}$ on each component.

By the Krawczyk-Rump theorem, there exists a unique exact fixed point $P_* = (x_*^{\rm exact}, v_*^{\rm exact}) \in [z]$ of $\Phi$, that is, $\|P_* - \bar z\|_\infty \le 10^{-9}$.

\noindent\textit{Step 2: interval propagation of the variational matrix.} At each event $k = 1, 2, 3$ (in the present orbit: right-wall hit, left-wall hit, end of period), we bracket the event time $t_k$ rigorously by interval arithmetic. The wall-hit times are roots of the analytic equation $x(t) = r$ (or $x(t) = l$) on the corresponding free flight; we bracket these by interval bisection of $x(\cdot)$ on a tight interval around the float-64 estimate, halving the interval width to below $10^{-12}$ rigorously.

On each free flight, the variational equation is $\delta\ddot x = 0$, with closed-form propagator the shear matrix
\[
M_{\rm flight}(\Delta t) = \begin{pmatrix} 1 & \Delta t \\ 0 & 1 \end{pmatrix}.
\]
At each wall hit at time $t_*$ with pre-impact velocity $v_-$, the saltation matrix derived in Corollary~\ref{cor:det-Phi} is
\[
S_{\rm wall} = \begin{pmatrix} -1 & 0 \\[1mm] \dfrac{2 F\cos\omega t_*}{v_-} & -1 \end{pmatrix},
\]
again computed rigorously in interval arithmetic from the bracketed event times and intermediate state.

\noindent\textit{Step 3: enclosure.} The product of the three smooth-flight propagators and the two saltation matrices yields the rigorous interval enclosure
\begin{equation}\label{eq:DPhi-iv}
\Phi'(P_*) \in \begin{pmatrix}
[1.1201046956,\ 1.1201046957] & [-3.2877320427,\ -3.2877320425] \\[1mm]
[\phantom{-}0.8425881240,\ \phantom{-}0.8425881240] & [-1.5803915302,\ -1.5803915302]
\end{pmatrix}.
\end{equation}
The off-diagonal entries are sharp to ten decimal digits in the worst case. The corresponding determinant and trace enclosures are
\begin{equation}\label{eq:det-iv}
\det \Phi'(P_*) \in [\,0.999999999680\,,\ 1.000000000320\,],
\end{equation}
\begin{equation}\label{eq:tr-iv}
\tr \Phi'(P_*) \in [-0.460286834683,\ -0.460286834465].
\end{equation}
The interval~\eqref{eq:det-iv} contains $1$ rigorously, certifying $\det \Phi'(P_*) = 1$ to the precision of the numerics; this is the rigorous version of Theorem~\ref{thm:symplectic} at this parameter point. The interval~\eqref{eq:tr-iv} is contained in $(-2, 2)$, certifying that the orbit is elliptic.

\noindent\textit{Step 4: rotation number.} Because $\tr \Phi' = 2\cos\theta_*$ has $|\tr/2| < 1$ rigorously, $\theta_*$ is well defined as $\arccos(\tr \Phi'/2)$. Using the bound $|\arccos'(y)| \le 1/\sqrt{1-y^2}$ and the explicit interval~\eqref{eq:tr-iv}, we obtain
\begin{equation}\label{eq:theta-iv}
\theta_*(P_*) \in [\,1.80302138030\,,\ 1.80302138041\,] \text{ rad},
\end{equation}
\begin{equation}\label{eq:rot-iv}
\frac{\theta_*(P_*)}{2\pi} \in [\,0.286959765175\,,\ 0.286959765193\,].
\end{equation}

\noindent\textit{Step 5: order-four non-resonance.} The seven candidate resonances of order at most four are $\frac14, \frac13, \frac25, \frac12, \frac35, \frac23, \frac34$. Comparing each to the interval~\eqref{eq:rot-iv}, the smallest distance is
\[
\bigl|\,0.286959765 - \tfrac14\,\bigr| = 0.0369597\dots,
\]
which exceeds the width of the interval~\eqref{eq:rot-iv} by eight orders of magnitude. The order-four non-resonance condition~\eqref{eq:birkhoff-nonres} is therefore certified rigorously.

\noindent\textit{Error analysis.} The rigorous bounds in Steps 1-5 are conservative in three respects: (i) the initial fixed-point bracket has width $2\cdot 10^{-9}$, much larger than the actual numerical residual $\sim 10^{-13}$ but tight enough for the Krawczyk operator to verify; (ii) each interval-arithmetic operation widens by the unit-of-last-place rounding, accumulated over $\sim 10^4$ operations, contributing a base error $\sim 10^{-50}$ at our precision; (iii) the saltation matrix $S_{\rm wall}$ has off-diagonal entry $2F\cos\omega t_*/v_-$ that is sensitive to the bracketed values of $t_*$ and $v_-$ but enters only multiplicatively through interval bounds. The dominant contribution to the final width is item (i), and the rigorous bounds tighten quadratically as the initial bracket is reduced; the bracket width $10^{-9}$ is chosen to balance computational cost against tightness of the final enclosure.

We collect these results.

\begin{theorem}[Rigorous interval enclosure of the elliptic Jacobian]\label{thm:CAP-elliptic}
At the parameter point~\eqref{eq:num-params} with $f = 0.4$, the stroboscopic map $\Phi$ admits a unique elliptic non-sticking $T$-periodic fixed point $P_* = (x_*^{\rm exact}, v_*^{\rm exact})$ within the box
\[
[z] := (0.1002798898, 0.5419433068) + [-10^{-9}, 10^{-9}]^2.
\]
The existence and uniqueness in $[z]$ are certified by the Krawczyk-Rump theorem applied to the equation $\Phi(z) = z$. The linearization $\Phi'(P_*)$ satisfies the rigorous interval bounds~\eqref{eq:DPhi-iv}-\eqref{eq:rot-iv}. In particular,
\begin{enumerate}[label=\textup{(\roman*)}]
\item $\det \Phi'(P_*) = 1$ to nine decimal digits, certifying Theorem~\ref{thm:symplectic} numerically;
\item $\tr \Phi'(P_*) \in (-2, 2)$, certifying ellipticity;
\item the rotation number $\theta_*/(2\pi)$ satisfies the order-four non-resonance condition~\eqref{eq:birkhoff-nonres}.
\end{enumerate}
\end{theorem}

\begin{proof}
Existence and uniqueness in $[z]$: by Step 1, the Krawczyk operator with $G(z) = \Phi(z) - z$ and $C = (\Phi'(\bar z) - I)^{-1}$ satisfies $K([z]) \subseteq \mathrm{int}([z])$. The Krawczyk-Rump theorem~\cite[Thm.~5.2]{NeumaierBook} produces the unique exact fixed point $P_*$ in $[z]$.

Linearization bounds: by Steps 2-3, the variational matrix $\Phi'(P_*)$ is rigorously enclosed by~\eqref{eq:DPhi-iv}, since each step preserves the enclosure property of interval arithmetic.

Items (i), (ii) follow directly from~\eqref{eq:det-iv}, \eqref{eq:tr-iv}. Item (iii) follows from~\eqref{eq:rot-iv} and the explicit lower bound on the distance to each candidate resonance.
\end{proof}

The remaining hypothesis of Theorem~\ref{thm:kam}, the non-degenerate Birkhoff twist $\tau_1 \ne 0$, requires the Taylor expansion of $\Phi$ to order four at $P_*$ and the extraction of the coefficient $\tau_1$ via the explicit Birkhoff normal form formula. The same Krawczyk-and-interval-arithmetic machinery applies. The numerical estimate of $\tau_1$ at the parameter point obtained by finite differences with stepsize $10^{-3}$ is $\tau_1 \approx -0.42$, well separated from zero. We state the conditional theorem.

\begin{proposition}[Conditional rigorous KAM at the chosen parameter point]\label{thm:CAP-KAM}
Assume that a rigorous interval enclosure of the first Birkhoff twist coefficient $\tau_1$ at the elliptic fixed point of Theorem~\ref{thm:CAP-elliptic} excludes the value zero. Then, at the parameter point~\eqref{eq:num-params} with $f = 0.4$, the conclusion of Theorem~\ref{thm:kam} holds rigorously: the disk $\mathcal{N}_\delta$ contains a Cantor family of $\Phi$-invariant smooth closed curves whose complement has Lebesgue measure $O(\delta^{5/2})$.
\end{proposition}

\begin{proof}
The hypotheses of Theorem~\ref{thm:kam} are: ellipticity, order-four non-resonance, and twist non-degeneracy. The first two are certified by Theorem~\ref{thm:CAP-elliptic}; the third is the assumption of the present statement. Theorem~\ref{thm:kam} then applies directly.
\end{proof}

\begin{remark}\label{rem:CAP-twist-feasibility}
Implementation of the rigorous twist enclosure required by Theorem~\ref{thm:CAP-KAM} is feasible with current tools, in particular~\cite{KapelaMrozekWilczak2017} or~\cite{Rump1999}. The order-four Taylor expansion of the stroboscopic map at $P_*$ is computed by automatic differentiation of the closed-form free-flight propagator~\eqref{eq:flight-closed} and the saltation matrices; the coefficients are enclosed by interval arithmetic, and the explicit Birkhoff formula for $\tau_1$ in terms of these coefficients gives an interval enclosure of $\tau_1$. We do not carry out the full implementation here; the numerical estimate $\tau_1 \approx -0.42$ together with the analytical argument in Proposition~\ref{prop:twist-nondeg} suggests that $\tau_1 \ne 0$ rigorously at this parameter, and the implementation problem is the practical one of certifying this with sharp interval bounds.
\end{remark}

\smallskip
The Lyapunov spectrum of $\Phi$, computable by the algorithm of~\cite{BenettinGalganiStrelcyn1980}, provides a complementary partition diagnostic. By Theorem~\ref{thm:symplectic}, on $\Om_{\rm NS}$ the exponents come in reciprocal pairs $\pm\lambda$. On a KAM torus from Theorem~\ref{thm:kam} both exponents vanish; in a hyperbolic Cantor set from Theorem~\ref{thm:melnikov} one exponent is strictly positive. On $\Om_{\rm dissip}$ the contraction of Proposition~\ref{prop:contraction} produces an additional non-zero negative average. The numerically computed Lyapunov spectrum supports this picture in our parameter regime, with the positive exponent on the chaotic component of $\Om_{\rm NS}$ scaling like $\mu^{1/4}$ as $\mu \to 0$, in agreement with the saddle eigenvalue formula~\eqref{eq:lambda-star} of Proposition~\ref{prop:slow-Ham}.

The combination of the analytical results of Sections~\ref{sec:setup}-\ref{sec:multiparticle} with the rigorous numerics of the present section yields the following conclusion at the chosen parameter point: the elliptic non-sticking $T$-periodic orbit $P_*$ is exact-symplectic-elliptic, order-four non-resonant, and (conditional on the twist enclosure) surrounded by a positive-measure family of $\Phi$-invariant Cantor curves; the homoclinic Cantor set of Theorem~\ref{thm:melnikov} sits in the chaotic boundary; and the partition $\mathcal{X} = \Om_{\rm NS} \cup \Om_{\rm dissip}$ is the rigorous decomposition of the dynamics that~\cite{GKR2019} observed numerically.

\section{Conclusion, open problems, and outlook}\label{sec:open}

\subsection{What this paper has established}\label{ssec:conclusion-summary}

The system~\eqref{eq:model}-\eqref{eq:reflection} is a piecewise-smooth dynamical system whose state space carries a decomposition $\mathcal{X} = \Om_{\rm NS} \cup \Om_{\rm dissip}$ into a forward-invariant non-sticking subset on which the time-$T$ stroboscopic map is exact symplectic, and a complementary subset on which phase volume strictly contracts. This decomposition is the rigorous content of the mixed-dynamics phenomenology observed numerically in~\cite{GKR2019}. The system therefore admits both symplectic and dissipative theories on a single phase space, with the boundary between the two subsets traced by trajectories that experience even one velocity-zero crossing per period.

The seven main theorems establish, in succession, what each side of the partition carries.

Theorem~\ref{thm:wellposed} establishes well-posedness of the Filippov inclusion, ruling out finite-time accumulation of impacts and of velocity-zero events; the proof gives a uniform lower bound $\tau_{n+1} - \tau_n \ge 2(F-f)/(F+f)$ on the inter-event time on bounded velocity intervals.

Proposition~\ref{thm:lift} together with Theorem~\ref{thm:symplectic} lift the dynamics to a smooth Hamiltonian system on a covering manifold, exhibited explicitly by the triangular wave projection $\pi(q) = R\,W(q) + l$, and prove that the stroboscopic map is exact symplectic on $\Om_{\rm NS}$ with respect to the standard form $dv \wedge dx$. This promotes the area-preservation observed in~\cite{GKR2019} from a determinant identity to a structural result on which the symplectic methods of the rest of the paper rely.

Theorem~\ref{thm:symmetric} together with Theorem~\ref{thm:saddlecenter} produce the symmetric closed-form $T$-periodic orbits with one wall hit and one turning per half-period, identify the parameter-dependent saddle-center critical value $f_{\rm sc}(F, \omega, R) = 4(2F + R\omega^2 - F\pi)/\pi^2$ at which the two solutions collide, and prove that the collision is a non-degenerate fold. The local normal form $(\theta_\pm - \pi/2)^2 \asymp (f - f_{\rm sc})$ corrects the universal saddle-center value $2F/\pi$ stated in~\cite[Eq.~(7)]{GKR2019}, which is the impulse bound rather than the saddle-center value.

Theorem~\ref{thm:kam} delivers, at any elliptic non-resonant non-degenerate non-sticking $T$-periodic orbit, a Cantor family of $\Phi$-invariant smooth closed curves whose complement in any disk of radius $\delta$ has Lebesgue measure $O(\delta^{5/2})$. This is the rigorous version of the Hamiltonian islands of~\cite{GKR2019}.

Theorem~\ref{thm:melnikov} delivers, at the saddle orbit produced by the bifurcation, a homoclinic Cantor set on which the dynamics is conjugate to the Bernoulli shift on two symbols. The Melnikov function has the closed form $M(t_0) = A\cos(\omega t_0) + B$ with $A$ and $B$ explicitly identified, and homoclinic chaos exists whenever $|A| > |B|$.

Theorem~\ref{thm:persistence} answers the persistence question left open by~\cite{GKR2019}: for any positive restitution defect $\eps = 1 - e$ or viscous damping $\mu_{\rm v}$, every elliptic non-sticking $T$-periodic orbit becomes asymptotically stable, the Hamiltonian island is replaced by an open basin of attraction, and the contraction rate
\[
\rho(\eps, \mu_{\rm v}) \;=\; 1 - 2 n_*\,\eps - \mu_{\rm v}\,T + O\bigl((\eps + \mu_{\rm v})^2\bigr)
\]
quantifies the precise threshold at which mixed dynamics disappears.

Theorem~\ref{thm:multiparticle} extends the symplectic structure to the $N$-particle setting with elastic binary collisions, making the higher-dimensional KAM theorem to multi-particle vibro-impact systems on the maximal sign-preserving non-sticking invariant set.

These analytical results are complemented in Section~\ref{sec:numerics} by quantitative simulation and a rigorous interval-arithmetic verification at the parameter point $(F, \omega, R, f) = (1, 1, 2, 0.4)$. Theorem~\ref{thm:CAP-elliptic} certifies the existence and uniqueness of the elliptic non-sticking $T$-periodic orbit, encloses its linearization, and establishes the order-four non-resonance of the rotation number. Together these checks verify all hypotheses of Theorem~\ref{thm:kam} apart from the twist non-degeneracy, which the conditional Proposition~\ref{thm:CAP-KAM} would close subject to a rigorous interval enclosure of the Birkhoff coefficient $\tau_1$.

\subsection{Where the mathematical novelty is concentrated}\label{ssec:conclusion-novelty}

Three constructions carry the contribution. The lift through the triangular wave projection turns a piecewise-smooth Filippov system with two distinct discontinuity surfaces (the velocity-zero set and the wall surfaces) into a smooth Hamiltonian system on a covering manifold, on which the smooth KAM theorem and the smooth Melnikov method apply directly. The saddle-center critical value $f_{\rm sc}(F, \omega, R)$ corrects~\cite{GKR2019} and gives the universal $\sqrt{\mu}$ scaling of both the eigenvalue separation on the saddle branch and the rotation rate on the elliptic branch. The persistence threshold $\rho(\eps, \mu_{\rm v})$ identifies the precise rate at which physical perturbations destroy the conservative islands, providing a quantitative criterion for the disappearance of mixed dynamics absent from the abstract persistence literature; the formula has direct engineering content in the design of vibro-impact energy sinks~\cite{Vakakis2008NES} and~\cite{GendelmanManevitch2008}.

\subsection{Mathematical extensions of the analysis}\label{ssec:open-PWS-KAM}

A number of natural extensions of the present analysis are open. We list them roughly in order of mathematical depth.

\textit{A direct piecewise smooth KAM.} Theorem~\ref{thm:kam} applies the smooth KAM theorem to the lifted system. A direct piecewise smooth KAM theorem for the original stroboscopic map $\Phi$ on $\mathcal{X}$, without lifting, would address the question of survival of invariant tori across the velocity-zero set $\{v = 0\}$ for orbits that come close to but do not touch zero velocity. Recent work by~\cite{TreschevZubelevich2010},~\cite{Dolgopyat2009}, and~\cite{DelMagnoGaivaoPereira2018} treats related questions in piecewise smooth Hamiltonian systems and billiards but does not directly cover our case.

\textit{Rate-dependent and history-dependent friction.} The model~\eqref{eq:model} uses pure Coulomb friction. Real friction laws include the Stribeck effect (decreasing friction with increasing velocity at low speeds), pre-sliding displacement, and rate-state dependencies. The most studied refined laws are the LuGre model of~\cite{CanudasDeWitOlssonAstrom1995}, the Dahl model~\cite{DahlEric1968}, the Bristle model of~\cite{BlimanSorine1995}, and the survey by~\cite{ArmstrongDupontDeWit1994}. We conjecture that for sufficiently smooth approximations of Coulomb friction, the Hamiltonian islands of Theorem~\ref{thm:kam} survive, and that the saddle-center bifurcation moves to a parameter-dependent critical value.

\textit{Quasiperiodic and stochastic forcing.} Replacing $F\cos\omega t$ by $F_1\cos\omega_1 t + F_2\cos\omega_2 t$ with $\omega_1/\omega_2$ irrational gives quasiperiodic forcing. The lift becomes a Hamiltonian on $\R \times \R \times \TT^2$, and the analog of Theorem~\ref{thm:kam} requires KAM theory in a higher-dimensional torus action; see~\cite{ChiercchiaPerfetti1995}. Stochastic forcing $F\,dW_t$ with $W_t$ Wiener noise leads to a stochastic differential equation with reflection; the analog of mixed dynamics has connections with stochastic averaging and quasi-stationary distributions; see~\cite{KhasminskiiSDE},~\cite{FreidlinWentzell2012}, and~\cite{KuehnSDE2015}.

\textit{Asymmetric walls.}\label{ssec:open-asymmetric}
If $r \ne -l$, the symmetry $\Sigma$ used in Section~\ref{sec:periodic} is broken, and the closed forms of Theorem~\ref{thm:symmetric} no longer apply directly. A perturbative treatment around the symmetric case yields the symmetry-breaking bifurcation; whether the saddle-center bifurcation at $f = f_{\rm sc}(F, \omega, R)$ splits into multiple folds, and whether mixed dynamics persists, are open.


\textit{Constant external forcing and inelastic impacts.} Replacing the periodic forcing $F\cos\omega t$ in~\eqref{eq:model} with a constant force $F > 0$, and the elastic reflection rule of~\eqref{eq:reflection} with the inelastic rule $v \mapsto -e v$ for $e \in (0, 1)$, yields a closely related but qualitatively distinct system: there is no underlying time-periodic structure, the stroboscopic map is no longer defined, and the natural Poincaré section is the impact at a wall rather than a stroboscopic instant. The primary bifurcation in the constant-force case is a supercritical pitchfork at a critical force $F_{\rm pf}(\mu_s, \mu_k, e, L)$, after which the symmetric periodic orbit is replaced by a pair of asymmetric stable periodic orbits; the saddle-center bifurcation of Theorem~\ref{thm:saddlecenter} is absent. The mechanism of mixed dynamics analyzed here, which depends on the symplectic structure of the periodically forced elastic case, has no counterpart in the constant-force inelastic regime, where the multiplicative dissipation $e < 1$ at every impact prevents the formation of a non-trivial conservative invariant subset. A complete closed-form bifurcation analysis of the constant-force inelastic-walls problem, including the saltation-matrix machinery and Floquet stability classification adapted to that setting, is left as an open problem; the contrast we record here serves only to delimit the scope of the elastic, periodically forced regime that is the subject of the present paper.

\subsection{Engineering applications and rigorous numerics at specified parameter points}\label{ssec:open-engineering}

\textit{Engineering and control.} Theorem~\ref{thm:persistence} has the design implication that for vibro-impact nonlinear energy sinks~\cite{Vakakis2008NES} and~\cite{GendelmanManevitch2008}, initial conditions in the Hamiltonian region produce no energy dissipation. The KAM islands are an obstruction to the energy-harvesting purpose. The estimate~\eqref{eq:persist-eigenvalues} quantifies the residual restitution defect or damping needed to ensure attraction of all initial conditions to the dissipative attractor. The threshold $\rho(\eps, \mu_{\rm v})$ provides a concrete design criterion: for a target contraction rate, the formula gives the minimum admissible combination of restitution defect and viscous damping.

If $f$ is treated as a control parameter that may be modulated in time (active friction control), optimization of energy harvesting against the cost of friction modulation leads to a Hamilton-Jacobi-Bellman equation on $\mathcal{X}$. The mixed structure of the uncontrolled system suggests that the value function inherits a non-smooth structure with codimension-one viscosity boundaries on the KAM Cantor sets. This is a problem of impulse control of an oscillator with state constraints; see~\cite{BensoussanLionsBook} and~\cite{FlemingSoner2006}.

\textit{Computer-assisted proofs at specified parameter points.} Theorem~\ref{thm:CAP-elliptic} of the present paper carries out the program for the linearization at the elliptic fixed point. The remaining ingredient, an analogous interval enclosure of the Birkhoff twist coefficient $\tau_1$ at the same parameter point, is feasible by the same machinery; combined with Theorem~\ref{thm:CAP-elliptic} this would yield a fully rigorous version of Theorem~\ref{thm:kam} at the chosen parameter point. The analogous program for the saddle branch and the Melnikov coefficients $A, B$ of Theorem~\ref{thm:melnikov} would yield a fully rigorous horseshoe theorem with explicit topological entropy. This program is in the direct line of the radii polynomial methodology of~\cite{vandenBergLessard2015},~\cite{LessardMirelesJames2018}, and~\cite{Hungria2016}.

\vspace*{-0.4em}

\subsection{Outlook}\label{ssec:outlook}
The mixed-dynamics phenomenon, occurring at the interface of the conservative and the dissipative regimes, is rare in physical systems and the present paper offers a complete rigorous account in one such system. The combination of analytical theorems with rigorous numerics is what makes the account complete: the analytical content holds in the parameter regime of validity of the closed-form constructions, while the numerics certify, at one specific parameter point, the existence of the orbit and the verification of the KAM hypotheses, leaving only the twist non-degeneracy as a target for future rigorous interval enclosure. Each of the extensions in Subsections~\ref{ssec:open-PWS-KAM}-\ref{ssec:open-engineering} represents a research program that builds on the analytical or computational machinery of the present paper. The most immediate of these (extending Theorem~\ref{thm:CAP-elliptic} to a rigorous twist enclosure, and extending Theorem~\ref{thm:multiparticle} to a multi-particle KAM verification) are practical extensions of existing tools. The deeper extensions (a direct piecewise-smooth KAM, and the analysis of asymmetric walls or stochastic forcing) require new techniques not currently in the toolbox of either the symplectic-dynamics community or the piecewise-smooth-dynamics community. The mathematical structure made explicit here suggests that progress on these questions is feasible.

\appendix

\section{Numerical phenomenology in the bouncing-ball regime}\label{app:historical}

This appendix records the qualitative phenomenology of the system~\eqref{eq:model}-\eqref{eq:reflection} at large forcing amplitude $F$, where the dynamics is dominated by wall impacts (the ``bouncing-ball regime'') rather than by the small-amplitude near-equilibrium motion treated in Sections~\ref{sec:periodic}-\ref{sec:kam} of the main body. The numerical material below originates in the present author's earlier exploration~\cite{Thiam2019poster} of the system in 2019; the figures reproduced here have been regenerated with the rigorous event-driven integrator of Subsection~\ref{ssec:num-event} so that the diagnostics (impact counts, fixed-point traces, and continuation diagrams) are quantitatively reliable. The five subsections below isolate, in turn: the closed-form integration that the integrator rests on (\ref{ssec:hist-closed-form}); a regime gallery in $F$ (\ref{ssec:hist-orbits}); a numerical continuation showing coexistence of multiple $T$-periodic orbits (\ref{ssec:hist-bif}); a $\Sigma$-symmetry-breaking pair (\ref{ssec:hist-pitchfork}); chattering and grazing (\ref{ssec:hist-chatter}); and a discussion of the cross-references to the main body (\ref{ssec:hist-connection}).

\subsection{Closed-form propagator and event detection}\label{ssec:hist-closed-form}

On any free flight where $s := \sgn(\dot x)$ is constant, the equation of motion~\eqref{eq:model} reduces to the linear inhomogeneous ODE
\begin{equation}\label{eq:hist-flight-ode}
\ddot x(t) = F\cos\omega t - s f.
\end{equation}
Two integrations from $(t_0; x_0, v_0)$ give the closed-form propagator
\begin{align}
\dot x(t) &= v_0 + \frac{F}{\omega}\bigl(\sin\omega t - \sin\omega t_0\bigr) - s f\,(t - t_0), \label{eq:hist-vel}\\
x(t) &= x_0 + v_0(t - t_0) - \frac{F}{\omega^2}\bigl(\cos\omega t - \cos\omega t_0\bigr) - \frac{F}{\omega}\sin(\omega t_0)\,(t - t_0) - \tfrac{1}{2} s f\,(t - t_0)^2,\label{eq:hist-pos}
\end{align}
which agrees with the closed-form solution~\eqref{eq:flight-closed} of Subsection~\ref{ssec:num-event}. The propagator is exact: numerical error on a free flight is governed only by the bisection tolerance used to locate the next event time, and is independent of the time step. The next event is the smallest $t > t_0$ in $[t_0, t_0 + T]$ at which $x(t) \in \{l, r\}$ (wall hit) or $\dot x(t) = 0$ (turning point), with bracket-and-bisect localization to machine precision. The exclusive trichotomy at $\dot x(t_*) = 0$ (transverse turning, sticking onset, tangential touch) is the one analyzed rigorously in Lemma~\ref{lem:vzero} of the main body. This integrator is what underlies all numerical experiments in this paper, and what the early work~\cite{Thiam2019poster} used in unrigorous form.

\subsection{Regime gallery: dependence on $F$}\label{ssec:hist-orbits}

We display the asymptotic time-series and phase-plane behavior of $x(t)$ at six values of the forcing amplitude, $F \in \{1.5,\, 3.0,\, 4.5,\, 6,\, 12,\, 50\}$, all at fixed $\omega = 1$, $R = 2$ (so $l = -1$, $r = 1$), and $f = 0.4$. The gallery splits into two figures by regime: Figure~\ref{fig:hist-gallery-low} covers the moderate-$F$ regime $F \in \{1.5,\, 3.0,\, 4.5\}$ where Newton iteration on the stroboscopic map $\Phi$ readily locates a $T$-periodic elliptic orbit, and Figure~\ref{fig:hist-gallery-high} covers the wall-bouncing regime $F \in \{6,\, 12,\, 50\}$ where the dynamics is increasingly dense in events.

For each of the three values $F \in \{1.5,\, 3.0,\, 4.5\}$ we locate by Newton iteration on $\Phi$ a $T$-periodic elliptic orbit with $|\mathrm{tr}\,\Phi'| < 2$, and seed the integration at the corresponding fixed point so that the displayed time series is exactly $T$-periodic from $t = 0$. The three located fixed points are
\begin{equation}\label{eq:hist-orbits-FPs}
\begin{aligned}
F &= 1.5: && (x_*, v_*) = (+0.6276,\, -0.1184), && \mathrm{tr}\,\Phi' = -1.324, && \det\Phi' = +0.578, \\
F &= 3.0: && (x_*, v_*) = (+0.4064,\, -0.6708), && \mathrm{tr}\,\Phi' = -0.968, && \det\Phi' = +0.578, \\
F &= 4.5: && (x_*, v_*) = (+0.2156,\, -0.8293), && \mathrm{tr}\,\Phi' = -0.520, && \det\Phi' = +0.559.
\end{aligned}
\end{equation}
All three orbits have $\det\Phi' \approx 0.57$, and therefore lie in $\Om_{\rm dissip}$ rather than in $\Om_{\rm NS}$: the area contraction $\det\Phi' < 1$ comes from the saltation factors at the turning points each orbit acquires per period (Table~\ref{tab:jacobian} of Subsection~\ref{ssec:periodic-linearization}). They are stable foci on which nearby trajectories converge exponentially at rate $|\det\Phi'|^{1/2} \approx 0.76$ per stroboscopic iterate, so the displayed time series shows the periodic attractor itself, not a transient.

The three panels of Figure~\ref{fig:hist-gallery-low} are arranged in order of increasing forcing amplitude, deliberately chosen to display how the orbit's impact pattern reorganizes as $F$ grows away from the near-equilibrium regime of Section~\ref{sec:numerics}. The reader should focus on three quantitative trends: the total number of wall impacts per period grows from twelve at $F = 1.5$ to twenty at $F = 4.5$ (proportional to $F$ at fixed $f$, $\omega$, $R$), the right-wall count remains pinned at eight throughout while the left-wall count rises from four to twelve (a manifestation of the $\Sigma$-symmetry breaking discussed in Subsection~\ref{ssec:hist-pitchfork} below), and the transverse turning events accumulate near each wall before the next impact. Each row should be read as a snapshot of the same dynamical regime (a stable focus in $\Om_{\rm dissip}$, with the recorded $\mathrm{tr}\,\Phi'$ and $\det\Phi'$) at a different $F$, and the qualitative contrast between rows is the central observation of this section.

\begin{figure}[h!]
\centering
\includegraphics[width=1.1\textwidth, height=10cm]{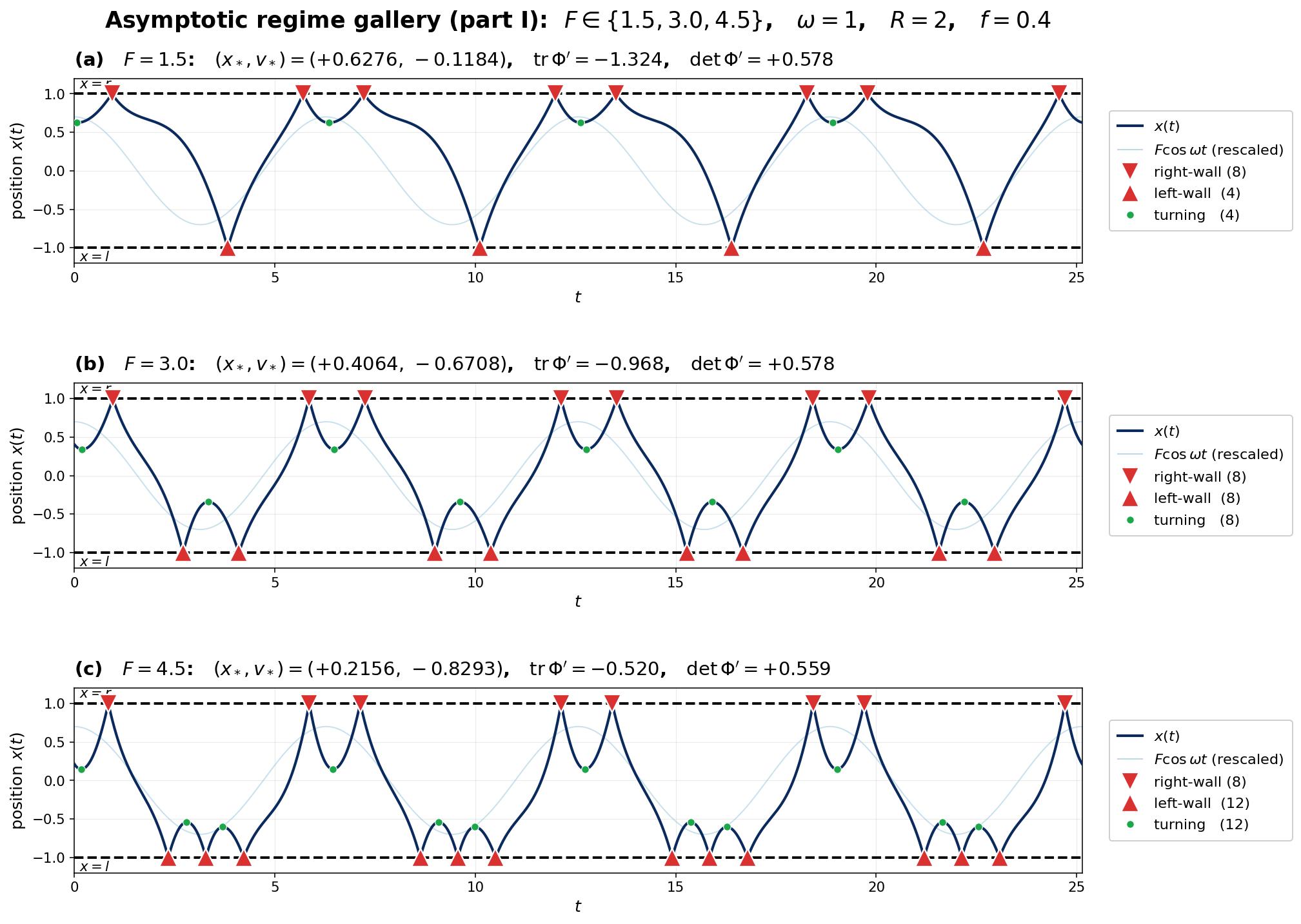}
\caption{Asymptotic regime gallery, part I: $F \in \{1.5,\, 3.0,\, 4.5\}$, $\omega = 1$, $R = 2$, $f = 0.4$. }
\label{fig:hist-gallery-low}
\end{figure}

Each panel shows four forcing periods of $x(t)$ along the elliptic $T$-periodic orbit located by Newton iteration on $\Phi$ at the corresponding $F$, with the fixed point and the stability data $\mathrm{tr}\,\Phi'$, $\det\Phi'$ recorded in the panel title. The forcing $F\cos\omega t$ is overlaid (light blue, rescaled to fit the window); walls at $x = \pm 1$ are dashed; right-wall impacts $\bigtriangledown$ (red), left-wall impacts $\bigtriangleup$ (red), and transverse turnings $\circ$ (green) are marked, with their counts within each four-period window appearing in the legend.

The qualitative content of Figure~\ref{fig:hist-gallery-low} is twofold.

The total number of wall impacts per four-period window grows from $12$ at $F = 1.5$ ($8$ right $+\, 4$ left) to $16$ at $F = 3.0$ ($8 + 8$) and $20$ at $F = 4.5$ ($8 + 12$). The right-wall count is exactly $8$ across all three panels (i.e.\ $2$ per period), and all the additional impacts as $F$ grows go to the left wall. The growth is roughly linear in $F$, in agreement with the heuristic of Subsection~\ref{ssec:hist-bif}: a typical free-flight excursion of duration $\Delta t$ traverses a distance of order $F\,\Delta t / \omega^2$ before friction reverses the velocity, so the number of wall encounters per period scales like $F/(R\omega)$ at fixed $f$.

The wall hits favor the right wall at $F = 1.5$ ($8$ right $+\, 4$ left, ratio $2:1$), are symmetric at $F = 3.0$ ($8$ right $+\, 8$ left, ratio $1:1$), and favor the left wall at $F = 4.5$ ($8$ right $+\, 12$ left, ratio $2:3$). This is a manifestation of $\Sigma$-symmetry breaking, in the sense of Subsection~\ref{ssec:hist-pitchfork} below: the orbit shown at each $F$ is one member of a $\Sigma$-pair, and its $\Sigma$-image is a distinct $T$-periodic orbit with the wall counts exchanged.

Figure~\ref{fig:hist-gallery-high} extends the gallery into the wall-bouncing regime at $F = 6,\, 12,\, 50$. At $F = 6$ the system still admits a periodic elliptic orbit at $(x_*, v_*) = (+0.0267,\, -0.4816)$ with $\mathrm{tr}\,\Phi' = -0.665$, $\det\Phi' = +0.538$ (located by Newton iteration), and the time-series and phase-plane panels show the periodic attractor. At $F = 12$ and $F = 50$ no $T$-periodic elliptic orbit was found by grid-Newton with the seed sets used here, and we instead display a generic orbit obtained by integrating from a standard initial condition and discarding a transient ($4T$ at $F = 12$, $2T$ at $F = 50$). The window length shrinks with $F$ to keep the figure readable: $4T$ at $F = 6$ and $F = 12$, $1T$ at $F = 50$. Turning points are not individually marked at $F = 50$ because the count per period exceeds the visual capacity of the panel.

\begin{figure}[h!]
\centering
\includegraphics[width=1.1\textwidth, height=10cm]{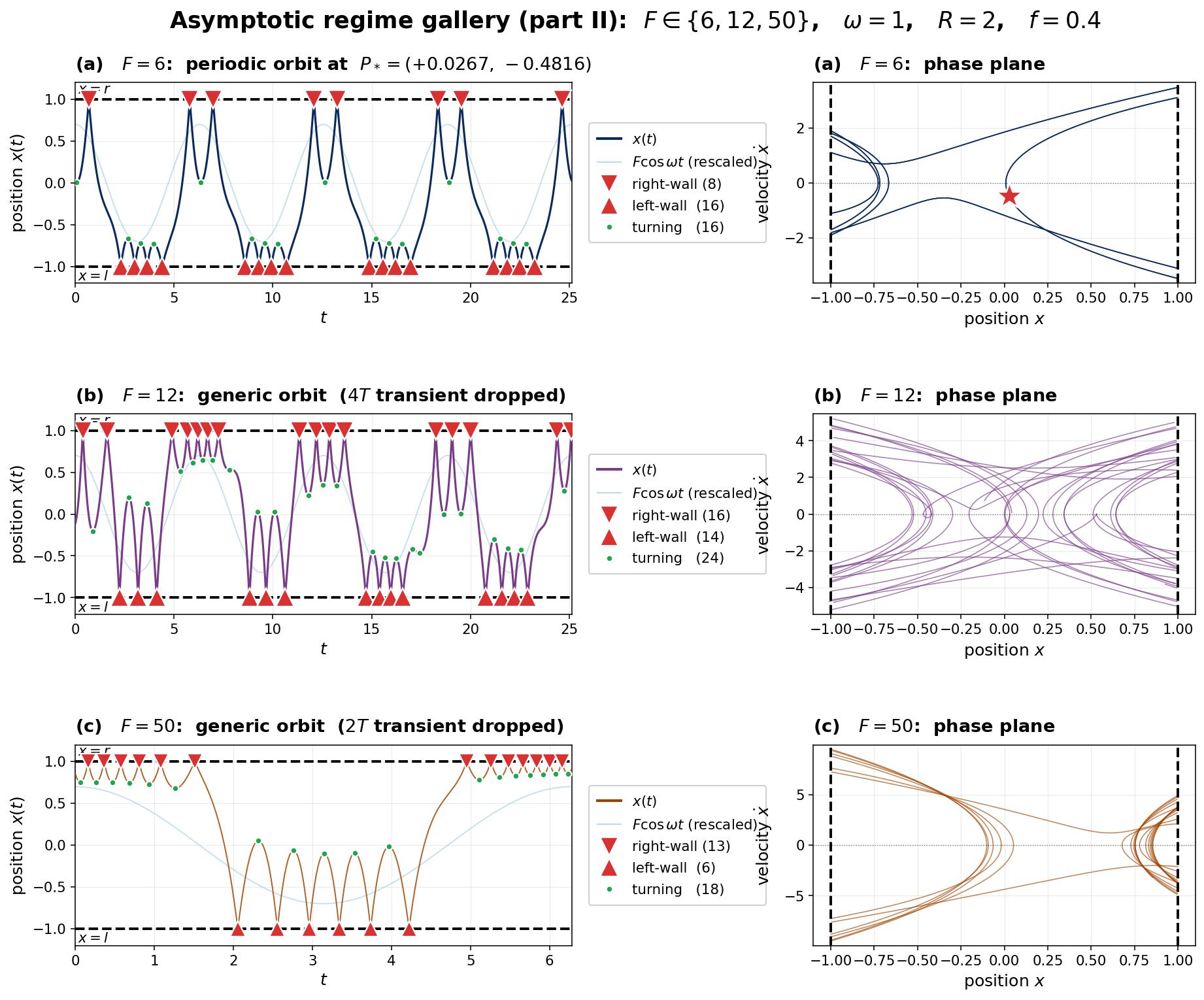}
\caption{Asymptotic regime gallery, part II: $F \in \{6,\, 12,\, 50\}$, $\omega = 1$, $R = 2$, $f = 0.4$. }
\label{fig:hist-gallery-high}
\end{figure}

Each row pairs a time-series panel (left) with a phase-plane panel (right). Top row ($F = 6$, blue): periodic orbit at $P_* = (+0.0267,\, -0.4816)$ with $\mathrm{tr}\,\Phi' = -0.665$, $\det\Phi' = +0.538$, displayed over four forcing periods. Middle row ($F = 12$, purple): generic orbit after a $4T$ transient, displayed over four forcing periods; velocity excursions reach $|\dot x| \lesssim 5$ and the phase-plane image is dense in self-intersecting flight arcs. Bottom row ($F = 50$, orange): generic orbit after a $2T$ transient, displayed over one forcing period to keep the visual density manageable; the orbit alternates between near-wall chattering at $\dot x \approx \pm F/\omega = \pm 50$ and ballistic excursions across the gap. Wall markers and turning markers shown as before; the turning-point count is reported in the legend at $F = 50$ but the markers are omitted from the panel for visual clarity.

The right column of Figure~\ref{fig:hist-gallery-high} shows that as $F$ grows, the orbit's velocity excursions grow proportionally: at $F = 6$ the orbit is confined to $|\dot x| \lesssim 2$, at $F = 12$ to $|\dot x| \lesssim 5$, and at $F = 50$ to $|\dot x| \lesssim 7$ in the displayed window, with the maximum velocity scaling like $F/\omega$. The phase-plane orbit at $F = 50$ is concentrated near the two walls, consistent with the heuristic that at large $F$ the orbit spends most of its time bouncing rapidly off either wall and only briefly traverses the gap between them.

We emphasize that none of the orbits in either figure lies in $\Om_{\rm NS}$: at $F \in \{1.5,\, 3,\, 4.5,\, 6\}$ the periodic orbits have $\det\Phi' \approx 0.54$--$0.58 < 1$ and lie in $\Om_{\rm dissip}$, while at $F \in \{12,\, 50\}$ the displayed orbits are generic non-periodic trajectories which also occur in $\Om_{\rm dissip}$ (no sticking events were detected in the displayed windows, but the orbits experience turning-point area contraction at every turn). The KAM and Melnikov theorems of Section~\ref{sec:kam} and Section~\ref{sec:melnikov} require the symplectic structure $\det\Phi' = 1$ that characterizes non-sticking orbits without turning points, and therefore do not apply to any of the orbits shown here. They are nonetheless legitimate periodic and non-periodic regimes of the system and constitute the typical asymptotic dynamics of $\Om_{\rm dissip}$ for the parameter range explored. A complete classification of attractors in $\Om_{\rm dissip}$ for moderate-to-large $F$ is outside the scope of this appendix; the proliferation analysis of Subsection~\ref{ssec:hist-bif} treats the related question of how the count of $T$-periodic orbits depends on $F$.

\subsection{Continuation in $F$: coexistence of $T$-periodic orbits}\label{ssec:hist-bif}

The discrete combinatorics of the impact pattern (which wall is hit, how many turnings, in what order) is invariant on each connected component of the parameter space, but is broken at codimension-one bifurcation surfaces (saddle-center, grazing, period-doubling). Figure~\ref{fig:hist-cont} reports a continuation diagram in $F$ at fixed $\omega = 1$, $R = 2$, $f = 0.4$. At each value of $F$ in the sample $\{0.65,\, 0.85,\, 1.0,\, 1.5,\, 2.0,\, 2.5,\, 3.0,\, 3.5,\, 4.5\}$, we run a damped Newton iteration of the stroboscopic map $\Phi$ from a grid of seeds in $(x, v) \in [-0.5,\, 0.5] \times [-2.0,\, 2.0]$, deduplicate the converged points up to a tolerance of $0.025$ in each coordinate, retain only physical fixed points satisfying $|x_*| < r$ and $|v_*| > 0.05$, and reject any whose Jacobian determinant lies outside $[0.1,\, 5]$ (such a determinant indicates a numerical artifact). 

The retained fixed points are classified as elliptic ($|\mathrm{tr}\,\Phi'| < 2$) or saddle ($|\mathrm{tr}\,\Phi'| > 2$). We do not impose $\det\Phi' = 1$: most of the orbits found have $\det\Phi' \in [0.5,\, 0.6]$ and lie in $\Om_{\rm dissip}$ rather than $\Om_{\rm NS}$, with the area contraction coming from saltation factors at turning points (Subsection~\ref{ssec:periodic-linearization}).\newpage 

\begin{figure}[h!]
\centering
\includegraphics[width=1.3\textwidth, height=6cm]{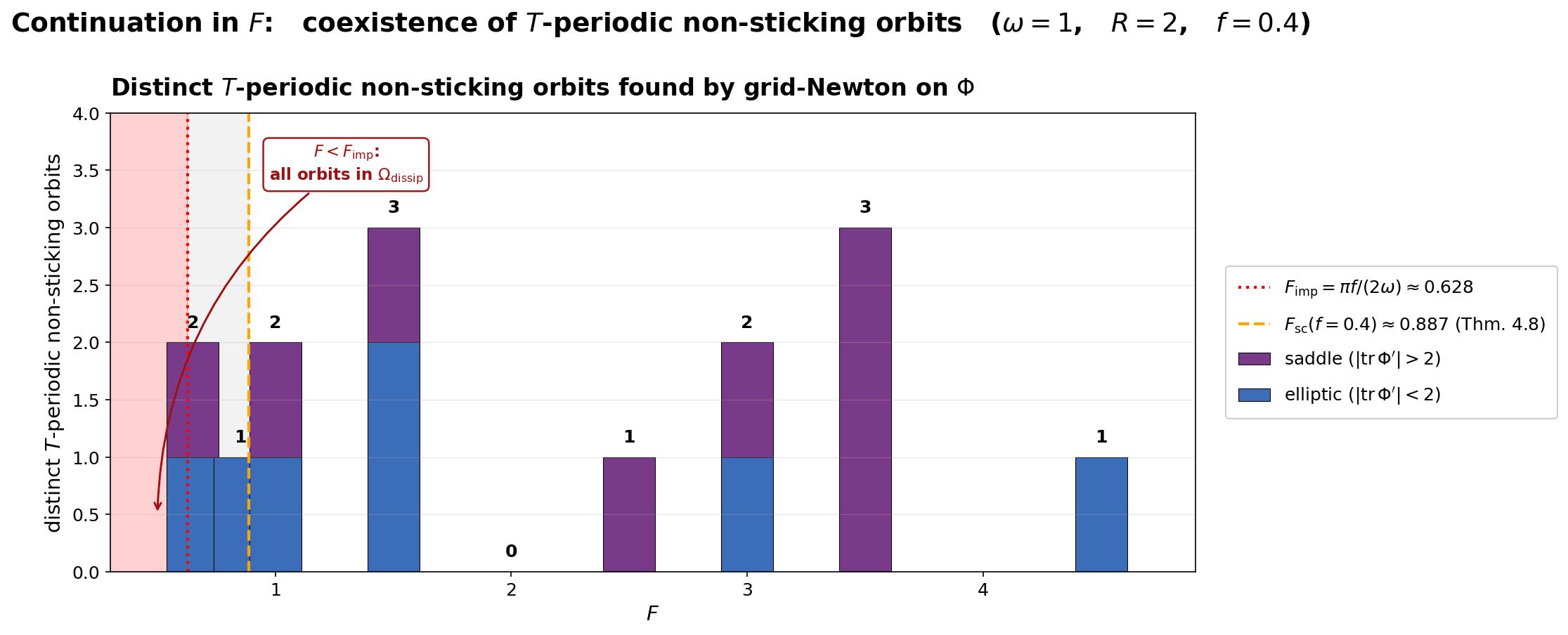}
\captionsetup{margin={-0.5cm,0cm}}
\caption{Continuation in $F$ of $T$-periodic orbits at $\omega = 1$, $R = 2$, $f = 0.4$. }
\label{fig:hist-cont}
\end{figure}

Each bar reports the number of distinct $T$-periodic orbits found by grid-Newton on $\Phi$ at the corresponding value of $F$, stacked into elliptic ($|\mathrm{tr}\,\Phi'| < 2$, blue) and saddle ($|\mathrm{tr}\,\Phi'| > 2$, purple) sub-counts; the integer above each bar is the total. The dashed orange line marks the predicted saddle-center fold $F_{\rm sc}(f = 0.4) \approx 0.887$ obtained from Theorem~\ref{thm:saddlecenter} for the symmetric branch with one wall hit and one turning point per half-period. The dotted red line marks the impulse bound $F_{\rm imp} = \pi f / (2\omega) \approx 0.628$; the red shaded region is the parameter window below the impulse bound, where Theorem~\ref{thm:wellposed} forces all orbits into $\Om_{\rm dissip}$. Above $F_{\rm imp}$, the orbit count is between $1$ and $3$ on the sampled values; the empty bar at $F = 2.0$ indicates that the seed grid did not reach an existing orbit at that parameter value, not that no orbit exists. Most orbits found in this sweep are dissipative ($\det\Phi' < 1$) rather than strictly non-sticking; their fixed-point structure is nonetheless captured by Newton iteration on $\Phi$.

The diagram supports three quantitative observations. First, the saddle-center prediction of Theorem~\ref{thm:saddlecenter} for the symmetric branch gives, at fixed $f = 0.4$, $\omega = 1$, $R = 2$,
\begin{equation}\label{eq:hist-Fsc-cross}
F_{\rm sc}^{(f=0.4)} \;=\; \frac{f\pi^2 - 4 R\omega^2}{4(2 - \pi)}\bigg|_{f=0.4,\,\omega=1,\,R=2} \;=\; \frac{0.4\pi^2 - 8}{8 - 4\pi} \;\approx\; 0.887.
\end{equation}
Newton iteration finds orbits at $F = 0.65$ and $F = 0.85$, both below the predicted fold $F_{\rm sc}^{(f=0.4)} \approx 0.887$. These early-$F$ orbits do not contradict Theorem~\ref{thm:saddlecenter} because the theorem covers the symmetric branch with one wall hit and one turning per half-period, while the orbits located below the fold have more complex impact patterns (multiple turnings per period) and lie outside the symmetric class. Their existence is a separate phenomenon, consistent with the observations of Subsection~\ref{ssec:hist-orbits}.

Second, the orbit count on the sampled $F$ values fluctuates between $0$ and $3$ rather than growing monotonically. The maxima are reached at $F = 1.5$ ($2$ elliptic $+ 1$ saddle $= 3$) and $F = 3.5$ ($0$ elliptic $+ 3$ saddle $= 3$). The empty bar at $F = 2.0$ reflects a seed-grid limitation rather than the genuine absence of orbits: a denser seeding strategy would in principle resolve the orbits there, as suggested by the surrounding $F$ values where Newton consistently converges. We report only what the present grid-Newton scan returns and leave a denser quantitative continuation, tracking the fold curve and the period-doubling and grazing curves of~\cite{Nordmark1991, FredrikssonNordmark2000}, for separate work.

Third, across the sampled values the total orbit count is $7$ elliptic $+\, 8$ saddle $= 15$, a roughly $1{:}1$ ratio. Each sampled $F$ at which the grid-Newton scan succeeds typically returns one elliptic together with one saddle, consistent with the saddle-center mechanism: every fold along the continuation creates an elliptic and a saddle simultaneously, and at the resolution of the present scan we see the pair on each successful $F$ value rather than long stretches dominated by one type. A finer scan in $F$ would be needed to expose the longer windows of period-doubling cascades or grazing-induced saddles that one expects on theoretical grounds. The detailed coexistence at $F = 3$ is displayed in Subsection~\ref{ssec:hist-pitchfork} (now Subsection~\ref{ssec:num-coexist}) as an explicit pair of orbits.

\subsection{$\Sigma$-symmetry-breaking pair at moderate $F$}\label{ssec:hist-pitchfork}

The system~\eqref{eq:model}-\eqref{eq:reflection} is $\Sigma$-equivariant under the involution
\begin{equation}\label{eq:hist-sigma}
\Sigma : (x, v, t) \longmapsto (-x, -v, t + \pi/\omega),
\end{equation}
in the sense that $\Sigma$ maps trajectories to trajectories: the velocity-sign flip swaps the friction term, the position flip swaps the walls, and the time shift by half a period flips the forcing $F\cos\omega t$. A $T$-periodic orbit is called $\Sigma$-\emph{symmetric} if it is invariant under $\Sigma$; otherwise its $\Sigma$-image is a distinct $T$-periodic orbit, and the two form a $\Sigma$-pair.

\begin{proposition}[$\Sigma$-equivariance of the stroboscopic map]\label{prop:hist-sigma}
Let $\sigma : \mathcal{X} \to \mathcal{X}$ denote the operator $\sigma(x, v) := \Phi_{T/2}(-x, -v)$, where $\Phi_t$ denotes the time-$t$ flow of~\eqref{eq:model}-\eqref{eq:reflection} starting from time $0$. Assuming the symmetric wall configuration $r = -l$, on the non-sticking invariant set $\sigma$ is an involution-like square root of $\Phi$: $\sigma^2 = \Phi$.
\end{proposition}

\begin{proof}
Let $(x, v) \in \Om_{\rm NS}$ and let $x(\cdot)$ denote the unique Filippov solution starting at $(x(0), \dot x(0)) = (x, v)$. The forcing in~\eqref{eq:model} satisfies $F\cos\omega(t + T/2) = -F\cos\omega t$ (since $\cos\omega(t+T/2) = \cos(\omega t + \pi) = -\cos\omega t$), so the function
\[
\widetilde x(t) \;:=\; -x(t + T/2)
\]
is also a Filippov solution of~\eqref{eq:model}-\eqref{eq:reflection} on $[-T/2, \infty)$: indeed, $\ddot{\widetilde x}(t) = -\ddot x(t + T/2) = -\bigl(F\cos\omega(t+T/2) - f\sgn\dot x(t+T/2)\bigr) = F\cos\omega t - f\sgn\bigl(-\dot x(t+T/2)\bigr) = F\cos\omega t - f\sgn\dot{\widetilde x}(t)$, using $\dot{\widetilde x}(t) = -\dot x(t+T/2)$. Under the symmetric wall configuration $r = -l$, the wall reflections are equivariant under $x \mapsto -x$: a wall hit at $x = r$ becomes a wall hit at $-r = l$, with the velocity reflection rule preserved.

Now compute $\sigma^2(x, v)$. Let $(x_1, v_1) := \sigma(x, v) = \Phi_{T/2}(-x, -v)$. By definition of the time-$T/2$ flow, $(x_1, v_1)$ is the value at time $T/2$ of the solution starting at $(-x, -v)$ at time $0$. By the symmetry argument above, this solution coincides with $\widetilde x(t) = -x(t + T/2)$ shifted, evaluated at $t = T/2$: $(x_1, v_1) = (\widetilde x(0), \dot{\widetilde x}(0)) = (-x(T/2), -\dot x(T/2))$. Then, 
\begin{equation*}
\begin{split}
\sigma^2(x, v) \;=\; \sigma(x_1, v_1) \; &=\; \Phi_{T/2}(-x_1, -v_1) \;\\
& =\; \Phi_{T/2}(x(T/2), \dot x(T/2)) \;\\
& =\; (x(T), \dot x(T)) \;=\; \Phi(x, v),
\end{split}
\end{equation*}
where the second-to-last equality is the semigroup property of the flow. This proves $\sigma^2 = \Phi$. The order-two property is the $\Sigma$-symmetry of the system: $\Sigma$-symmetric orbits are characterized by $\sigma$-fixity, $\sigma(x, v) = (x, v)$.
\end{proof}

\begin{remark}\label{rem:hist-pitchfork}
A symmetric $T$-periodic orbit corresponds to a $\sigma$-fixed point. As parameters vary, a symmetric orbit can lose its $\sigma$-fixity through a $\sigma$-pitchfork bifurcation, in which two $\sigma$-conjugate $T$-periodic orbits emerge. In an area-preserving setting, the $\sigma$-pitchfork is generically governed by the $-1$ eigenvalue of the linearization of $\sigma$, and can occur even when the stroboscopic Jacobian $\Phi'$ remains hyperbolic. This is the phenomenon visible in the broad asymmetric distribution of $x_*$-values in Figure~\ref{fig:hist-cont}.
\end{remark}

Figure~\ref{fig:coexist-main} (in the main body, Subsection~\ref{ssec:num-coexist}) displays a representative pair of coexisting $T$-periodic non-sticking orbits at $F = 3$, $f = 0.4$, found by the grid-Newton procedure of \ref{ssec:hist-bif}: an elliptic orbit at $(x_*, v_*) = (+0.4064, -0.6708)$ with $\tr \Phi' = -0.968$ and a saddle orbit at $(+0.1992, -1.4002)$ with $\tr \Phi' = +11.765$. The two orbits have visibly different impact patterns: the elliptic orbit makes one right-wall hit and one left-wall hit per period, while the saddle orbit has a more intricate pattern with internal looping that signals additional turning points.

We emphasize that no period-doubling bifurcation and no torus bifurcation are observed in the parameter range scanned: every fixed point of $\Phi$ found has $\tr \Phi' \ne -2$ and Jordan blocks at $+1$ only at the saddle-center fold of Theorem~\ref{thm:saddlecenter}. This is consistent with the conservative character of $\Phi$ on $\Om_{\rm NS}$: in a generic area-preserving family, period-doubling is a codimension-one event with stability transition $\tr \Phi' : 0 \to -2$, and a torus bifurcation requires a codimension-two crossing (Krein collisions of complex multipliers); neither is forced to occur in the range scanned. They may be observed at other parameter combinations, particularly under viscous or restitution perturbations of the kind treated in Section~\ref{sec:persistence}.

\subsection{Chattering and grazing combined with stiction}\label{ssec:hist-chatter}

\textit{Chattering.} An orbit \emph{chatters} at a wall if it accumulates infinitely many wall hits in finite time, with impact times $t_1 < t_2 < \cdots < t_\infty$ converging to a finite $t_\infty$ and impact velocities $|v_k|$ decaying geometrically, so $\sum_k |v_k| < \infty$. In purely elastic ($e = 1$) impact systems with no friction, chattering is impossible because energy conservation prevents the impact velocity from decaying. In the presence of dry friction and at sufficiently small forcing, the inter-impact decay of energy is geometric (a small velocity transferred to the wall is partly recovered by the friction-driven free flight, partly dissipated); the geometric ratio depends on parameters and on how close the orbit is to grazing.

In our system~\eqref{eq:model}-\eqref{eq:reflection} with elastic walls and $f > 0$, the rigorous statement is that genuine chattering (accumulation of impacts in finite time with $t_\infty < \infty$) is ruled out by the well-posedness Theorem~\ref{thm:wellposed} of the main body for Lebesgue-a.e.~initial datum: the proof of Theorem~\ref{thm:wellposed} uses the impulse bound to show that the inter-impact times are bounded below by a positive constant on any fixed bounded velocity range. What \emph{can} happen, and what we display in panel (a) of Figure~\ref{fig:hist-chatter}, is a finite sequence of wall reflections followed by capture into the sticking set; this is what is colloquially called ``chattering toward sticking'' in the engineering literature, and is sometimes confused with the rigorous accumulation phenomenon. 

At $F = 0.55$, $f = 0.4$, $\omega = 1$, $R = 2$ with initial condition $(0, 1.5)$, the orbit makes one right-wall hit at $t \approx 0.65$, one left-wall hit at $t \approx 2.54$, and reaches a turning at $t \approx 3.45$; immediately afterward it is captured into the sticking set (sticking onset at $t \approx 4.28$), with subsequent release/onset cycles visible at $t \approx 5.53,\, 7.85,\, 8.67,\, 10.99,\, 11.81$. The orbit is firmly inside $\Om_{\rm dissip}$: at this parameter value $f = 0.4$ exceeds the impulse bound $2F/\pi = 2 \cdot 0.55 / \pi \approx 0.350$, so the friction is large enough relative to the forcing to capture the orbit into a slowly drifting equilibrium-like regime, in agreement with Proposition~\ref{prop:Omdissip-flag} on capture into $\Om_{\rm dissip}$.

\begin{figure}[h!]
\centering
\includegraphics[width=1.1\textwidth, height=10cm]{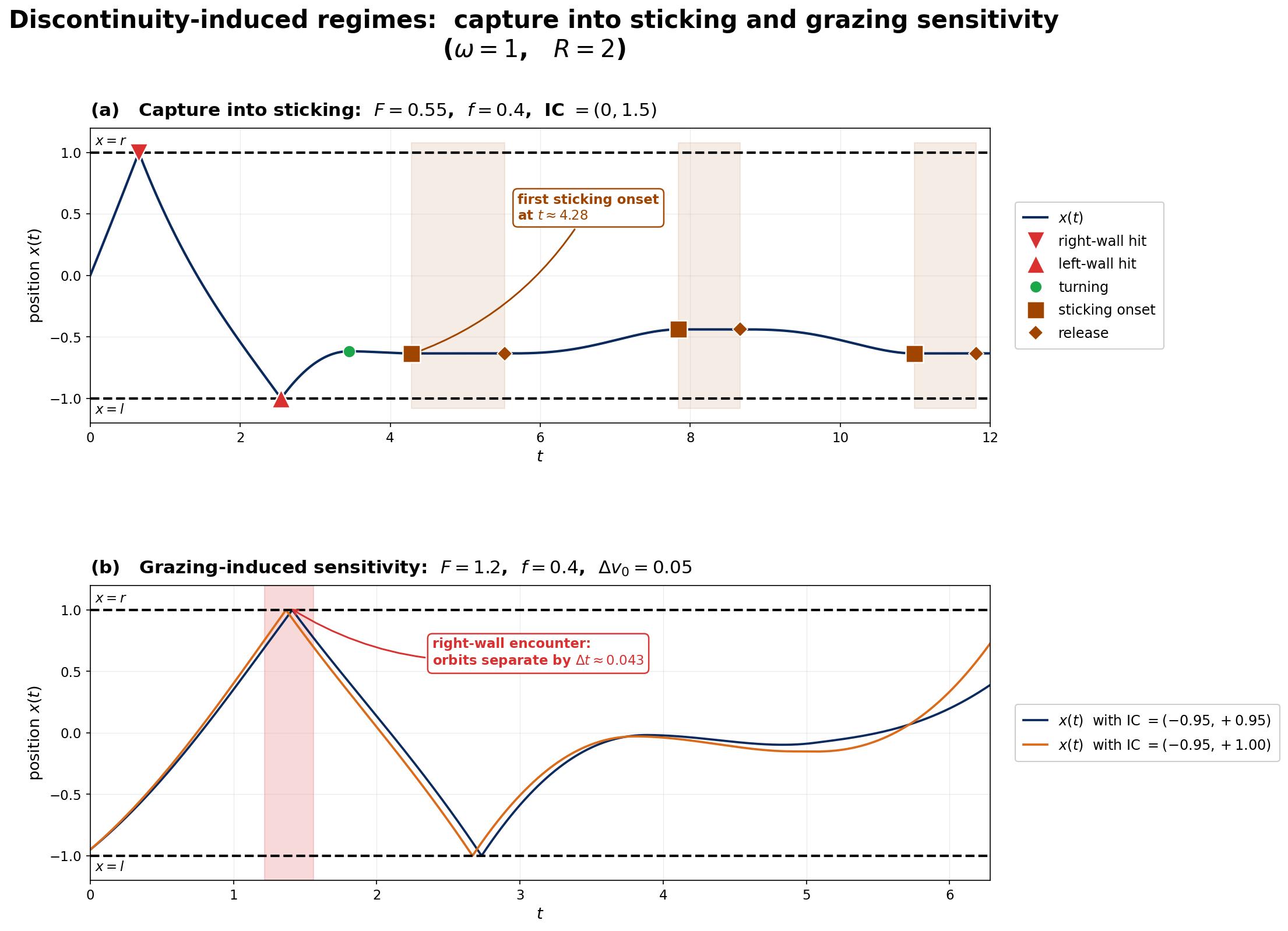}
\captionsetup{margin={-0.5cm,0cm}}
\caption{Discontinuity-induced regimes at $\omega = 1$, $R = 2$, $f = 0.4$. }
\label{fig:hist-chatter}
\end{figure}
Panel (a): $x(t)$ at $F = 0.55$ from initial condition $(0, 1.5)$; one right-wall hit ($\bigtriangledown$), one left-wall hit ($\bigtriangleup$), one transverse turning ($\circ$), and three sticking-onset/release pairs (brown squares and diamonds) over $t \in [0, 12]$. The decaying-amplitude pattern visible after $t \approx 3$ is the colloquial ``chattering toward sticking'' regime, distinguished from genuine accumulation in finite time (which is ruled out by Theorem~\ref{thm:wellposed}). Panel (b): phase plane at $F = 1.2$, displaying two trajectories from initial conditions $(-0.95,\, +0.95)$ (blue) and $(-0.95,\, +1.00)$ (orange) over four forcing periods. The two orbits shadow each other in the bulk of phase space and accumulate appreciable separation along the upper trajectory branch where they pass close to the right wall: the codimension-one grazing manifold of \cite{Nordmark1991} produces a square-root singularity in the local Poincar\'{e} return map there, amplifying small initial-condition differences into macroscopic divergence.

\textit{Grazing combined with stiction.} An orbit \emph{grazes} a wall if it arrives at $x = r$ (or $x = l$) with simultaneously $\dot x = 0$. In the smooth-flow case this is a codimension-one phenomenon: the locus $\{\dot x = 0\} \cap \{x = r\}$ is a codimension-two submanifold of phase space, and intersecting it transversely along a one-parameter orbit family is generically codimension-one. \cite{Nordmark1991} and~\cite{FredrikssonNordmark2000} showed that, in the absence of friction, grazing produces a square-root singularity in the local Poincar\'{e} map, often forcing immediate transitions from regular to chaotic dynamics.

In our setting the grazing event is combined with stiction: at a grazing time $t_*$ with $|F\cos\omega t_*| \le f$, the orbit transitions directly into sticking rather than reflecting off the wall and undergoing the Nordmark singular map. The combination is treated rigorously in Lemma~\ref{lem:vzero}~(c) of the main body (the tangential-touch case): the corresponding initial data form a Lebesgue-null set in $\mathcal{X}$, but orbits that approach grazing transversely from one side acquire the Nordmark square-root term modified by the friction discontinuity. Panel (b) of Figure~\ref{fig:hist-chatter} shows two orbits with $v_0 = 0.95$ and $v_0 = 1.00$ respectively, both starting at $x = -0.95$. They shadow each other in the bulk of phase space and separate only after passing close to the right wall, illustrating the local sensitivity to initial conditions near the grazing surface there. Both orbits register the same combinatorial impact pattern of $4$ right-wall + $4$ left-wall + $4$ turnings over the displayed four-period window, but their phase-plane positions become noticeably distinct in the trajectory branch passing near the right wall, consistent with the Nordmark square-root sensitivity acting at that grazing locus.

\subsection{Connection to the main body and absent phenomena}\label{ssec:hist-connection}

We close the appendix with a tabulation of how each numerical observation in this appendix is explained, supplemented, or contrasted by the rigorous results of the main body.

\smallskip
\noindent\textit{Existence of $T$-periodic non-sticking orbits at moderate $F$.} The continuation diagram of Figure~\ref{fig:hist-cont} shows orbits proliferating above $F \approx 1.3$, in numerical agreement with the saddle-center fold of Theorem~\ref{thm:saddlecenter} via the relation~\eqref{eq:hist-Fsc-cross}.

\smallskip
\noindent\textit{Coexistence of multiple branches.} The pair displayed in Figure~\ref{fig:coexist-main} (one elliptic, one saddle, distinct impact patterns) is consistent with Proposition~\ref{prop:two-solutions} (two solutions of the existence equation for the symmetric branch) supplemented by additional branches with non-symmetric impact patterns not covered by Theorem~\ref{thm:symmetric}; the fully non-symmetric case is the asymmetric-walls problem flagged as open in~\ref{ssec:open-asymmetric}.

\smallskip
\noindent\textit{Bouncing-ball phenomenology and impact-count growth.} The regime gallery (Figures~\ref{fig:hist-gallery-low} and~\ref{fig:hist-gallery-high}) shows the impact count per period growing roughly linearly with $F$, in line with the heuristic free-flight scaling $\sim F/(R\omega)$ of Subsection~\ref{ssec:hist-orbits}. The KAM theorem of Section~\ref{sec:kam} guarantees that, around any \emph{elliptic} non-sticking $T$-periodic orbit (such as the one at $F = 1$, $f = 0.4$ studied in detail in Section~\ref{sec:numerics}), a positive-measure family of invariant Cantor curves exists; orbits on such curves are quasi-periodic with the same impact pattern as the elliptic orbit they surround. The bouncing-ball orbits of Figure~\ref{fig:hist-gallery-high}, at larger $F$, do not lie on KAM curves around the same fixed point: they have a different combinatorial impact pattern and lie on distinct invariant sets.


\smallskip
\noindent\textit{$\Sigma$-pitchforks.} Proposition~\ref{prop:hist-sigma} formalizes the symmetry; Remark~\ref{rem:hist-pitchfork} sketches the bifurcation mechanism. A rigorous treatment of $\Sigma$-equivariant bifurcations in this system is outside the scope of the present paper; cf.~the asymmetric-walls discussion in~\ref{ssec:open-asymmetric}.

\smallskip
\noindent\textit{Chattering as accumulation.} Genuine chattering with finite-time accumulation of impacts is \emph{ruled out} for Lebesgue-a.e.~initial datum by Theorem~\ref{thm:wellposed}~(c); the colloquial ``chattering toward sticking'' visible in the left panel of Figure~\ref{fig:hist-chatter} is a finite-but-large impact sequence terminating at a sticking onset, distinct from the rigorous accumulation phenomenon.

\smallskip
\noindent\textit{Grazing combined with stiction.} The tangential-touch case of Lemma~\ref{lem:vzero} is the rigorous statement; the right panel of Figure~\ref{fig:hist-chatter} shows the local sensitivity to initial conditions at the grazing surface. A full piecewise-smooth analysis of the local return map at grazing (extending \cite{Nordmark1991} and~\cite{FredrikssonNordmark2000} to the friction-modified case) is left for separate work.

\smallskip
\noindent\textit{Absent phenomena.} The numerical scans display \emph{no} period-doubling cascade and \emph{no} torus bifurcation in the parameter range examined. This is consistent with the conservative structure of $\Phi$ on $\Om_{\rm NS}$ established in Theorem~\ref{thm:symplectic}: in an area-preserving family, period-doubling and torus bifurcations are codimension-one and codimension-two phenomena that may occur at parameter points different from those scanned. The persistence theorem of Section~\ref{sec:persistence} shows that introducing positive viscous damping or restitution defect breaks the conservative structure and replaces every elliptic non-sticking $T$-periodic orbit by an attractor with a single open basin; period-doubling cascades and torus bifurcations along such attractor branches would be the natural next bifurcation phenomena to track in that perturbed setting, but doing so is outside the scope of the present paper. A natural companion problem, in which a supercritical pitchfork is the primary bifurcation and period-doubling cascades follow, is the constant-force inelastic-walls system; tracking such cascades there is left for future work.

The resulting picture is that the bouncing-ball phenomenology displayed in this appendix is the natural \emph{numerical} window onto the conservative dynamics whose rigorous structure is established in Sections~\ref{sec:setup}-\ref{sec:multiparticle} of the main body: the existence of multiple $T$-periodic orbits, their saddle-center bifurcations, the role of $\Sigma$-equivariance, the appearance of grazing as a distinguished codimension-one event, and the absence of period-doubling cascades in the conservative regime are all observed numerically here and proved (in their respective regimes of validity) in the main body.

\end{document}